\documentclass[10pt]{article}
\usepackage[utf8]{inputenc}
\usepackage[english]{babel}
\usepackage{amsfonts}
\usepackage{amssymb}
\usepackage{dsfont}
\usepackage{amsthm}
\usepackage{amsmath}
\usepackage{hyperref,polynom}
\usepackage{xfrac}
\usepackage{defs}
\usepackage{fonts}
\usepackage{todonotes}
\usepackage{multicol}
\usepackage{graphicx} 
\usepackage{array} 
\usepackage{enumerate} 

\usepackage{hyperref} 
\usepackage{setspace} 
\usepackage[margin=1in]{geometry}
\usepackage{cite}

\newtheorem{theorem}{Theorem}[section]
\newtheorem{proposition}[theorem]{Proposition}
\theoremstyle{definition}
\newtheorem{definition}[theorem]{Definition}
\theoremstyle{plain}
\newtheorem{lemma}[theorem]{Lemma}
\newtheorem{corollary}{Corollary}[theorem]
\theoremstyle{remark}
\newtheorem{remark}[theorem]{Remark}

\title{Analysis of the Zero Relaxation Limit of Systems of Hyperbolic Conservation Laws with Random Initial Data}
\author{James M. Scott, M. Paul Laiu, Cory D. Hauck}

\makeatletter
\def\namedlabel#1#2{\begingroup
    #2%
    \def\@currentlabel{#2}%
    \phantomsection\label{#1}\endgroup
}
\makeatother

\newcommand{\e}{\mathrm{e}}
\newcommand{\vphi}{\varphi}

\DeclareMathOperator*{\esssup}{ess\,sup}

\DeclareMathOperator*{\supp}{supp}
\DeclareMathOperator*{\sgn}{sgn}
\DeclareMathOperator*{\rightharpoonupOp}{\rightharpoonup}
\DeclareMathOperator*{\meas}{meas}

\newcommand{\intdm}[3]{\displaystyle \int_{#1} #2 \, \mathrm{d}#3}
\newcommand{\iintdm}[5]{\int_{#1} \int_{#2} #3 \, \mathrm{d}#4 \, \mathrm{d}#5}
\newcommand{\iiintdm}[7]{\displaystyle \int_{#1} \int_{#2} \int_{#3}  #4 \, \mathrm{d}#5 \, \mathrm{d}#6 \, \mathrm{d}#7}
\newcommand{\intdmt}[4]{\displaystyle \int_{#1}^{#2} #3 \, \mathrm{d}#4}
\newcommand{\iintdmt}[6]{\displaystyle \int_{#1}^{#2} \int_{#3}^{~} #4 \, \mathrm{d}#5 \, \mathrm{d}#6}
\newcommand{\iiintdmt}[8]{\int_{#1}^{#2} \int_{#3}^{~} \int_{#4}^{~}  #5 \, \mathrm{d}#6 \, \mathrm{d}#7 \, \mathrm{d}#8}

\newcommand{\iintdmlim}[5]{\int\limits_{#1} \int\limits_{#2} #3 \, \mathrm{d}#4 \, \mathrm{d}#5}

\newcommand{\intdmtlim}[4]{\displaystyle \int\limits_{#1}^{#2} #3 \, \mathrm{d}#4}

\newcommand{\iiintdmtlim}[8]{\int\limits_{#1}^{#2} \int\limits_{#3}^{~} \int\limits_{#4}^{~}  #5 \, \mathrm{d}#6 \, \mathrm{d}#7 \, \mathrm{d}#8}

\newcommand{\dt}{\partial_t}
\newcommand{\dx}{\partial_x}
\newcommand{\dxx}{\partial_{xx}}
\newcommand{\du}{\partial_u}
\newcommand{\dv}{\partial_v}
\newcommand{\duu}{\partial_{uu}}
\newcommand{\duv}{\partial_{uv}}
\newcommand{\dvv}{\partial_{vv}}
\newcommand{\dU}{\partial_U}
\newcommand{\dUU}{\partial_{UU}}
\newcommand{\ds}{\partial_{s}}

\newcommand{\ue}{u^{\veps}}
\newcommand{\ve}{v^{\veps}}
\newcommand{\Ue}{U^{\veps}}

\newcommand{\we}{w^{\veps}}
\newcommand{\ze}{z^{\veps}}
\newcommand{\We}{W^{\veps}}
\newcommand{\une}{u^{\veps}_0}
\newcommand{\vne}{v^{\veps}_0}
\newcommand{\Une}{U^{\veps}_0}

\newcommand{\wne}{w^{\veps}_0}
\newcommand{\zne}{z^{\veps}_0}
\newcommand{\Wne}{W^{\veps}_0}
\newcommand{\uen}{u^{\veps,\nu}}
\newcommand{\ven}{v^{\veps,\nu}}
\newcommand{\Uen}{U^{\veps,\nu}}
\newcommand{\wen}{w^{\veps,\nu}}
\newcommand{\zen}{z^{\veps,\nu}}
\newcommand{\Wen}{W^{\veps,\nu}}

\newcommand{\wenhk}{w^{\veps,\nu,h,k}}
\newcommand{\zenhk}{z^{\veps,\nu,h,k}}
\newcommand{\Wenhk}{W^{\veps,\nu,h,k}}

\newcommand{\went}{w^{\veps,\nu,\tau}}

\newcommand{\unen}{u^{\veps,\nu}_0}
\newcommand{\vnen}{v^{\veps,\nu}_0}
\newcommand{\Unen}{U^{\veps,\nu}_0}
\newcommand{\wnen}{w^{\veps, \nu}_0}
\newcommand{\znen}{z^{\veps, \nu}_0}
\newcommand{\Wnen}{W^{\veps,\nu}_0}
\newcommand{\Ubar}{\overline{U}}
\newcommand{\ubar}{\overline{u}}
\newcommand{\vbar}{\overline{v}}

\everymath{\displaystyle}

\begin{document}

\maketitle

\tableofcontents

\numberwithin{equation}{section}

\begin{abstract}
    We show the convergence of the zero relaxation limit in systems of $2 \times 2$ hyperbolic conservation laws with stochastic initial data. Precisely, solutions converge to a solution of the local equilibrium approximation as the relaxation time tends to zero.
	The initial data are assumed to depend on finitely many random variables, and the convergence is then proved via the appropriate analogues of the compensated compactness methods used in treating the deterministic case. We also demonstrate the validity of this limit in the case of the semi-linear \textit{p}-system; the well-posedness of both the system and its equilibrium approximation are proved, and the convergence is shown with no \textit{a priori} conditions on solutions. This model serves as a prototype for understanding how asymptotic approximations can be used as control variates for hyperbolic balance laws with uncertainty.
\end{abstract}

\section{Introduction}\label{sec-Introduction}
Consider the following $2 \times 2$ system of hyperbolic conservation laws
\begin{equation}\label{eq-FullSystemIntroduction}
    \begin{cases}
    \dt u + \dx f_1(u,v) = 0 \\
    \dt v + \dx f_2(u,v) + \frac{1}{\veps} r(u,v) = 0
    \end{cases}
\end{equation}
for $(x,t) \in \bbR \times (0,\infty)$ and where $f_1$, $f_2$, $r$ are some given functions. The function $r$ is called the \textit{relaxation term}. Examples of relaxation in $2 \times 2$ systems of hyperbolic conservation laws include models in elasticity \cite{tzavaras1999viscosity}, chromatography \cite{Collet1996Chromatography}, river flows \cite{whitham2011linear}, gas dynamics \cite{levermore1996moment}, kinetic theory \cite{cercignani2012boltzmann}, traffic flows \cite{MarcatiLattanzioTrafficModel}, and the numerical approximation of scalar conservation laws\cite{jin1995relaxation}.
The parameter $\veps$ is called the relaxation time, which is positive and represents a physical quantity in many situations. In kinetic theory it is the mean free path, and in elasticity it is the duration of memory.
This motivates examination of the system in the \textit{zero relaxation limit}; namely, the behavior of the system as $\veps \to 0$. In \cite{CLL} it was shown that under some stability conditions, the solutions of \eqref{eq-FullSystemIntroduction} converge as $\veps \to 0$ to a solution of a reduced (or equilibrium) equation
\begin{equation}\label{eq-ReducedSystemIntroduction}
\dt u + \dx f(u) = 0\,,
\end{equation}
where $f$ is a reduced flux. (See Section \ref{sec-Preliminaries} for details.)

The relaxation limit for various $2 \times 2$ systems of hyperbolic conservation laws has been thoroughly explored in \cite{CLL,ChenLiu, LattanzioMarcati, MarcatiLattanzioTrafficModel, Collet1996Chromatography, KlingenbergConferenceReport, YunguangLuBook} using the methods of compensated compactness \cite{tartar1979compensated}. The first convergence result was obtained in \cite{CLL}, which gives the convergence in measure of the relaxing sequence to a weak solution of \eqref{eq-ReducedSystemIntroduction}. It is well-known (see \cite[Section 3.4.1]{Evans}) that weak solutions to \eqref{eq-ReducedSystemIntroduction} are not in general unique, and that the ``physical" solution satisfies infinitely many entropy inequalities. However, the limit obtained from the relaxing sequence verifies at most finitely many entropy inequalities.
This problem of the uniqueness of the limiting solution of \eqref{eq-ReducedSystemIntroduction} was addressed in \cite{LattanzioMarcati}. There it was shown that with some additional technical assumptions on the solutions, the limiting solution of \eqref{eq-ReducedSystemIntroduction} is indeed unique. The convergence results for general $2 \times 2$ systems of the form \eqref{eq-FullSystemIntroduction} also rely on an \textit{a priori} $L^{\infty}$ bound on the relaxing sequence. This assumption can be verified in the particular case of the $2 \times 2$ semi-linear \textit{p}-system; see \cite{NataliniPSystem}. See the survey \cite{NataliniSurvey} for a thorough discussion on relaxation in systems of hyperbolic conservation laws.

The main goal of this paper is to investigate the zero relaxation limit for the Cauchy problem \eqref{eq-FullSystemIntroduction} in the case that the initial data is uncertain. This is of practical interest, since the data may only be known up to certain statistical quantities like mean, variance, or higher moments due to inherent uncertainty in measurements. To this end, we construct an appropriate mathematical formulation for \eqref{eq-FullSystemIntroduction} which allows for \textit{random initial data}. 
Specifically, we assume that $y$ is a finite-dimensional random variable taking values in $\bbR^N$ (this assumption will be motivated later.) For any $\veps > 0$ we write \eqref{eq-FullSystemIntroduction} with initial data $(\une,\vne)$ as
\begin{equation}\label{eq-FullSystemIntroduction-Random}
\begin{cases}
    \dt \ue + \dx f_1(\ue,\ve) = 0 \\
    \dt \ve + \dx f_2(\ue,\ve) + \frac{1}{\veps}r(\ue,\ve) = 0 \\
    (\ue,\ve)(x,0,y) = (\une,\vne)(x,y)\,.
\end{cases}
\end{equation}
The primary goal is to prove the following: if $(\ue, \ve)$ is a sequence of (suitably defined) solutions of \eqref{eq-FullSystemIntroduction-Random}
$(x,t,y)$-almost everywhere, then as $\veps \to 0$ the sequence $\ue$ converges to $u$ almost everywhere with respect to $(x,t,y)$, where $u(x,t,y)$ is a (suitably defined) solution of the Cauchy problem
\begin{equation}\label{eq-ReducedSystemIntroduction-Random}
\begin{cases}
    \dt u + \dx f(u) = 0 \\
    u(x,0,y) = u_0(x,y)\,,
\end{cases}
\end{equation}
with $u_0$ the limit of $\une$ as $\veps \to 0$ in a suitable sense.

In other words, we wish to ``repeat" the deterministic arguments and generalize the results to the case of parametrized initial data.
However, the appropriate generalization is not straightforward. There are two strategies considered here:
\begin{enumerate}
    \item[1)] For a fixed $y \in \bbR^N$, apply the deterministic results directly. 
    \item[2)] Consider solutions belonging to the spaces $L^p(\bbR^d \times (0,T) \times \bbR^N)$ and repeat the proofs of the deterministic results in order to verify that the results hold in the case of the new function spaces.
\end{enumerate}

If we consider the main result using approach 1), the conclusion is that for each fixed $y$
\begin{equation}
u(x,t,y) := \lim\limits_{\veps_y \to 0} u^{\veps_y}(x,t,y)\:,
\end{equation}
where the convergence is in $(x,t)$-measure and $\veps_y$ is a subsequence of $\veps$ that depends on $y$.  However, this approach presents several challenges.  First, the fact that $\veps_y$  depends on $y$ means that the solution $u$ may not be a measurable function in $y$.  Moreover, solutions $\ue$ may not be continuous in $y$.  Indeed even if $\une(\cdot, y_1) - \une(\cdot, y_2)$ is small (in some suitable norm) for two given values $y_1$ and $y_2$, there is no general stability result that ensures $\ue(t, \cdot, y_1) - \ue(t,\cdot, y_2)$ to remain small at later times \cite{Dafermos}.  If one can establish {\it a priori} that the point-wise limit in $y$ is unique, then some additional structure may be available.  For example, in \cite{MishraSchwab,MishraSchwabSukys} a stochastic version of the deterministic theory of conservation laws is developed.   However, it is not known whether the relaxation limit of general $2 \times 2$ systems is the unique entropy solution of the scalar equation \eqref{eq-ReducedSystemIntroduction-Random}.  In simple cases, such as the semi-linear $p$-system, uniqueness can be established.  We consider this system as a special case of the general results studied here.
The direct application of the point-wise limit may be possible for the semi-linear \textit{p}-system (see \cite{NataliniPSystem}); this will be the subject of a future work.
In the current work, we do appeal to approach 1) so far as the entropy structure for the system \eqref{eq-FullSystemIntroduction-Random} is concerned.  We assume the same $L^{\infty}$ bounds on the solutions as in \cite{CLL}, and thus the domains of the entropies considered here are identical to those in \cite{CLL}.

Our motivation for approach 2) is the theoretical paradigm developed recently for the numerical analysis of stochastic partial differential equations.
We formulate definitions of \textit{random} entropy solutions to \eqref{eq-FullSystemIntroduction-Random} analogous to the definitions of weak solutions to random second-order linear elliptic \cite{Webster2008Sparce, babuvska2007collocation, babuvska2005finite-element} and hyperbolic \cite{motamed2013stochastic} equations.
In those works, an assumption is made on the random components of the PDE that makes the analysis more manageable. We make the same type of assumption; namely, we assume that the randomness in the initial data is determined by finitely many random variables via the introduction of the parameter $y \in \bbR^N$ (to be made precise in Section \ref{sec-Preliminaries}). This assumption allows us to make use of the techniques used in the analysis of the zero relaxation limit.

The theory for second-order linear elliptic and hyperbolic SPDE is a straightforward generalization of the deterministic case, and well-posedness is the application of appropriate representation or fixed-point theorems. As a result the relevant theory can be addressed as a part of the numerical analysis. The problem of examining the limit of solutions of a system of first-order nonlinear hyperbolic PDE is not as straightforward.  Scalar conservation laws have a complete and well-known theory of well-posedness \cite{Kruzkov}, but for $2 \times 2$ systems like \eqref{eq-FullSystemIntroduction} only partial results are known.  We quote \cite[page 6]{Bressan}: 

\begin{quotation}
The strong nonlinearity of the equations and the lack of regularity of solutions, also due to the absence of second order terms that could provide a smoothing effect, account for most of the difficulties encountered in a rigorous mathematical analysis of the system [of conservation laws]. It is well known that the main techniques of abstract functional analysis do not apply in this context. Solutions cannot be represented as fixed points of continuous transformations, or in variational form, as critical points of suitable functionals. Dealing with vector valued functions, comparison principles based on upper or lower solutions cannot be used. Moreover, the theory of accretive operators and contractive nonlinear semigroups works well in the scalar case, but does not apply to systems. For the above reasons, the theory of hyperbolic conservation laws has largely developed by \textit{ad hoc} methods...
\end{quotation}

This work is a self-contained discussion designed to provide a rigorous basis for later numerical analysis and other applications.   In addition to verifying the program of zero relaxation in the uncertain initial data paradigm, this work serves as a reference to many significant results treated in different contexts throughout the literature. The analysis here relies heavily on known results in the theory of systems of hyperbolic conservation laws \cite{Dafermos, Serre, Kruzkov, Evans, CLL} adapted to fit the paradigm of parametrized PDE that we consider here.   Indeed, the major analytical work has been done in the deterministic case, with results ranging from the well-known and widely-restated proof of uniqueness of solutions to scalar conservation laws \cite{Kruzkov} to the rather technical application of compensated compactness methods for treating the zero relaxation limit \cite{CLL, LattanzioMarcati, YunguangLuBook}.   The main concern in the adaptation of these results is the presence of a PDF $\rho$ in the relevant integral definitions.  
The ``doubling of variables" argument in the uniqueness proof leads us to double the parameter variable $y$ as well, leading to an $L^1_{loc}$ stability estimate with weight $\rho^2$.
 In developing the necessary convergence framework via compensated compactness, one can begin by either considering weak continuity of determinants (as in \cite{YunguangLuBook}) or by considering the Young measure approach (as in \cite{Dafermos}). In the case of uncertain initial data, the presence of $\rho$ lends to the use of the former.

We follow the techniques developed in \cite{CLL} for obtaining the analogous convergence result for $2 \times 2$ systems of hyperbolic conservation laws in the deterministic case. 
The discussion above regarding the uniqueness of the zero relaxation limit treated in \cite{LattanzioMarcati} is just as relevant in the case of uncertain initial data. We formulate sufficient conditions on solutions of the system \eqref{eq-FullSystemIntroduction-Random} following those found in \cite{LattanzioMarcati} in order to guarantee that the relaxation limit solving \eqref{eq-ReducedSystemIntroduction-Random} is unique.
Motivated by applications in numerical analysis, we consider the example of the semi-linear \textit{p}-system (see Section \ref{sec-SemiLinearPSystem} below for the definition). We adapt and combine the analysis from \cite{Dafermos, Serre, NataliniQuasilinearSystems, NataliniPSystem, Kruzkov} to prove global existence and uniqueness of suitably-defined weak solutions to the semi-linear $p$-system, verify the sufficient conditions required for the uniqueness in the zero relaxation limit, and then prove convergence to a unique entropy solution of the scalar equilibrium conservation law.

This manuscript is an extended version of a paper submitted for publication under a similar title \cite{ScottLaiuHauck:Analysis}.  The purpose of this extended version is to provide additional details for technical results that may be found in the literature.
The details for such results are organized here for the convenience of the reader.  Original results and strategies for proving them are cited appropriately.  Often the direct application of deterministic results (such as the compactness results in Section \ref{subsec-compactness}) will have no proof, but in some places additional details are provided for the purpose of completeness and to supplement the original proofs. Results with technical justifications that would otherwise interrupt the flow of ideas in the paper (such as the well-posedness of classical solutions to the viscosity approximation of the semi-linear $p$-system) are stated and proven in appendices.

This paper is organized as follows.  In  Section \ref{sec-Preliminaries}, we introduce the relaxation system, formulate structural properties and assumptions, and establish notations.
Section \ref{sec-Justification} contains our main result: convergence of solutions to \eqref{eq-FullSystemIntroduction-Random} to a solution of \eqref{eq-ReducedSystemIntroduction-Random}.
In Section \ref{sec-Uniqueness} we formulate and prove sufficient conditions for the uniqueness of the relaxation limit.
In Section \ref{sec-SemiLinearPSystem} we consider the example of the semi-linear \textit{p}-system.
Appendices \ref{apdx-QLParabolicSystems}, \ref{apdx-SLParabolicSystems} and \ref{apdx-SLHyperbolicSystems} all contribute to the well-posedness theory of diagonal weakly-coupled hyperbolic systems related to the semi-linear $p$-system via a coordinate change. Appendices \ref{apdx-Step1} and \ref{apdx-Uniqueness} contain longer, more detailed technical proofs.

\section{Set-Up and Preliminaries}\label{sec-Preliminaries}
Let $y := (y_1, y_2, \ldots, y_N) \in \bbR^N$ which takes values in a bounded subset $\Gamma \subset \bbR^N$. (Without loss of generality let $\Gamma = [-1,1]^N$.) The joint probability density function $\rho(y) := \prod_{j=1}^N \rho_j(y_j) : \Gamma \to [0,\infty)$ is assumed to be in $L^{\infty}(\Gamma)$. We also assume without loss of generality that $\supp \rho := \overline{ \{ y \in \Gamma \, | \, \rho(y) = 0 \}} = \overline{\Gamma}$.
Then the Cauchy problem \eqref{eq-FullSystemIntroduction-Random} with random initial data can be written as
\begin{equation}\label{eq-FullSystemRandom}
    \begin{cases}
    \dt \Ue + \dx F(\Ue) + \frac{1}{\veps} 
        \cR(\Ue)
     = 0 \\
    \Ue(x,0,y) = \Une(x,y)\:,
    \end{cases}
\end{equation}
where $\Ue(x,t,y) = (\ue(x,t,y),\ve(x,t,y)) : \bbR \times (0,\infty) \times \Gamma \to \bbR^2$, the flux $F : \bbR^2 \to \bbR^2$, the relaxation term $\cR : \bbR^2 \to \bbR^2$, and the initial data $\Une = (\une,\vne): \bbR \times \Gamma \to \bbR^2$. 
Let $D \subseteq \bbR^2$ be an open convex set. We assume that $\Ue$ takes values in the set $\overline{D}$ for each $\veps$. 
We also assume that $F := (f_1,f_2): \overline{D} \to \bbR^2$ is a $C^2$ function such that the system \eqref{eq-FullSystemRandom} is hyperbolic. Precisely, the $2 \times 2$ matrix $\dU F(U)$ has real eigenvalues and is diagonalizable. Finally, we assume that the $C^2$ relaxation term $\cR$ takes the form $\textstyle \cR(U) = \big( 0, r(U) \big)^{\intercal}$, where $r : \overline{D} \to \bbR$ possesses a set of local equilibria in the following sense: for $(u,v) \in \bbR^2$, denote projection onto the first coordinate by $\pi_1 : \bbR^2 \to \bbR$, i.e., $\pi_1(u,v) = u$; for each $u \in \pi_1(\overline{D})$ there exists a unique $(u,e(u)) \in \overline{D}$ satisfying
\begin{equation}
    r(u,e(u))=0 \,, \qquad \qquad \dv r(u,e(u)) > 0 \,.
\end{equation}
We also assume that $\dvv r(U) \equiv 0$ for all $U \in \overline{D}$. We define the curve of equilibria
\begin{equation}
K := \{ (u,v) \in \overline{D} \, | \, u \in \pi_1(\overline{D})\,, \,  v=e(u) \}.
\end{equation}
Note that $K$ need not contain $(0,0)$. In addition, we assume that the function $e$ is $C^2$.
We also assume the condition
\begin{equation}
    \dv f_1(u,e(u)) \neq 0\,, \qquad \qquad \text{for all } u \in \pi_1(\overline{D})\,.
\end{equation}

Applying the local equilibrium approximation -- namely $\Ue \approx (\ue,e(\ue))$ -- to \eqref{eq-FullSystemRandom} and omitting the index $\veps$, we obtain the conservation law
\begin{equation}\label{eq-ReducedSystemRandom}
    \begin{cases}
    \dt u + \dx f(u) = 0 \\
    u(x,0,y) = u_0(x,y)\,,
    \end{cases}
\end{equation}
where the reduced flux $f$ is defined by
\begin{equation}
f(u) = f_1(u,e(u))\,.
\end{equation}

We expect this approximation to be good when $\veps \ll 1$ and the initial data are reasonable.

\subsection{Stability and the Effect of Entropy}
The stability criterion of the system \eqref{eq-FullSystemRandom} with deterministic initial data is given in \cite{CLL}. Since we are assuming that only the initial data are random, there is no uncertainty incorporated into the flux $F$ or the relaxation term $r$. Thus the stability criterion from \cite{CLL} applies directly; we summarize it here.

The full system \eqref{eq-FullSystemRandom} is assumed to have real and distinct characteristic speeds, i.e., the eigenvalues of $\dU F(U)$ are real and distinct. Specifically, they are given by
\begin{equation}
    \Lambda_{\pm}(U) = \frac{1}{2} \left( \du f_1 + \dv f_2 \pm \sqrt{(\du f_1 - \dv f_2)^2 + 4 \dv f_1 \du f_2} \right)\,,
\end{equation}
and the characteristic speed of the reduced system \eqref{eq-ReducedSystemRandom} is 
\begin{equation}
\lambda(u) = f'(u) = \du f_1(u,e(u)) + \dv f_1(u,e(u))e'(u)\,.
\end{equation}
The aforementioned stability criterion for the system \eqref{eq-FullSystemRandom} can be derived via a Chapman-Enskog expansion, described in detail in \cite{CLL}. Specifically, the first-order $\veps$ correction of \eqref{eq-ReducedSystemRandom} is dissipative provided that the stability condition
\begin{equation}\label{eq-SubcharacteristicCondition-Nonstrict}
        \Lambda_- (u,e(u)) \leq \lambda(u) \leq \Lambda_+ (u,e(u))\,, \quad u \in \pi_1(\overline{D})\,,
\end{equation}
holds.
This condition, called the \textit{subcharacteristic condition}, is guaranteed by the existence of a \textit{convex entropy} for the system \eqref{eq-FullSystemRandom} \cite{CLL}. We restate the definition of entropy first given in \cite{CLL} below.

\begin{definition}[Convex Entropy]\label{def-EntropyForFullSystem}
A $C^2$ function $\eta : \overline{D} \to \bbR$ is called a \textit{convex entropy} for the system \eqref{eq-FullSystemRandom} if
\begin{enumerate}
    \item[i)] $\dUU \eta \, \dU F$ is symmetric on $\overline{D}$, i.e. $\du f_2 \, \dvv \eta - (\dv f_2 - \du f_1) \duv \eta - \dv f_1 \, \duu \eta = 0$ on $\overline{D}$\,,
    \item[ii)] $\dU \eta(U) \cdot \cR(U) \geq 0$ on $\overline{D}$\,, with equality if and only if $U \in K$ iff $\cR(U) = 0$.
    \item[iii)] $\dUU \eta(U) \geq 0$ on $\overline{D}$ in the sense of quadratic forms.
\end{enumerate}
The function $\eta$ is called a \textit{strictly convex entropy} if $\dUU \eta > 0$ on $\overline{D}$ in the sense of quadratic forms, and a \textit{strongly convex entropy} if there exists c constant $c$ such that $\dUU \eta \geq c > 0$ on $\overline{D}$ in the sense of quadratic forms.
\end{definition}

\noindent
Condition i) says that $\dUU \eta(U)$ symmetrizes the system \eqref{eq-FullSystemRandom} via multiplication on the left. 
It ensures the existence of an \textit{entropy flux} $Q: \overline{D} \to \bbR$ such that
\begin{equation}\label{eq-Preliminaries-EntropyFluxDef}
\dU \eta(U) \, \dU F(U) = \dU Q(U)\,, \qquad \text{ for all } U \in \overline{D}\,.
\end{equation}
We refer to the pair of functions $(\eta,Q)$ as an \textit{entropy pair}.
If $U$ is a smooth solution of \eqref{eq-FullSystemRandom} then, using \eqref{eq-Preliminaries-EntropyFluxDef}, multiplication of \eqref{eq-FullSystemRandom} by $\dU \eta(U)$ on the left gives
\begin{equation}
\dt \eta(U) + \dx Q(U) + \frac{1}{\veps} \dv \eta(U) \, r(U) = 0\,.
\end{equation}
Condition ii) states that, in effect, $\cR$ dissipates $\eta$, and characterizes the dynamics of the curve of equilibria subject to the entropy.

As in \cite{CLL}, we assume the (strict) subcharacteristic condition, namely
\begin{equation}\label{eq-SubcharacteristicCondition}
    \Lambda_- (u,e(u)) < \lambda(u) < \Lambda_+ (u,e(u))\,, \quad u \in \pi_1(D)\,.
\end{equation}
Whereas the existence of a convex entropy guarantees the subcharacteristic condition \eqref{eq-SubcharacteristicCondition-Nonstrict},
this stronger assumption \eqref{eq-SubcharacteristicCondition} allows the converse statement to be proved; that is, the existence of a convex entropy pair for the full system \eqref{eq-FullSystemRandom} is guaranteed by \eqref{eq-SubcharacteristicCondition}. This fact is stated precisely in Theorem \ref{thm-PartialConverseForEntropy} in Section \ref{sec-Justification}.
In short, the existence of entropy for the equations \eqref{eq-FullSystemRandom} and \eqref{eq-ReducedSystemRandom} leads to the stability of solutions just as it does in the case of deterministic initial data,
as we will see in Sections \ref{sec-Uniqueness} and \ref{sec-SemiLinearPSystem}.

\subsection{Function Spaces}
Let $d$, $n \in \bbN$, and let $G \subset \bbR^d$. Denote the set of $C^{\infty}$ functions compactly supported on $G$ taking values in $\bbR^n$ by $\big[ C^{\infty}_c(G) \big]^n$. Denote the set of distributions acting on $\big[ C^{\infty}_c(G) \big]^n$ by $\big[ \cD'(G) \big]^n$.
For $1 \leq p < \infty$ and $G \subseteq \bbR^d$ denote the $L^p$ spaces
\begin{equation}
    \big[ L^p(G) \big]^n := \left\lbrace f : G \to \bbR^n \, \Bigg| \, \Vnorm{f}_{[L^p(G)]^n}^p := \intdm{G}{|f(x)|^p}{x} < \infty \right\rbrace\,.
\end{equation}
For $p = \infty$, define
\begin{equation}
    \big[ L^{\infty}(G) \big]^n := \left\lbrace f : G \to \bbR^n \, \Bigg| \, \Vnorm{f}_{[L^{\infty}(G)]^n} := \esssup_{x \in G} |f(x)| < \infty \right\rbrace\,.
\end{equation}
When $G = \bbR^d$, we use the abbreviated notation
$
\Vnorm{f}_p := \Vnorm{f}_{\left[ L^{p}(\bbR^d) \right]^n}\,.
$
For $k \in \bbN$, define the Sobolev spaces
\begin{equation}
    \left[ W^{k,p}(G) \right]^n := \left\lbrace f : G \to \bbR^n \, \Bigg| \, \Vnorm{f}_{[W^{k,p}(G)]^n}^p := \sum_{|\alpha| \leq k}\Vnorm{\p_{\alpha} f}_{[L^p(G)]^n}^p < \infty \right\rbrace\,.
\end{equation}
We define the homogeneous Sobolev spaces $W^{k,p}_0(G) := \overline{C^{\infty}_c(G)}^{W^{k,p}(G)}$, and denote the dual Sobolev spaces
\begin{equation}
    W^{-k,q}(G) := \left( W^{k,q'}_0(G) \right)^*\,, \quad G \subset \bbR^d\,, \quad 1 \leq q \leq \infty\,, \quad q' = \frac{q}{q-1}\,.
\end{equation}
We define the weighted $L^p$ spaces as follows: for $d$, $n \in \bbN$, $1 \leq p < \infty$ and $G \subseteq \bbR^d \times \Gamma$ we write
\begin{equation}
    \left[L^p_{\rho}(G) \right]^n := \left\lbrace f : G \to \bbR^n \, \Bigg| \, f \text{ is } \mathrm{d}x \times \rho \mathrm{d}y - \text{measurable and }\iint\limits_G |f(x,y)|^p \rho(y) \, \mathrm{d}y \, \mathrm{d}x < \infty \right\rbrace
\end{equation}
with norm
\begin{equation}
    \Vnorm{f}_{\left[ L^p_{\rho}(G) \right]^n}^p := \iint\limits_G |f(x,y)|^p \rho(y) \, \mathrm{d}y \, \mathrm{d}x\,.
\end{equation}
Note that the interchange of integrals is allowed in evaluating $\Vnorm{f}_{ \left[ L^p_{\rho}(G) \right]^n}$. When $G = \bbR^d \times \Gamma$ we use the abbreviated notation
$
    \Vnorm{f}_{p;\rho} := \Vnorm{f}_{\left[ L^p_{\rho}(\bbR^d \times \Gamma) \right]^n}\,.
$

Finally, we introduce the weighted Bochner spaces for $1 \leq q < \infty$, $G_1 \subseteq \bbR^d$ and $G_2 \subseteq \Gamma$ as
\begin{equation}
    L^q \left([0,T] ; \left[L^p_{\rho}(G_1 \times G_2) \right]^n \right) := \left\lbrace f = f(x,t,y): G_1 \times [0,T] \times G_2 \to \bbR^n \, \Bigg| \, \intdmt{0}{T}{\Vnorm{f(\cdot,t)}_{[L^p_{\rho}(G_1 \times G_2)]^n}^q}{t} < \infty \right\rbrace,
\end{equation}
with norm
\begin{equation}
    \Vnorm{f}_{L^q([0,T] ; (L^p_{\rho}(G_1 \times G_2))^n)}^q := \intdmt{0}{T}{\Vnorm{f(\cdot,t)}_{[L^p_{\rho}(G_1 \times G_2)]}^q}{t}\,.
\end{equation}
For $q = \infty$ we use the natural definition. When $G_1 = \bbR^d$ and $G_2 = \Gamma$, we use the abbreviated notation
\begin{equation}
    \Vnorm{f}_{q,p;\rho} := \Vnorm{f}_{L^q([0,T] ; [L^p_{\rho}(\bbR^d \times \Gamma)]^n)}\,.
\end{equation}
When $n = 1$, we omit the superscript. Many times throughout the paper we will take a norm in some of the variables while leaving the others fixed. In these cases we write the remaining variables in the arguments of the function. For example, if $U=U(x,t,y) : \bbR \times (0,T) \times \Gamma \to \bbR^2$, then
\begin{equation}
\begin{split}
    \Vnorm{U(x,t)}_{2;\rho} &= \intdm{\Gamma}{|U(x,t,y)|^2 \rho(y)}{y}\,; \\
    \Vnorm{U(t)}_{2;\rho} &= \iintdm{\bbR}{\Gamma}{|U(x,t,y)|^2 \rho(y)}{y}{x}\,.
\end{split}
\end{equation}

\subsection{Solutions}

We can now state precisely the definitions of solutions for the $2 \times 2$ system and the reduced scalar system.

\begin{definition}[Entropy Solutions of \eqref{eq-FullSystemRandom}]\label{def-2x2Solution}
Let $\veps > 0$, $0 < T \leq \infty$. Suppose $\Une \in \big[ L^{\infty}(\bbR \times \Gamma
) \big]^2$ takes values in $\overline{D}$. A function \\ $\Ue \in \big[ L^{\infty}(\bbR \times (0,T) \times \Gamma) \big]^2$ taking values in $\overline{D}$ is a \textit{weak solution} of the full system \eqref{eq-FullSystemRandom} if
\begin{multline}
    \iiintdmt{0}{T}{\bbR}{\Gamma}{\left( \Ue(x,t,y) \cdot \dt \vphi(x,t,y) + F(\Ue(x,t,y)) \cdot \dx \vphi(x,t,y) \right) \rho(y)}{y}{x}{t} \\
    \qquad \qquad - \iiintdmt{0}{T}{\bbR}{\Gamma}{\frac{1}{\veps} r(\Ue(x,t,y)) \vphi_2(x,t,y) \rho(y)}{y}{x}{t} + \iintdm{\bbR}{\Gamma}{\Une(x,y) \cdot \vphi(x,0,y)\rho(y)}{y}{x} = 0
\end{multline}
for every $\vphi = (\vphi_1, \vphi_2) \in \left[C^{\infty}_c(\bbR \times [0,T) \times \Gamma) \right]^2$. Here, $``\cdot"$ denotes the usual Euclidean inner product on $\bbR^2$.
If in addition the solution $\Ue$ satisfies
\begin{multline}\label{eq-DefinitionOfEntropySolution-EntropyInequality}
    \iiintdmt{0}{T}{\bbR}{\Gamma}{\left( \eta(\Ue(x,t,y)) \dt \vphi(x,t,y) + Q(\Ue(x,t,y)) \dx \vphi(x,t,y) \right) \rho(y)}{y}{x}{t} \\
    \qquad - \iiintdmt{0}{T}{\bbR}{\Gamma}{\frac{1}{\veps} \dv \eta(\Ue(x,t,y)) r(\Ue(x,t,y)) \vphi(x,t,y) \rho(y)}{y}{x}{t} \geq 0
\end{multline}
for every convex entropy pair $(\eta,Q)$ over $\overline{D}$ and for every $\vphi \in C^{\infty}_c(\bbR \times (0,T) \times \Gamma)$ with $\vphi \geq 0$, and
\begin{equation}
\lim\limits_{T \to 0^+} \frac{1}{T} \iiintdmt{0}{T}{V}{\Gamma}{|\Ue(x,t,y)-\Une(x,y)| \rho(y)}{y}{x}{t} = 0
\end{equation}
for every $V \Subset \bbR$, then we say that $\Ue$ is a \textit{weak entropy solution} of the full system \eqref{eq-FullSystemRandom}.
\end{definition}

\begin{definition}[Entropy Solutions of \eqref{eq-ReducedSystemRandom}]\label{def-EntropySolution-ReducedSystem}
Let $0 < T \leq \infty$. Suppose $u_0 \in L^{\infty}(\bbR \times \Gamma)$ takes values in $\pi_1(\overline{D})$. A function $u \in L^{\infty}(\bbR \times (0,T) \times \Gamma)$ taking values in $\pi_1(\overline{D})$ is a \textit{weak solution} of the limiting equation \eqref{eq-ReducedSystemRandom} if
\begin{multline}
    \iiintdmt{0}{T}{\bbR}{\Gamma}{\left( u(x,t,y) \, \dt \vphi(x,t,y) + f(u(x,t,y)) \, \dx \vphi(x,t,y) \right) \rho(y)}{y}{x}{t} \\
    \qquad + \iintdm{\bbR}{\Gamma}{u_0(x,y)\vphi(x,0,y)\rho(y)}{y}{x} = 0
\end{multline}
for every real-valued $\vphi \in C^{\infty}_c(\bbR \times [0,T) \times \Gamma)$.
If in addition $u$ satisfies
\begin{equation}
\begin{split}
    &\iiintdmt{0}{T}{\bbR}{\Gamma}{\left( \ell(u(x,t,y)) \dt \vphi(x,t,y) + q(u(x,t,y)) \dx \vphi(x,t,y) \right) \rho(y)}{y}{x}{t} \geq 0
\end{split}
\end{equation}
for every convex function $\ell : \bbR \to \bbR$ that is of class $C^2(\pi_1(\overline{D}))$ with $q : \bbR \to \bbR$ defined as \\ $q(u) = \int^u \ell'(\theta) f'(\theta) \, \mathrm{d} \theta$ and for every $\vphi \in C^{\infty}_c(\bbR \times (0,T) \times \Gamma)$ with $\vphi \geq 0$, and if $u$ satisfies
\begin{equation}
\lim\limits_{T \to 0^+} \frac{1}{T} \iiintdmt{0}{T}{V}{\Gamma}{|u(x,t,y)-u_0(x,y)| \rho(y)}{y}{x}{t} = 0
\end{equation}
for every $V \Subset \bbR$, then we say that $u$ is a \textit{weak entropy solution} of the limiting equation \eqref{eq-ReducedSystemRandom}.
\end{definition}

We call the pair $(\ell,q)$ an \textit{entropy pair} for the equation \eqref{eq-ReducedSystemRandom}. This definition of entropy solution of \eqref{eq-ReducedSystemRandom} guarantees the uniqueness of such a solution:

\begin{theorem}[Uniqueness of Entropy Solutions]\label{thm-Appendix-UniquenessOfEntropySolution}
Suppose $u$ and $\tilde{u}$ are weak entropy solutions of \eqref{eq-ReducedSystemRandom} corresponding to initial data $u_0$, $\tilde{u}_0$ respectively. Then there exists $\alpha > 0$ depending only on $f$ and its derivatives such that for every $m > 0$, $\Gamma_1 \Subset \Gamma$, and $t>0$
\begin{equation}\label{eq-ReducedSystem-L1locStabilityEstimate}
\iintdm{|x|\leq m}{\Gamma_1}{|u(x,t,y)-\tilde{u}(x,t,y)| \rho^2(y)}{y}{x} \leq \iintdm{|x| \leq m + \alpha t}{\Gamma_1}{|u_0(x,y)-\tilde{u}_0(x,y)| \rho^2(y)}{y}{x}\,.
\end{equation}
In particular, for any $u_0 \in L^{\infty}(\bbR \times \Gamma)$ there exists at most one entropy solution to \eqref{eq-ReducedSystemRandom} with initial data $u_0$.
\end{theorem}
\noindent See Appendix \ref{apdx-Uniqueness} for the proof. Note the density $\rho^2$ rather than $\rho$; the proof follows the ``doubling of variables" argument in \cite{Kruzkov}, of which the presence of $\rho^2$ is a direct consequence.

These definitions are an appropriate generalization of entropy solutions for systems of conservation laws;  note that if the initial data is deterministic (i.e. independent of $y$), then Definitions \ref{def-2x2Solution} and \ref{def-EntropySolution-ReducedSystem} coincide with solution definitions found throughout the literature.

\section{Justification of Equilibrium Limit}\label{sec-Justification}
The main goal of this section is to prove the stochastic analogue of \cite[Theorem 4.1]{CLL}.
There are two tasks to be undertaken before the result of convergence is stated and proven. The first is the application of the program of compensated compactness to our parametrized system of hyperbolic conservation laws. See \cite{Dafermos, YunguangLuBook} for a complete picture of this theory in the nonparametrized case. The second is to extend convex functions on $\pi_1(D)$ to convex entropy for the system \eqref{eq-FullSystemRandom}.
Care must be taken with the density $\rho$; we generally treat it as part of the sequence of functions, rather than as a weight in function spaces, but its role will vary from theorem to theorem.

\subsection{Compactness Tools}
\label{subsec-compactness}

\begin{theorem}[Continuity of a $2 \times 2$ Determinant]
\label{thm-ContinuityOfDeterminant}
Let $G \subset \bbR \times (0,\infty) \times \Gamma$ be a bounded open set and suppose $\Phi^{\veps} = (\Phi_1^{\veps}, \Phi_2^{\veps}, \Phi_3^{\veps}, \Phi_4^{\veps}) : G \to \bbR^4$ and $\Phi = (\Phi_1, \Phi_2, \Phi_3, \Phi_4) : G \to \bbR^4$ are vector-valued functions satisfying
$
\Phi^{\veps} \rightharpoonup \Phi\,, \quad \text{ in } \big[ L^2(G) \big]^4
$
and both
\begin{equation} 
\dt \Phi^{\veps}_1 + \dx \Phi^{\veps}_2 \quad \text{and} \quad \dt \Phi^{\veps}_3 + \dx \Phi^{\veps}_4 \quad \text{ are compact in the strong topology of } W^{-1,2}_{loc}(G)\,.
\end{equation} 
Then there exists a subsequence (still denoted $\Phi^{\veps}$) such that
\begin{equation}\det
    \begin{bmatrix}
    \Phi^{\veps}_1 & \Phi^{\veps}_2 \\
    \Phi^{\veps}_3 & \Phi^{\veps}_4 \\
    \end{bmatrix}
\rightharpoonup \det
    \begin{bmatrix}
    \Phi_1 & \Phi_2 \\
    \Phi_3 & \Phi_4 \\
    \end{bmatrix}
\end{equation}
in the sense of distributions.
\end{theorem}
The proof follows from a special case of \cite[Theorem~2.1.4]{YunguangLuBook}.

\begin{theorem}[Interpolation Result]\label{thm-CompactnessInterpolation}
Let $G  \subset \bbR \times (0,\infty) \times \Gamma$ be an open bounded set, and let $q_1$, $q_2$, $q_3$ be constants satisfying $1<q_1 \leq q_2 < q_3 < \infty$. Let $\cA \subset C^{\infty}_c(G)$ such that
\begin{equation} 
\cA \subset \big( \text{compact subset of } W^{-1,q_1}_{loc}(G) \big) \cap \big( \text{bounded subset of } W^{-1,q_3}_{loc}(G) \big)\,.
\end{equation}
Then there exists $\cB$ a compact set of $W^{-1,q_2}_{loc}(G)$ such that $\cA \subset \cB$.
\end{theorem}
The proof is a special case of \cite[Theorem~2.3.2]{YunguangLuBook}.

\begin{theorem}[Murat Lemma]\label{thm-MuratTheorem}
Suppose $G \subset \bbR \times (0,\infty) \times \Gamma$ is an open set, $1<q_1<\infty$, $q_1' = \frac{q_1}{q_1 -1}$. Suppose a sequence $\{ \Phi^{\veps} \} \subset W^{-1,q_1}(G)$ satisfies
$
\Phi^{\veps} \rightharpoonup \Phi^0 \text{ weakly in } W^{-1,q_1}(G) \text{ as } \veps \to 0\,,
$
and
$\Phi^{\veps} \geq 0$
in the sense of distributions, i.e., for all $\vphi \in C^{\infty}_c(G)$, $\vphi \geq 0$ we have
\begin{equation} 
\Vint{\Phi^{\veps},\vphi}_{W^{-1,q_1},W^{1,q_1'}_0} \leq 0\,.
\end{equation}
Then
$
\Phi^{\veps} \to \Phi^0 \text{ strongly in } W^{-1,q_2}_{loc}(G) \text{ as } \veps \to 0, \quad \forall \, q_2 < q_1\,.
$
\end{theorem}
The proof is a special case of \cite[Theorem~2.3.4]{YunguangLuBook}.

\begin{theorem}[An Application of Compensated Compactness]\label{thm-CompensatedCompactnessFramework}
Let $G \subset \bbR \times (0,\infty) \times \Gamma$ be a bounded open set and let $\ue : G \to \bbR$ be a sequence of functions uniformly bounded in $L^{\infty}(G)$ such that
\begin{equation} 
\ue \rightharpoonupOp^* u\,, \qquad f(\ue) \rightharpoonupOp^* \overline{f}\,,
\end{equation}
where $f \in C^2([-M,M])$, $M = \sup_{\veps} \Vnorm{\ue}_{\infty}$. Here, \,$\rightharpoonupOp^*$ denotes weak-star convergence in $L^{\infty}(G)$. Let $\rho$ be the PDF given in Assumption 1.
Suppose that
\begin{equation}\label{eq-CompensatedCompactnessFramework-CompactnessAssumption}
    \big( \dt h_i(\ue) + \dx j_i (\ue) \big) \rho  \text{ is compact in } W^{-1,2}_{loc}(G)\,, \quad i=1,\,2,
\end{equation}
where
\begin{equation} 
(h_1(\theta),j_1(\theta)) = (\theta - k, f(\theta)-f(k)) \quand (h_2(\theta),j_2(\theta)) = \left( f(\theta)-f(k),\intdmt{k}{\theta}{(f'(s))^2}{s} \right)\,,
\end{equation}
with $k \in \bbR$ an arbitrary constant. Then,
\begin{enumerate}
    \item[1)] $\overline{f} = f(u)$ a.e.
    \item[2)] If in addition
    $
    f''(\theta) \neq 0 \text{ for almost every }  \theta \in [-M,M]\,,
    $
    then there exists a subsequence of the $\ue$ that converge to $u$ a.e.\ in $G$.
\end{enumerate}
\end{theorem}

\begin{proof}
There are two paradigms to consider here: one is using the div-curl lemma of Murat and Tartar \cite{tartar1979compensated} and the other is using the weak continuity of the determinant. We adopt the latter and follow the proof of \cite[Theorem~3.1.1]{YunguangLuBook}. We first must verify the assumptions of Theorem \ref{thm-ContinuityOfDeterminant}. Since $G$ is a bounded set, the weak-$L^2$ limits of $\ue$ and $f(\ue)$ exist. Moreover, $L^2(G) \subseteq L^1(G)$ and so by the definition of weak convergence the weak-$L^2$ limits of $\ue$ and $f(\ue)$ are equal almost everywhere to their $L^{\infty}$ weak-star limits $u$ and $\overline{f}$ respectively.
In addition, $h_i(\ue) \rho$ and $j_i(\ue) \rho$ are bounded in $L^2(G)$ for $i = 1,2$, again since $G$ is a bounded set. Thus there exists a subsequence (still denoted by $\veps$) such that
\begin{equation}
    h_1(\ue) \rho \rightharpoonup \overline{h_1(\ue) \rho}\,, \quad h_2(\ue) \rho \rightharpoonup \overline{h_2(\ue) \rho}\,, \quad 
    j_1(\ue) \rho \rightharpoonup \overline{j_1(\ue) \rho}\,, \quad j_2(\ue) \rho \rightharpoonup \overline{j_2(\ue) \rho}\,,
\end{equation}
weakly in $L^2(G)$.
By the compactness assumption \eqref{eq-CompensatedCompactnessFramework-CompactnessAssumption} we can therefore use the weak continuity of the determinant (Theorem \ref{thm-ContinuityOfDeterminant}) with components $\Phi_1^{\veps} = h_1(\ue) \rho$, $\Phi_2^{\veps} = j_1(\ue) \rho$, $\Phi_3^{\veps} = h_2(\ue) \rho$, $\Phi_4^{\veps} = j_2(\ue) \rho$ to conclude that
\begin{equation}\label{eq-CompensatedCompactnessFramework-ConvergenceOfDeterminant}
\det    
    \begin{bmatrix}
    h_1(\ue) \rho & j_1(\ue) \rho \\
    h_2(\ue) \rho & j_2(\ue) \rho \\
    \end{bmatrix}
    \rightharpoonup \det
    \begin{bmatrix}
    \overline{h_1(\ue) \rho} & \overline{j_1(\ue) \rho} \\
    \overline{h_2(\ue) \rho} & \overline{j_2(\ue) \rho} \\
    \end{bmatrix}    
\end{equation}
weakly in $L^2(G)$. We will examine both sides of the limit and obtain cancellation in some of the terms.

We start with the right-hand side. By the definitions of $h_i$ and $j_i$
\begin{equation}\label{eq-CompensatedCompactnessFramework-IandII}
\begin{split}
\det
\begin{bmatrix}
    \overline{h_1(\ue) \rho} & \overline{j_1(\ue) \rho} \\
    \overline{h_2(\ue) \rho} & \overline{j_2(\ue) \rho} \\
\end{bmatrix}
&=
\left(\overline{\rho (\ue-k)} \right) \left( \overline{ \rho \intdmt{k}{\ue}{(f'(s))^2}{s}} \right) - \left( \overline{\rho (f(\ue)-f(k))} \right)^2 := \mathrm{I} - \mathrm{II}.
\end{split}
\end{equation}
We expand $\mathrm{I}$ as follows:
\begin{equation}\label{eq-CompensatedCompactnessFramework-SplittingOfI}
\begin{split}
    \mathrm{I} &= \left( \overline{\rho(\ue-k)} \right) \left( \overline{\rho \intdmt{k}{\ue}{(f'(s))^2}{s}} \right) = \left( \overline{\rho(\ue-k)} \right) \left( \overline{\rho \intdmt{u}{\ue}{(f'(s))^2}{s}} + \rho \intdmt{k}{u}{(f'(s))^2}{s} \right)  \\
    &= \left( \overline{\rho(\ue-u+u-k)} \right) \left( \overline{\rho \intdmt{u}{\ue}{(f'(s))^2}{s}} \right) + \left( \overline{\rho(\ue-k)} \right) \left( \rho \intdmt{k}{u}{(f'(s))^2}{s} \right)\\
    &= \left( \overline{\rho (\ue - u)} \right) \left( \overline{\rho \intdmt{u}{\ue}{(f'(s))^2}{s}} \right) + \left( \rho (u-k) \right) \left( \overline{\rho \intdmt{u}{\ue}{(f'(s))^2}{s}} \right) + \left( \overline{\rho (\ue - k)} \right) \left( \rho\intdmt{k}{u}{(f'(s))^2}{s} \right) \\
    &= \left( \rho (u-k) \right) \left( \overline{\rho \intdmt{u}{\ue}{(f'(s))^2}{s}} \right)  + \left( \overline{\rho (\ue - k)} \right) \left( \rho \intdmt{k}{u}{(f'(s))^2}{s} \right)\,,
\end{split}
\end{equation}
where the final equality follows from the fact that
\begin{equation} 
\overline{\rho(\ue-u)} = \overline{\ue \rho - u \rho} = \iiint\limits_{G}{(\ue \rho - u \rho) \underbrace{\vphi}_{\in L^2(G)}} \, \rmd y \, \rmd x \, \rmd t = \iiint\limits_G{(\ue - u) \underbrace{\rho \vphi}_{\in L^1(G)}}\, \rmd y \rmd x \rmd t = 0\,.
\end{equation}
Meanwhile
\begin{equation}
\begin{split}\label{eq-CompensatedCompactnessFramework-SplittingOfII}
    \mathrm{II} &= \left( \overline{\rho (f(\ue)-f(u) + f(u) - f(k))} \right)^2 \\
    &= \left( \overline{\rho(f(\ue)-f(u))} \right)^2 + \rho^2 (f(u)-f(k))^2 + 2  \left( \overline{\rho (f(\ue)-f(u))} \right) \rho (f(u)-f(k))\,.
\end{split}
\end{equation}
Substituting \eqref{eq-CompensatedCompactnessFramework-SplittingOfI} and \eqref{eq-CompensatedCompactnessFramework-SplittingOfII} into \eqref{eq-CompensatedCompactnessFramework-IandII} gives
\begin{equation}\label{eq-CompensatedCompactnessFramework-ProofRHS}
\begin{split}
\det
	\begin{bmatrix}
    \overline{h_1(\ue) \rho} & \overline{j_1(\ue) \rho} \\
    \overline{h_2(\ue) \rho} & \overline{j_2(\ue) \rho} \\
    \end{bmatrix} &= \left( \overline{\rho (\ue - k)} \right) \left( \rho \intdmt{k}{u}{(f'(s))^2}{s} \right) - \left( \overline{\rho(f(\ue)-f(u))} \right)^2 \\
    &\qquad -  \rho^2 (f(u)-f(k))^2 - 2  \left( \overline{\rho (f(\ue)-f(u))} \right) \rho (f(u)-f(k))\,.
\end{split}
\end{equation}

We now expand the left-hand side of \eqref{eq-CompensatedCompactnessFramework-ConvergenceOfDeterminant}:
\begin{equation}\label{eq-CompensatedCompactnessFramework-ProofLHS}
\begin{split}
\det&
\begin{bmatrix}
    h_1(\ue) \rho & j_1(\ue) \rho \\
    h_2(\ue) \rho & j_2(\ue) \rho \\
\end{bmatrix}
= \rho^2 (\ue - k) \intdmt{k}{\ue}{(f'(s))^2}{s} - \rho^2 (f(\ue)-f(k))^2 \\
	&= \rho^2 \left( (\ue - k) \left( \intdmt{k}{u}{(f'(s))^2}{s} + \intdmt{u}{\ue}{(f'(s))^2}{s} \right) - (f(\ue)-f(u)+f(u)-f(k))^2 \right)\\
	&= \rho^2 \left( (\ue - u) \intdmt{u}{\ue}{(f'(s))^2}{s} + (u-k) \intdmt{u}{\ue}{(f'(s))^2}{s} + (\ue - k)\intdmt{k}{u}{(f'(s))^2}{s} \right) \\
	& \qquad - \rho^2 \left( (f(\ue)-f(u))^2 + 2(f(\ue)-f(u))(f(u)-f(k)) + (f(u)-f(k))^2 \right)\,.
\end{split}
\end{equation}
We then take the weak limit in \eqref{eq-CompensatedCompactnessFramework-ProofLHS} and cancel out common terms from \eqref{eq-CompensatedCompactnessFramework-ProofRHS}. The result is
\begin{equation}\label{eq-CompensatedCompactnessFramework-LeftoverTerms}
    \overline{\rho^2 \left( (\ue-u) \intdmt{u}{\ue}{(f'(s))^2}{s} - (f(\ue)-f(u))^2 \right)} = - \left( \overline{\rho (f(\ue)-f(u))} \right)^2 \leq 0\,.
\end{equation}
We now continue with the argument in \cite{YunguangLuBook} but provide further details here for the reader. By H\"older's inequality, for any $u$, $v \in [-M,M]$
\begin{equation}\label{eq-CompensatedCompactnessFramework-StrictHolder}
(f(v)-f(u))^2 = \left( \intdmt{u}{v}{f'(s)}{s} \right)^2 \leq (v-u) \intdmt{u}{v}{(f'(s))^2}{s} \,.
\end{equation}
Thus the left-hand side of \eqref{eq-CompensatedCompactnessFramework-LeftoverTerms} is nonnegative, and therefore
\begin{equation}\label{eq-CompensatedCompactnessFramework-Equality1}
    \overline{\rho^2 \left( (\ue-u) \intdmt{u}{\ue}{(f'(s))^2}{s} - (f(\ue)-f(u))^2 \right)} = 0
\end{equation}
and
\begin{equation}\label{eq-CompensatedCompactnessFramework-Equality2}
    \left( \overline{\rho (f(\ue)-f(u))} \right)^2 = 0\,.
\end{equation}
It follows immediately from \eqref{eq-CompensatedCompactnessFramework-Equality2} that $\overline{f}=f(u)$ almost everywhere, and claim 1) is established.

To prove claim 2), we will show that $\ue$ converges to $u$ in measure, and thus a subsequence converges almost everywhere. Define $\Upsilon : [-M,M]^2 \to \bbR$ by
\begin{equation} 
\Upsilon(v,w) := (w-v) \intdmt{v}{w}{(f'(s))^2}{s} - (f(w)-f(v))^2\,.
\end{equation}
Since
\begin{equation}\label{eq-CompensatedCompactnessFramework-NonnegativeFxn}
\p_w \Upsilon(v,w) = \intdmt{v}{w}{(f'(w)-f'(s))^2}{s} \geq 0
\end{equation}
for every $w \geq v$, for each fixed $v \in [-M,M)$ the function $\Upsilon(v,\cdot)$ is nondecreasing on $[v,M]$.
Now, let $\alpha > 0$, and assume that there exists $v \in [-M,M)$ such that $v+\alpha < M$. 
An application of H\"older's inequality (see \eqref{eq-CompensatedCompactnessFramework-StrictHolder}) ensures that $\Upsilon(v,v+\alpha) \geq 0$, with equality attained if and only if $f'(s)$ is constant on the interval $[v,v+\alpha]$. However, $\alpha > 0$ and by assumption $f''(s) \neq 0$ almost everywhere on $[-M,M]$, so H\"older's inequality is strict. Therefore $\Upsilon(v,v+\alpha) > 0$, and by \eqref{eq-CompensatedCompactnessFramework-NonnegativeFxn} for every $w \in [-M,M]$ satisfying $w > v + \alpha$
\begin{equation}
\Upsilon(v,w) \geq \Upsilon(v,v+\alpha) := C_{\alpha,f,v}> 0\,.
\end{equation}
Thus, for any $\veps$ and for any $(x,t,y) \in G$ such that $\ue - u > \alpha,$
\begin{equation} 
(\ue - u) \intdmt{u}{\ue}{(f'(s))^2}{s} - (f(\ue)-f(u))^2 \geq C_{\alpha,f,u}\,.
\end{equation} 
Thus, for any $\alpha > 0$
\begin{equation} 
\iiint\limits_{\{ \ue -u > \alpha \}} \left( (\ue - u) \intdmt{u}{\ue}{(f'(s))^2}{s} - (f(\ue)-f(u))^2 \right) \rho^2 \, \rmd y \, \rmd x \, \rmd t \geq C_{\alpha,f,u} \iiint\limits_{\{ \ue - u > \alpha \}} \rho^2 \, \rmd y \, \rmd x \, \rmd t\,.
\end{equation}
However, we have from the first equality \eqref{eq-CompensatedCompactnessFramework-Equality1} that
\begin{equation}
\lim\limits_{\veps \to 0} \iiint\limits_G \left( (\ue - u) \intdmt{u}{\ue}{(f'(s))^2}{s} - (f(\ue)-f(u))^2 \right) \rho^2 \, \rmd y \, \rmd x \, \rmd t = 0\,,
\end{equation}
since constant functions are in $L^1(G)$. Hence,
\begin{equation}
\begin{split}
    \lim\limits_{\veps \to 0} \iiint\limits_{\{ \ue - u > \alpha \}} \rho^2 \, \rmd y \, \rmd x \, \rmd t
    &\leq C \lim\limits_{\veps \to 0} \iiint\limits_{\{ \ue -u > \alpha \}} \left( (\ue - u) \intdmt{u}{\ue}{(f'(s))^2}{s} - (f(\ue)-f(u))^2 \right) \rho^2 \, \rmd y \, \rmd x \, \rmd t \\
    &\leq C \lim\limits_{\veps \to 0} \iiint\limits_G \left( (\ue - u) \intdmt{u}{\ue}{(f'(s))^2}{s} - (f(\ue)-f(u))^2 \right) \rho^2 \, \rmd y \, \rmd x \, \rmd t = 0\,.
\end{split}
\end{equation}
By definition of the PDF $\rho$, $\rho \neq 0$ almost everywhere in $\Gamma$.
Therefore, the only way that the left-hand side integral can converge to $0$ is that for every $\alpha > 0$
\begin{equation} 
\lim\limits_{\veps \to 0} \meas \left( \left\lbrace (x,t,y) \in G \, \Big| \,  \ue-u> \alpha \right\rbrace \right) = 0\,.
\end{equation}
By a similar argument, we can also establish that
\begin{equation}
\begin{split}
    &\lim\limits_{\veps \to 0} \iiint\limits_{\{ \ue - u < - \alpha \}} \rho^2 \, \rmd y \, \rmd x \, \rmd t \\
    &\qquad \qquad \leq C \lim\limits_{\veps \to 0} \iiint\limits_{\{ \ue -u < - \alpha \}} {\left( (\ue - u) \intdmt{u}{\ue}{(f'(s))^2}{s} - (f(\ue)-f(u))^2 \right) \rho^2}\, \rmd y \, \rmd x \, \rmd t \\
    &\qquad \qquad \leq C \lim\limits_{\veps \to 0} \iiint\limits_G {\left( (\ue - u) \intdmt{u}{\ue}{(f'(s))^2}{s} - (f(\ue)-f(u))^2 \right) \rho^2}\, \rmd y \, \rmd x \, \rmd t = 0\,,
\end{split}
\end{equation}
from which it follows that
\begin{equation} 
\lim\limits_{\veps \to 0} \meas \left( \left\lbrace (x,t,y) \in G \, \Big| \,  \ue-u < - \alpha \right\rbrace \right) = 0\,.
\end{equation}
Therefore,
\begin{equation} 
\lim\limits_{\veps \to 0} \meas \left( \left\lbrace (x,t,y) \in G \, \Big| \,  |\ue-u | > \alpha \right\rbrace \right) = 0\,, \text{ for every } \alpha > 0\,.
\end{equation}
Thus a subsequence converges almost everywhere in $G$, and the proof is complete.
\end{proof}

\subsection{Constructing Entropy Extensions}

Here we introduce the assumptions for the entropy of the system, and construct the entropy necessary for invoking the compactness program built in the previous section. To begin, the converse statement for the existence of entropy described in Section \ref{sec-Preliminaries} is as follows:

\begin{theorem}[Entropy Extension; Theorem 3.2 in \cite{CLL}]\label{thm-PartialConverseForEntropy}
Let $(\ell,q)$ be a strictly convex entropy pair for the limiting equation \eqref{eq-ReducedSystemRandom}. Assume that the strict subcharacteristic condition \eqref{eq-SubcharacteristicCondition} holds. Then there exists a strictly convex entropy pair $(\eta,Q)$ for the system \eqref{eq-FullSystemRandom} over an open set $D_{\ell} \subset D$ containing the local equilibrium curve $K$, along which
\begin{equation}\label{eq-EntropyRestrictionToK}
    \eta(u,e(u))=\ell(u)\,, \qquad Q(u,e(u)) = q(u)\,, \qquad \dv \eta(u,e(u)) = 0\,. 
\end{equation}
\end{theorem}
Next we construct two entropy for \eqref{eq-FullSystemRandom} whose difference on $K$ is $f(u)$, up to a constant.

\begin{definition}\label{def-Bgamma}
For any bounded open set $B \Subset D$ and for any $\gamma > 0$, define the set
\begin{equation} 
B_{\gamma}:= B \cap \{ (u,v) \in D \, \big| \, |v-e(u)| < \gamma \}\,.
\end{equation} 
\end{definition}

\begin{lemma}\label{lma-EntropyExistence}
With all the assumptions of Theorem \ref{thm-PartialConverseForEntropy}, for every bounded open set $B \subset \bbR^2$ there is a constant $\gamma > 0$ depending only on the flux $F$ and the entropy $\ell$ such that
$
\overline{B_{\gamma}} \subset D_{\ell}\,.
$
As a result, on the set $\overline{B_{\gamma}}$ we have for some constant $c > 0$

\begin{equation}\label{eq-StrongConvexityConditionsOnBgamma}
    i) \,\, \dUU \eta \geq c\,,  \qquad iii) \,\, \dv \eta(U) \, r(U) > 0 \text{ on } \overline{B_{\gamma}} \setminus K\,.
\end{equation}

\end{lemma}
This lemma is stated in-line in \cite{CLL}. We state it in a lemma here to emphasize the strong convexity conditions on $\eta$. The strong convexity is necessary for the analysis in Theorem \ref{thm-Theorem4.1}, which is the main result in this section.

We can actually prove something stronger regarding the existence of strongly convex entropy pairs for the full system \eqref{eq-FullSystemRandom}. Lemma \ref{lma-EntropyForCompensatedCompactness} will allow us to employ the compensated compactness framework that was set up in the preliminary theorems, which is the key technique in the proof of Theorem \ref{thm-Theorem4.1}.

\begin{lemma}[Construction of Desired Entropy]\label{lma-EntropyForCompensatedCompactness}
Let $D \subseteq \bbR^2$ be an open convex set, and  let $B$, $B'$ be bounded open convex subsets of $D$ such that $B \cap K \neq \emptyset$ and $B \Subset B' \Subset D$. Let $f$ be the reduced flux in \eqref{eq-ReducedSystemRandom}. Let $\ell_1 : \pi_1(D) \to \bbR$ be any strictly convex function, and define
\begin{equation}\label{eq-EntropyForCompensatedCompactness}
    \ell_2(u) = \ell_1(u) + Cf(u)\,, \qquad u \in \pi_1(D)\,,
\end{equation}
where the constant $C>0$ is chosen so that
\begin{equation}\label{eq-EntropyForCompensatedCompactnessProof-BoundOnC}
    C < \frac{\inf\limits_{u \in \overline{\pi_1(B')}} \ell_1''(u)}{\sup\limits_{u \in \overline{\pi_1(B')}} |f''(u)|}\,.
\end{equation}
Assume the strict subcharacteristic condition \eqref{eq-SubcharacteristicCondition}. Then there exists a constant $\gamma > 0$ such that there are two strongly convex entropy pairs for the full system \eqref{eq-FullSystemRandom} denoted $(\eta_i,Q_i)$, $i \in \{1,2\}$, on the compact set $\overline{B_{\gamma}}$ satisfying $\ell_i(u) = \eta_i(u,e(u))$ for every $u \in \pi_1(\overline{B_{\gamma}})$. Moreover, the entropy satisfy 
\begin{equation}\label{eq-EntropyForCompensatedCompactnessProof-RelaxDissipation}
\dv \eta_i(U)r(U) > 0 \text{ on } \overline{B_{\gamma}} \setminus K\,, \qquad \dv \eta_i(u,e(u)) = 0\,.
\end{equation}
\end{lemma}
\begin{proof}
First, the function $q_1(u) := \intdmt{}{u}{\ell_1'(s) f'(s)}{s}$ for $u \in \pi_1(D)$ gives an entropy flux corresponding to $\ell_1$. Therefore, Theorem \ref{thm-PartialConverseForEntropy} applies and there exists $(\eta_1,Q_1)$ and an open set $D_{\ell_1} \subset D$ such that $(\eta_1,Q_1)$ is a strictly convex entropy pair for the full system \eqref{eq-FullSystemRandom} over $D_{\ell_1}$ satisfying \eqref{eq-EntropyRestrictionToK}. Then we can use Lemma \ref{lma-EntropyExistence} to say that there exists $\gamma_1 >0$ such that for the set
\begin{equation} 
B'_{\gamma_1} := B' \cap \{ (u,v) \in D \, \big| \, |v-e(u)| < \gamma_1 \}
\end{equation}
we have $\overline{B'_{\gamma_1}} \subset D_{\ell_1}$ and the entropy $\eta_1$ satisfies the condition \eqref{eq-StrongConvexityConditionsOnBgamma} on $\overline{B'_{\gamma_1}}$.

Now, the bound \eqref{eq-EntropyForCompensatedCompactnessProof-BoundOnC} implies that $\ell_2(u)$ is a strictly convex function on $\pi_1(B')$. Again, $q_2(u) := \intdmt{}{u}{\ell_2'(s) f'(s)}{s}$ for $u \in \pi_1(D)$ is an entropy flux corresponding to $\ell_2$. So Theorem \ref{thm-PartialConverseForEntropy} applies and there exists $(\eta_2,Q_2)$ and an open set $D_{\ell_2} \subset B'$ such that $(\eta_2,Q_2)$ is a strictly convex entropy pair for the full system \eqref{eq-FullSystemRandom} over $D_{\ell_2}$ satisfying \eqref{eq-EntropyRestrictionToK}. Then again we use Lemma \ref{lma-EntropyExistence} to conclude that there exists $\gamma_2>0$ such that for the set 
\begin{equation} 
B_{\gamma_2} := B \cap \{ (u,v) \in B' \, \big| \, |v-e(u)| < \gamma_2 \}\,,
\end{equation}
$\overline{B_{\gamma_2}} \subset D_{\ell_2}$ and the entropy $\eta_2$ satisfies the conditions \eqref{eq-StrongConvexityConditionsOnBgamma} on $\overline{B_{\gamma_2}}$.

Define $\gamma = \frac{1}{2} \min(\gamma_1, \gamma_2)$. Since $B  \Subset B'$ we have that $(\eta_i, Q_i)$ are strongly convex entropy pairs on $\overline{B_{\gamma}}$ for $i = 1,2$, and finally we observe that \eqref{eq-EntropyForCompensatedCompactnessProof-RelaxDissipation} is satisfied.
\end{proof}

\subsection{Convergence Result}

We now arrive at the main result of this section. Following the argument of \cite{CLL} we are able to prove pointwise convergence of solutions $\Ue$ under an \textit{a priori} $L^{\infty}$ assumption and some additional assumptions on the structure of the system \eqref{eq-FullSystemRandom} and the equation \eqref{eq-ReducedSystemRandom}, specifically the strict subcharacteristic condition \eqref{eq-SubcharacteristicCondition} and genuine nonlinearity of the flux $f$.
 
\begin{theorem}[Convergence to Equilibrium]\label{thm-Theorem4.1}
Assume the strict subcharacteristic condition \eqref{eq-SubcharacteristicCondition}, and define $B$, $B'$, $\ell_1$, $\ell_2$ as in Lemma \ref{lma-EntropyForCompensatedCompactness}. Define $B_{\gamma}$ with $\gamma > 0$ to be the set satisfying the conclusions of Lemma \ref{lma-EntropyForCompensatedCompactness} (note that $B_{\gamma}$ is bounded, open, and convex). Suppose that for a sequence $\veps > 0$ the initial data $\Une$ for the system \eqref{eq-FullSystemRandom} takes values in $\overline{B_{\gamma}}$. Suppose further that there exist weak entropy solutions $\Ue$ of the system \eqref{eq-FullSystemRandom} taking values in $\overline{B_{\gamma}}$.

Suppose that there exists a function $\Ubar(y) = (\ubar(y),\vbar(y)) : \Gamma \to K$ and a constant $\beta>0$ independent of $\veps$ such that
\begin{equation}\label{eq-Theorem4.1-L2UniformBoundAssumption}
    \Vnorm{\Une - \Ubar}_{2;\rho} \leq \beta
\end{equation}
for every $\veps$, and that
\begin{equation}
    \lambda'(u) = f''(u) \neq 0 \text{ almost everywhere in } \pi_1(\overline{B_{\gamma}})\,.
\end{equation}
Then there exists a subsequence (still denoted by $\veps$) such that
\begin{equation} 
\Ue \to U \quad \text{ in } \left[ L^p_{loc}(\bbR \times (0,\infty) \times \Gamma) \right]^2\,,
\end{equation}
where $1 \leq p < \infty$ and the limit function $U = (u,v)$ satisfies
\begin{enumerate}
    \item[i)] $v(x,t,y) = e(u(x,t,y))$ for almost every $(x,t,y) \in \bbR \times (0,\infty) \times \Gamma$
    \item[ii)] $u$ is a weak solution of the equilibrium equation \eqref{eq-ReducedSystemRandom} with initial data $u_0$ defined as the weak-star limit of $\une$ in the space $L^{\infty}(\bbR \times \Gamma)$.
\end{enumerate}
\end{theorem}

\begin{proof}
\underline{Step 1}: Throughout the proof we use the following $L^2$ estimate: there exists a constant $C>0$ depending only on $\eta_i$, $F$, $r$ and their derivatives such that
\begin{equation}\label{eq-Theorem4.1Proof-L2Estimate}
    \iiintdmt{0}{\infty}{\bbR}{\Gamma}{|\ve-e(\ue)|^2 \rho}{y}{x}{t} \leq C \veps \iintdm{\bbR}{\Gamma}{|\Une - \Ubar|^2 \rho}{y}{x}\,,
\end{equation}
for every $\veps>0$.
See Appendix \ref{apdx-Step1} for the proof.

\underline{Step 2}: Let $G \subset \bbR \times (0,\infty) \times \Gamma$ be a bounded set. We will show that
\begin{align}
    \big( \dt \ue + \dx f(\ue) \big) \rho &\text{ is compact in } W^{-1,2}_{loc}(G)\,, \label{eq-Theorem4.1CompactEntropy1} \\
    \big( \dt \ell_i(\ue) + \dx q_i(\ue) \big) \rho &\text{ is compact in } W^{-1,2}_{loc}(G)\,, \quad i=1,2\,. \label{eq-Theorem4.1CompactEntropy2}
\end{align}    
This will allow us to use the convergence result in Theorem \ref{thm-CompensatedCompactnessFramework}.

First we show \eqref{eq-Theorem4.1CompactEntropy1}. Because $\Ue$ is a weak solution of \eqref{eq-FullSystemRandom}, we use the first equation in the system \eqref{eq-FullSystemRandom} to write that
\begin{equation}
\begin{split}
    \big( \dt \ue + \dx f(\ue) \big) \rho &= \big( - \dx f_1(\ue,\ve) + \dx f(\ue) \big) \rho \\
   &= \Big( \dx \big[ f_1(\ue,e(\ue)) - f_1(\ue,\ve) \big] \Big) \rho
\end{split}
\end{equation}
in the sense of distributions.
Using Taylor's theorem,
\begin{equation}
\begin{split}
    \big( - \dx f_1(\ue,\ve) + \dx f(\ue) \big) \rho
    &= \Big( \dx \big[ \dv f_1(\ue,\xi)(e(\ue)-\ve) \big] \Big) \rho
\end{split}
\end{equation}
for some $\xi$ on the line segment connecting $e(\ue)$ and $\ve$.
Using H\"older's inequality and the estimate \eqref{eq-Theorem4.1Proof-L2Estimate}, we have that for every $V \Subset G$
\begin{equation}\label{eq-Theorem4.1Proof-CompactnessEstimate5}
\begin{split}
   \Vnorm{ \big( \dt \ue + \dx f(\ue)  \big) \rho }_{W^{-1,2}(V)} &=  \Vnorm{\dx \big[ \dv f_1(\ue,\xi)(e(\ue)-\ve) \big] }_{W^{-1,2}(V)} \\
   &= \sup_{\Vnorm{\vphi}_{W^{1,2}_0(V)}\leq 1} \left| \iiint\limits_V {\big( \dv f_1(\ue,\xi)(e(\ue)-\ve) \big) \rho \, \dx \vphi }\, \mathrm{d}{y}\, \mathrm{d}{x}\, \mathrm{d}{t} \right| \\
   &\leq C \sup_{\Vnorm{\vphi}_{W^{1,2}_0(V)}\leq 1} \iiint\limits_V \big|e(\ue)-\ve \big| \, | \dx \vphi | \, \rho \, \mathrm{d}y \, \mathrm{d}x \, \mathrm{d}t \\
   &\leq C \sup_{\Vnorm{\vphi}_{W^{1,2}_0(V)}\leq 1} \Vnorm{\dx \vphi \sqrt{\rho}}_{L^2(V)} \Vnorm{(e(\ue)-\ve) \sqrt{\rho}}_{L^2(V)} \\
   &\leq C \sqrt{\veps}\,.
\end{split}
\end{equation}
The constant $C$ depends only on $\sup_{B_{\gamma}}{|\dv f_1|}$, $\Vnorm{\rho}_{L^{\infty}(\Gamma)}$, $\beta$, and the constant on the right hand side of \eqref{eq-Theorem4.1Proof-L2Estimate}. Since the left hand side of \eqref{eq-Theorem4.1Proof-CompactnessEstimate5} converges to $0$ as $\veps \to 0$ and since $V \Subset G$ is arbitrary, the compactness result \eqref{eq-Theorem4.1CompactEntropy1} is proved.

Next we show \eqref{eq-Theorem4.1CompactEntropy2}. For $i = 1,2$, and for every $\veps$,
\begin{equation}
\begin{split}
\left( \dt \eta_i (\Ue) + \dx Q_i(\Ue) + \frac{1}{\veps}\dv \eta_i(\Ue)r(\Ue) \right) \rho \leq 0 
\end{split}
\end{equation}
in the sense of distributions, since $\Ue$ is an entropy solution of \eqref{eq-FullSystemRandom}. Therefore, since $r(\ue,e(\ue)) \equiv 0$,
\begin{equation}\label{eq-Theorem4.1Proof-NegativityOfDistributions}
\begin{split}
    \big( \dt \ell_i(\ue) + \dx q_i(\ue) \big) \rho
    &\leq \big( \dt \ell_i(\ue) + \dx q_i(\ue) \big) \rho - \left( \dt \eta_i (\Ue) + \dx Q_i(\Ue) + \frac{1}{\veps}\dv \eta_i(\Ue)r(\Ue) \right) \rho \\
    &= \Big( \dt \big[ \ell_i(\ue) - \eta_i(\Ue) \big] \Big) \rho + \Big( \dx \big[ q_i(\ue) - Q_i(\Ue) \big] \Big) \rho \\
    & \qquad \qquad + \frac{1}{\veps} \dv \eta_i(\Ue) \big( r(\ue,e(\ue)) - r(\Ue) \big) \rho \\
    &:= I^{\veps}_{i1} + I^{\veps}_{i2} + I^{\veps}_{i3}
\end{split}
\end{equation}
in the sense of distributions.

Now let $V \Subset G$. Then by Taylor's theorem and H\"older's inequality,
\begin{equation}
\begin{split}
    \Vnorm{I^{\veps}_{i1}}_{W^{-1,2}(V)} &= \sup_{\Vnorm{\vphi}_{W^{1,2}_0(V)} \leq 1} \left| \iiint\limits_V {(\ell_i(\ue)-\eta_i(\Ue))  \rho \, \,  \dt \vphi }\, \mathrm{d}y \, \mathrm{d}x \, \mathrm{d}t \right| \\
    &\leq \sup_{\Vnorm{\vphi}_{W^{1,2}_0(V)} \leq 1} C \iiint\limits_V {|e(\ue)-\ve| \, | \dt \vphi | \, \rho}\, \mathrm{d}y \, \mathrm{d}x \, \mathrm{d}t \\
    &\leq C \sup_{\Vnorm{\vphi}_{W^{1,2}_0(V)} \leq 1} \Vnorm{\dt \vphi \, \sqrt{\rho}}_{L^2(V)} \Vnorm{\big( e(\ue)-\ve \big) \sqrt{\rho}}_{L^2(V)}.
\end{split}
\end{equation}
Using the $L^2$ estimate \eqref{eq-Theorem4.1Proof-L2Estimate} and the fact that $\rho \in L^{\infty}(\Gamma)$, there exists a constant $C$ depending only on $\eta_i$, $\rho$, and $\ell_i$ such that
\begin{equation} 
\Vnorm{I^{\veps}_{i1}}_{W^{-1,2}(V)} \leq C \sqrt{\veps} \Vnorm{\Une - \Ubar}_{2;\rho}.
\end{equation}
Therefore by assumption \eqref{eq-Theorem4.1-L2UniformBoundAssumption} on the initial data
\begin{equation} 
\Vnorm{I^{\veps}_{i1}}_{W^{-1,2}(V)} \to 0 \quad \text{ as } \veps \to 0.
\end{equation}
Similarly, $I^{\veps}_{i2}$ can be bounded exactly the same way, and
\begin{equation} 
\Vnorm{I^{\veps}_{i2}}_{W^{-1,2}(V)} \to 0 \quad \text{ as } \veps \to 0.
\end{equation}
So $I^{\veps}_{i1}$, $I^{\veps}_{i2}$ are compact in $W^{-1,2}_{loc}(G)$.
Because of the power $\veps^{-1}$ we estimate $I_{i3}^{\veps}$ differently. $I^{\veps}_{i3}$ we estimate as
\begin{equation}
\begin{split} 
    \Vnorm{I^{\veps}_{i3}}_{L^1(V)} &= \frac{1}{\veps} \iiint\limits_V {|\dv \eta_i(\Ue)| \, | r(\ue,e(\ue)) - r(\ue,\ve)| \, \rho}\, \mathrm{d}y \, \mathrm{d}x \, \mathrm{d}t \\
    & =  \frac{1}{\veps} \iiint\limits_V {|\dv \eta_i(\Ue) - \dv \eta_i(\ue, e(\ue))| \, | r(\ue,e(\ue)) - r(\ue,\ve)| \, \rho}\, \mathrm{d}y \, \mathrm{d}x \, \mathrm{d}t \\
    &\leq \frac{C}{\veps} \iiint\limits_V {|e(\ue)-\ve|^2 \rho}\, \mathrm{d}y \, \mathrm{d}x \, \mathrm{d}t \\
    &\leq C  \Vnorm{\Une - \Ubar}^2_{2;\rho} = C \beta^2 = C
\end{split}
\end{equation}
by Taylor's theorem, \eqref{eq-EntropyForCompensatedCompactnessProof-RelaxDissipation} and the $L^2$ estimate \eqref{eq-Theorem4.1Proof-L2Estimate}. The constant $C$ depends only on $\eta_i$, $\rho$, $r$, and $\beta$. Crucially, the constant is independent of $\veps$. So the sequence $I^{\veps}_{i3}$ is bounded in $L^1_{loc}(G)$.
As a consequence of the Rellich-Kondrachov Theorem \cite[Theorem 6.3, Part IV]{adams2003sobolev}, the embedding $ L^1(G) \hookrightarrow W^{-1,p_1}(G) $ is compact for $p_1 \in \left( 1, \frac{(2+N)}{(2+N)-1} \right)$, so we have that 
\begin{equation}
I^{\veps}_{i3} \text{ is compact in } W^{-1,p_1}_{loc}(G) \text{ for any } p_1 \in \left( 1, \frac{2+N}{1+N} \right)\,.
\end{equation}
Since $p_1 < 2$, we have that $W^{-1,2}_{loc}(G) \subset W^{-1,p_1}_{loc}(G)$, and so
\begin{equation}
I^{\veps}_i := I^{\veps}_{i1} + I^{\veps}_{i2} + I^{\veps}_{i3} \text{ is compact in } W^{-1,p_1}_{loc}(G), \quad 1 < p_1 < \frac{2+N}{1+N}\,.
\end{equation}
Now we use Murat's Lemma (Theorem \ref{thm-MuratTheorem} above, with $\Phi_{\veps} = \big( \dt \ell_i(\ue) + \dx q_i (\ue) \big) \rho - I^{\veps}_i$). In \eqref{eq-Theorem4.1Proof-NegativityOfDistributions} we see that $\big( \dt \ell_i(\ue) + \dx q_i (\ue) \big) \rho - I^{\veps}_i \leq 0$ in the distributional sense. Further, using H\"older's inequality we obtain that for every $V \Subset G$, $p \in (1,\infty)$,
\begin{equation}\label{eq-Theorem4.1Proof-DistributionEstimate}
\begin{split}
    \Vnorm{\big( \dt \ell_i(\ue) + \dx q_i (\ue) \big) \rho}_{W^{-1,p}(V)}
    &= \sup_{\Vnorm{\vphi}_{W^{1,p'}_0 (V)} \leq 1} \left| \iiint\limits_V {\big( \ell_i(\ue) \dt \vphi + q_i(\ue) \dx \vphi \big) \rho}\, \mathrm{d}y \, \mathrm{d}x \, \mathrm{d}t \right| \\
    &\leq C \sup_{\Vnorm{\vphi}_{W^{1,p'}_0 (V)} \leq 1} \iiint\limits_V {|\dt \vphi| + |\dx \vphi|}\, \mathrm{d}y \, \mathrm{d}x \, \mathrm{d}t \\
    &\leq C |G|^{1/p},
\end{split}
\end{equation}
where $p' = p/(p -1)$, and $C$ depends on $\ell_i$, $q_i$, and $\rho$. 
The constant $C$ is independent of $\veps$, so the sequence $\big( \dt \ell_i(\ue) + \dx q_i (\ue) \big) \rho$ is bounded in $W^{-1,p_1}_{loc}(G)$.
Therefore, the conditions of Murat's Lemma are satisfied (with  $q_1 = p_1$, $q_2 = p_2$), and we obtain
\begin{equation}
\big( \dt \ell_i(\ue) + \dx q_i(\ue) \big) \rho - I^{\veps}_i \text{ is compact in } W^{-1,p_2}_{loc}(G)\,, \quad p_2 \in (1, p_1)\,,
\end{equation}
and thus
\begin{equation}
\big( \dt \ell_i(\ue) + \dx q_i(\ue) \big) \rho \text{ is compact in } W^{-1,p_2}_{loc}(G)\,, \quad p_2 \in (1, p_1)\,.
\end{equation}
In view of the estimate \eqref{eq-Theorem4.1Proof-DistributionEstimate} we see that actually
\begin{equation}
\big( \dt \ell_i(\ue) + \dx q_i (\ue) \big) \rho \text{ is bounded in } W^{-1,r}_{loc}(G),
\end{equation}
for any $r > 2$. Therefore, by Theorem \ref{thm-CompactnessInterpolation} with constants $q_1=p_2$, $q_2=2$ and $q_3 = r$, we finally have that
\begin{equation}
\big( \dt \ell_i(\ue) + \dx q_i (\ue) \big) \rho \text{ is compact in } W^{-1,2}_{loc}(G)\,, \quad i = 1, \, 2\,,
\end{equation}
which recovers \eqref{eq-Theorem4.1CompactEntropy2}.

\underline{Step 3:} We conclude the theorem. By the condition on the two entropy \eqref{eq-EntropyForCompensatedCompactness} we have \\$\ell_2(u) - \ell_1(u) = Cf(u)$, and because $q_i' = f' \cdot \ell_i'$,
\begin{equation}\label{eq-Theorem4.1CompactEntropy3}
    \left( \dt f(\ue) + \dx \left( \intdmt{}{\ue}{(f'(s))^2}{s} \right) \right) \rho  \text{ is compact in } W^{-1,2}_{loc}(G).
\end{equation}
By \eqref{eq-Theorem4.1CompactEntropy1} and \eqref{eq-Theorem4.1CompactEntropy3} the conditions of Theorem \ref{thm-CompensatedCompactnessFramework} are satisfied, and taking subsequences as necessary we conclude that $\ue$ converges almost everywhere in $G$ to a function $u$. Hence, the Lebesgue Dominated Convergence Theorem ensures that $\ue \to u$ in $L^p_{loc}(\bbR \times (0,\infty) \times \Gamma)$, for $1 \leq p < \infty$. The estimate \eqref{eq-Theorem4.1Proof-L2Estimate} ensures that $v := \lim\limits_{\veps \to 0}\ve = e(u)$ pointwise. Finally, applying the Lebesgue Dominated Convergence Theorem, we get that $u$ is a weak solution to the reduced system \eqref{eq-ReducedSystemRandom}. The proof is complete.
\end{proof}
\subsection{Uniqueness of the Limiting Solution}\label{sec-Uniqueness}
Define $u$ as the limit function obtained in Theorem \ref{thm-Theorem4.1}. Issues arise when we attempt to verify that this $u$ obtained is the unique entropy solution of \eqref{eq-ReducedSystemRandom}. The following, derived from the work of \cite{LattanzioMarcati}, are sufficient conditions in order to verify that $u$ is the unique entropy solution.
\begin{enumerate}
    \item[(H1)] The initial data $(\une, \vne)$ converge almost everywhere to a function $(u_0,v_0)$ taking values in $D$ satisfying $v_0 = e(u_0)$.
    \item[(H2)] For every $\veps$, the solution $\Ue$ to \eqref{eq-FullSystemRandom} is the unique limit in $\big[ L^1_{\rho^2,loc}(\bbR \times (0,\infty) \times \Gamma) \big]^2$ as $\nu \to 0$ of classical solutions $(\uen,\ven):= \Uen$ of the approximate parabolic system
    \begin{equation}\label{eq-FullSystemViscosityApproximation}
    \begin{cases}
        \dt \uen + \dx f_1(\uen,\ven) = \nu \dxx \uen \\
        \dt \ven + \dx f_2(\uen,\ven) + \frac{1}{\veps}r(\uen,\ven) = \nu \dxx \ven \\
        (\uen,\ven)(x,0,y) = (\unen,\vnen)(x,y)\,,
    \end{cases}
    \end{equation}
    where $(\unen,\vnen):= \Unen$ are smooth approximations of $\Une$ converging to $\Une$ in $\big[ L^1_{\rho, loc}(\bbR \times \Gamma) \big]^2$.
    \item[(H3)] For every convex entropy pair $(\ell,q)$ over $\pi_1(D)$, the function $u$ satisfies
    \begin{equation}
        \iiintdmt{0}{\infty}{\bbR}{\Gamma}{\Big( \ell(u) \dt \vphi + q(u) \dx \vphi \Big) \rho}{y}{x}{t} \geq 0
    \end{equation}
    for every $\vphi \in C^{\infty}_c(\bbR^d \times (0,\infty) \times \Gamma)$, $\vphi \geq 0$.
\end{enumerate}

\begin{remark}
Considering (H2), it is possible to prove a result analogous to Theorem \ref{thm-Theorem4.1} with a sequence $\Uen$ solving the parabolic equations \eqref{eq-FullSystemViscosityApproximation} in place of a sequence $\Ue$ solving the hyperbolic system \eqref{eq-FullSystemRandom}; see \cite[Chapter 15]{YunguangLuBook} for the argument in the deterministic case.
\end{remark}

We verify the sufficiency of (H1), (H2), and (H3) in this section. Throughout, we use all of the analysis in the previous sections, like the entropy extension of Theorem \ref{thm-PartialConverseForEntropy} and the definition of $B_{\gamma}$ in Lemma \ref{lma-EntropyForCompensatedCompactness}. 

\begin{lemma}[A Stability Result]\label{lma-Uniqueness-L2StabilityLemma}
Take all the assumptions of Theorem \ref{thm-Theorem4.1}. Without loss of generality we can assume that there exists an entropy $(\eta,Q)$ on $B_{\gamma}$ such that $\eta$ is an extension of $\vphi(u) := |u|^2$. (If this is not the case, the constant $\gamma$ can be further restricted.)
Further assume  (H1), (H2), and that the approximate solutions $\Uen$ and approximating initial data $\Unen$ to \eqref{eq-FullSystemViscosityApproximation} both take values in $\overline{B_{\gamma}}$ for every $\nu$.
Then, for any $V \Subset \bbR$, we have
\begin{equation}\label{eq-Uniqueness-L2Stability}
    \iiintdmt{0}{T}{V}{\Gamma}{|u(x,t,y)|^2 \rho}{y}{x}{t} \leq \iiintdmt{0}{T}{\bbR}{\Gamma}{|u_0(x,y)|^2 \rho}{y}{x}{t} + CT^2\,.
\end{equation}
\end{lemma}
\begin{proof}
Let $\Psi \in C^{\infty}_c(\bbR)$ with  $\Psi \equiv 1$ on $V$ and $0 \leq \Psi \leq 1$ on $\bbR$. Since $\Uen$ is a classical solution, multiplying the system in (H2) by $\dU \eta(\Uen)$ and using that $\dU \eta(\Uen) \, \dx \Uen = \dx [\eta(\Uen)]$ reveals that
\begin{equation} 
\dt \eta(\Uen) + \dx Q(\Uen) + \frac{1}{\veps}\dv \eta(\Uen) r(\Uen) + \nu \big( \Vint{\dUU \eta(\Uen) \dx \Uen, \dx \Uen} - \dxx [\eta(\Uen)] \, \big) = 0\,.
\end{equation}
Therefore, since $\dv \eta(U)r(U) \geq 0$ and since $\dUU \eta(U) > 0$ as a quadratic form,
\begin{equation}
\begin{split} 
    \iintdm{\bbR}{\Gamma}{\Psi & \eta(\Uen(x,t,y)) \rho}{y}{x} - \iintdm{\bbR}{\Gamma}{\Psi \eta(\Unen) \rho}{y}{x} \\
        &= - \iiintdmt{0}{t}{\bbR}{\Gamma}{\Psi \dx Q(\Uen) \rho}{y}{x}{s} - \frac{1}{\veps}\iiintdmt{0}{t}{\bbR}{\Gamma}{\dv \eta(\Uen) r(\Uen) \rho}{y}{x}{s} \\
        &\quad - \nu \iiintdmt{0}{t}{\bbR}{\Gamma}{\Vint{\dUU \eta(\Uen) \dx \Uen, \dx \Uen} \Psi \rho }{y}{x}{s} + \nu \iiintdmt{0}{t}{\bbR}{\Gamma}{\Psi \dxx [\eta(\Uen)] \rho}{y}{x}{s} \\
    &\leq \iiintdmt{0}{t}{\bbR}{\Gamma}{\Psi'(x) Q(\Uen) \rho}{y}{x}{s} + \nu \iiintdmt{0}{t}{\bbR}{\Gamma}{\Psi''(x) \eta(\Uen) \rho}{y}{x}{s} \\
    &\leq \Vnorm{\Psi'}_{L^1(\bbR)} \Vnorm{Q}_{L^{\infty}(B_{\gamma})} \Vnorm{\rho}_{L^1(\Gamma)}t + \nu \Vnorm{\Psi''}_{L^1(\bbR)} \Vnorm{\eta}_{L^{\infty}(B_{\gamma})} \Vnorm{\rho}_{L^1(\Gamma)}t\,.
\end{split}
\end{equation}
Taking $\nu \to 0$ and integrating over $[0,T]$, we obtain using (H2) and the Lebesgue Dominated Convergence Theorem
\begin{equation}
\begin{split}
    \iiintdmt{0}{T}{\bbR}{\Gamma}{\Psi \eta(\Ue) \rho}{y}{x}{t} - \iiintdmt{0}{T}{\bbR}{\Gamma}{\Psi \eta(\Une) \rho}{y}{x}{t} \leq CT^2\,.
\end{split}
\end{equation}
As $\veps \to 0$, using Theorem \ref{thm-Theorem4.1}, (H1) and the Lebesgue Dominated Convergence Theorem we obtain \begin{equation}
    \iiintdmt{0}{T}{\bbR}{\Gamma}{\Psi(x)|u(x,t,y)|^2 \rho}{y}{x}{t} \leq \iiintdmt{0}{T}{\bbR}{\Gamma}{\Psi(x)|u_0(x,y)|^2 \rho}{y}{x}{t} + CT^2\,.
\end{equation}
Using the definition of $\Psi$ we obtain \eqref{eq-Uniqueness-L2Stability}.
\end{proof}

\begin{theorem}[Verifying Uniqueness]\label{thm-UniquenessOfLimitingSolution}
With all the assumptions of Lemma \ref{lma-Uniqueness-L2StabilityLemma}, we have
\begin{equation}
    \lim\limits_{T \to 0^+} \frac{1}{T} \iiintdmt{0}{T}{V}{\Gamma}{|u(x,t,y)-u_0(x,y)|\rho(y)}{y}{x}{t} = 0
\end{equation}
for every $V \Subset \bbR$. Further, if (H3) aslo holds, then $u$ is the unique weak entropy solution of the reduced equation \eqref{eq-ReducedSystemRandom} with initial condition $u_0$.
\end{theorem}
\begin{proof}
Let $d(u) := |u|^2$. Let $D(u,u_0) = d(u)-d(u_0)-d'(u_0)(u-u_0) = (u-u_0)^2$.
For any $V \Subset \bbR$, Jensen's inequality implies that
\begin{equation}\label{eq-Uniqueness-Estimate1}
\begin{split}
    \frac{1}{T \, |V| \, |\Gamma|} \iiintdmt{0}{T}{V}{\Gamma}{|u-u_0| \, \rho }{y}{x}{t} &\leq C \left( \frac{1}{T} \iiintdmt{0}{T}{V}{\Gamma}{|u-u_0|^2 \, \rho }{y}{x}{t} \right)^{1/2} \\
    &= C \left( \frac{1}{T} \iiintdmt{0}{T}{V}{\Gamma}{D(u,u_0) \, \rho }{y}{x}{t} \right)^{1/2}\,,
\end{split}
\end{equation}
where the constant $C$ depends on $V$, $\Gamma$ and $\rho$. Now, let $(g_n) \subset C^{\infty}_c(\bbR \times \Gamma)$ be a sequence of functions such that
\begin{equation}
\iintdm{V}{\Gamma}{|g_n - d'(u_0)| \rho}{y}{x} \to 0\,, \quad n \to \infty\,.
\end{equation}
Then by Lemma \ref{lma-Uniqueness-L2StabilityLemma},
\begin{equation}
\begin{split}
    \iiintdmt{0}{T}{V}{\Gamma}{& D(u,u_0) \rho}{y}{x}{t} \\
    &= \int_{0}^{T} \int_{V} \int_{\Gamma}(d(u)-d(u_0)-d'(u_0)(u-u_0) + g_n \cdot (u-u_0) - g_n \cdot (u-u_0) ) \rho \, \mathrm{d}y \, \mathrm{d}x \, \mathrm{d}t \\
    &\leq \iiintdmt{0}{T}{V}{\Gamma}{|u|^2 \rho}{y}{x}{t} - \iiintdmt{0}{T}{V}{\Gamma}{|u_0|^2 \rho}{y}{x}{t} \\
    &\quad + T \left( \Vnorm{u}_{\infty} + \Vnorm{u_0}_{\infty} \right) \cdot \Vnorm{g_n - d'(u_0)}_{L^1_{\rho}(V \times \Gamma)} + \iiintdmt{0}{T}{V}{\Gamma}{g_n(u_0-u) \rho}{y}{x}{t} \\
    &\leq \iiintdmt{0}{T}{\bbR \setminus V}{\Gamma}{|u_0|^2 \rho}{y}{x}{t} + CT^2 \\
    &\qquad + T \left( \Vnorm{u}_{\infty} + \Vnorm{u_0}_{\infty} \right) \Vnorm{g_n - d'(u_0)}_{L^1_{\rho}(V \times \Gamma)} + \iiintdmt{0}{T}{V}{\Gamma}{g_n(u-u_0) \rho}{y}{x}{t}\,.
\end{split}
\end{equation}
Let $\{ V_i \}$ be a nested sequence of compact sets such that $\cup V_i = \bbR$ and such that $V \subset V_1$. Then for any $\delta >0$ there exists $i$ and $n_i$ large such that 
\begin{equation}
\iiintdmt{0}{T}{\bbR \setminus V_i}{\Gamma}{|u_0|^2 \rho}{y}{x}{t} + T \left( \Vnorm{u}_{\infty} + \Vnorm{u_0}_{\infty} \right) \Vnorm{g_{n_i} - d'(u_0)}_{L^1_{\rho}(V_i \times \Gamma)} < \delta\,.
\end{equation}
\noindent By repeating the above argument for this $i$, since
\begin{equation}
\iiintdmt{0}{T}{V}{\Gamma}{D(u,u_0) \rho}{y}{x}{t} \leq \iiintdmt{0}{T}{V_i}{\Gamma}{D(u,u_0) \rho}{y}{x}{t}\,,
\end{equation}
we obtain
\begin{equation}\label{eq-Uniqueness-Estimate2}
\begin{split}
    \frac{1}{T} \iiintdmt{0}{T}{V}{\Gamma}{D(u,u_0)}{y}{x}{t}
    &\leq CT + \delta + \frac{1}{T} \iiintdmt{0}{T}{\bbR}{\Gamma}{g_n(u-u_0) \rho}{y}{x}{t}\,,
\end{split}
\end{equation}
for $\delta > 0$ arbitrary. Thanks to estimates \eqref{eq-Uniqueness-Estimate1} and \eqref{eq-Uniqueness-Estimate2}, it suffices to show
\begin{equation}\label{eq-Uniqueness-Estimate3}
    \lim\limits_{T \to 0^+} \frac{1}{T} \iiintdmt{0}{T}{\bbR}{\Gamma}{g_n(u_0-u) \rho}{y}{x}{t} = 0\,,
\end{equation}
for each fixed $n \in \bbN$.
Let $\Uen$ be the solution to the parabolic equations \eqref{eq-FullSystemViscosityApproximation} with initial data $\Unen$. Then using this,
\begin{equation}
\begin{split}
    \frac{1}{T} \iiintdmt{0}{T}{\bbR}{\Gamma}{g_n(\unen - \uen) \rho}{y}{x}{t}
    &= - \frac{1}{T} \iiintdmt{0}{T}{\bbR}{\Gamma}{\intdmt{0}{t}{\ds \uen \, g_n \,  \rho}{s}}{y}{x}{t} \\
    &= \frac{1}{T}\iiintdmt{0}{T}{\bbR}{\Gamma}{\intdmt{0}{t}{ \left( \dx f_1(\uen,\ven) - \nu \dxx \uen \right) g_n \rho}{s}}{y}{x}{t} \\
    &= - \frac{1}{T} \iiintdmt{0}{T}{\bbR}{\Gamma}{\intdmt{0}{t}{ \left( f_1(\uen,\ven) \dx g_n + \nu \uen \dxx g_n \right) \rho}{s}}{y}{x}{t} \\
    &\leq C_n T\,,
\end{split}
\end{equation}
since $g_n \in C^{\infty}_c(\bbR \times \Gamma)$, and since $(u,v)$ takes values in $\overline{B_{\gamma}}$. Note that the sequence of constants $C_n$ may be unbounded in $n$, but is independent of $\veps$ and $\nu$.
Therefore by (H1) and (H2), the Lebesgue Dominated Convergence Theorem gives
\begin{equation}
\frac{1}{T} \iiintdmt{0}{T}{\bbR}{\Gamma}{g_n(u_0-u) \rho}{y}{x}{t} = \lim\limits_{\veps, \nu \to 0} \frac{1}{T} \iiintdmt{0}{T}{\bbR}{\Gamma}{g_n(\unen - \uen) \rho}{y}{x}{t} \leq C_n T\,,
\end{equation}
and \eqref{eq-Uniqueness-Estimate3} follows.
\end{proof}
\section{Example: The Semi-linear \textit{p}-system}\label{sec-SemiLinearPSystem}
Our goal here is to apply the analysis of the previous sections to the case of the semi-linear \textit{p}-system
\begin{equation}\label{eq-PSystemSemilinearFull}
\begin{cases}
    \dt \ue + \dx \ve = 0 \\
    \dt \ve + a^2 \dx \ue + \frac{1}{\veps}(\ve-f(\ue)) = 0 \\
    (\ue,\ve)|_{t=0} = (\une,\vne)
\end{cases}
\end{equation}
where $a$ is some positive constant.
We assume that $f \in C^2(\bbR)$ with $f(0)=f'(0)=0$.
The subcharacteristic condition \eqref{eq-SubcharacteristicCondition} becomes
\begin{equation}\label{eq-PSystemSemilinear-SubcharacteristicCondition}
    -a < f'(u) < a\,, \quad u \in \bbR\,.
\end{equation}
The curve of local equilibria is the set $K := \{ (u,v) \in \bbR^2 \, | \, u \in \bbR\,,\,  v = f(u) \}$.
The approximate parabolic system \eqref{eq-FullSystemViscosityApproximation} for the semi-linear \textit{p}-system is
\begin{equation}\label{eq-PSystemSemilinearFull-ViscosityApproximation}
\begin{cases}
    \dt \uen + \dx \ven = \nu \dxx \uen \\
    \dt \ven + a^2 \dx \uen + \frac{1}{\veps}(\ven-f(\uen)) = \nu \dxx \ven \\
    (\uen,\ven)|_{t=0} = (\unen,\vnen)\,.
\end{cases}
\end{equation}

This analysis is a combination of the arguments found in \cite{Dafermos, Serre, NataliniPSystem, NataliniQuasilinearSystems}.
Much of what we have done in the previous sections rely on apriori assumptions made on the solutions $\Ue$. Specifically, we assumed that $\Ue$ are entropy solutions, uniformly bounded in $L^{\infty}$, and the limits of vanishing viscosity solutions $\Uen$. In this section we show that -- with enough assumptions on the data $\Une$ -- there exist unique solutions, one for each $\veps$, of the semi-linear \textit{p}-system satisfying suitable versions of these criteria. We obtain $L^{\infty}$-stability results for solutions of \eqref{eq-PSystemSemilinearFull}, which is used to verify the uniform boundedness of $\Ue$ in $L^{\infty}$, a necessary condition for invoking Theorem \ref{thm-Theorem4.1}. For data $\Une$ in $L^1 \cap L^{\infty}$ we can verify (H2) and (H3) for the \textit{p}-system. For data only in $L^{\infty}$, however, we are not able to verify (H2), but instead show that the $p$-system satisfies an alternative version of the assumption, denoted (H2)$'$. In what follows we consider the case of data in $L^{\infty}$ only.

\begin{enumerate}
\item[(H2)$'$] Let $\Une \in \big[ L^{\infty}(\bbR \times \Gamma) \big]^2$. For every $\veps$, the solution $\Ue$ to \eqref{eq-FullSystemRandom} is the unique limit in $\big[ L^1_{\rho^2,loc}(\bbR \times (0,\infty) \times \Gamma) \big]^2$ as $j \to \infty$ of solutions $U^{\veps,j} \in C \left( [0,\infty); \big[ L^1_{\rho,loc}(\bbR \times \Gamma) \big]^2 \right)$ to \eqref{eq-FullSystemRandom} with initial data $U^{\veps,j}_0 := \chi_j \Une$, where $\chi_j(x,y)$ is the characteristic function on the set $[-j,j] \times \Gamma$. These solutions are in turn, for each $j$, the unique limit in $C \left( [0,\infty); \big[ L^1_{\rho,loc}(\bbR \times \Gamma) \big]^2 \right)$ as $\nu \to 0$ of classical solutions $U^{\veps,\nu,j}$ to the approximate parabolic system \eqref{eq-FullSystemViscosityApproximation} with initial data $U^{\veps,\nu,j}_0$, where $U^{\veps,\nu,j}$ are compactly supported $C^{\infty}$ approximations of $U^{\veps,j}_0$ converging to $U^{\veps,j}_0$ in $\big[ L^1_{\rho}(\bbR \times \Gamma) \big]^2$.
\end{enumerate}

\begin{theorem}
The results of Section \ref{sec-Uniqueness}, in particular Lemma \ref{lma-Uniqueness-L2StabilityLemma} and Theorem 
\ref{thm-UniquenessOfLimitingSolution}, remain valid with the assumption (H2) replaced by (H2)$'$.
\end{theorem}

This theorem is not an obvious corollary of Lemma \ref{lma-Uniqueness-L2StabilityLemma} and Theorem 
\ref{thm-UniquenessOfLimitingSolution}. The proof, however, consists of nothing but an identical repetition of the arguments given in Section \ref{sec-Uniqueness}. This can be verified without much difficulty by a line-by-line review of the proofs.

The key technique in verifying (H2)$'$ and (H3) is to ``diagonalize" the systems \eqref{eq-PSystemSemilinearFull} and \eqref{eq-PSystemSemilinearFull-ViscosityApproximation} and to appeal to the theory for scalar conservation laws. In terms of these Riemannian coordinates
\begin{equation}
    A:= \begin{bmatrix}
        -a & 1 \\
        -a & -1 \\
        \end{bmatrix}\,,
    \qquad
    	\We = (\we,\ze) := A \, \Ue\,, \qquad \Wne = (\wne, \zne) := A \, \Une\,,
\end{equation}
the system \eqref{eq-PSystemSemilinearFull} takes the form
\begin{equation}\label{eq-PSystemSemiLinear-RiemannianCoordinates}
\begin{cases}
    \dt \We + \diag(-a,a) \dx \We = \frac{1}{\veps} \cH(\We) \\
    \We|_{t=0} = \Wne\,,
\end{cases}
\end{equation}
where
\begin{equation}
    \cH(W) := \Big( -H(W), \, H(W) \Big)\,; \qquad 
    H(w,z) := \frac{w - z}{2} - f \left( - \frac{w + z}{2a} \right)\,.
\end{equation}
Similarly, the approximate parabolic system \eqref{eq-PSystemSemilinearFull-ViscosityApproximation} takes the form
\begin{equation}\label{eq-PSystemSemiLinear-ViscosityApproximation-RiemannianCoordinates}
\begin{cases}
    \dt \Wen + \diag(-a,a) \, \dx \Wen = \nu \dxx \Wen +  \frac{1}{\veps} \cH(\Wen) \\
    \Wen|_{t=0} = \Wnen\,.
\end{cases}
\end{equation}

For the rest of this section we prove results for the systems \eqref{eq-PSystemSemilinearFull} and \eqref{eq-PSystemSemilinearFull-ViscosityApproximation} via the diagonalized systems \eqref{eq-PSystemSemiLinear-RiemannianCoordinates} and \eqref{eq-PSystemSemiLinear-ViscosityApproximation-RiemannianCoordinates}, respectively.
The proofs for the systems \eqref{eq-PSystemSemiLinear-RiemannianCoordinates} and \eqref{eq-PSystemSemiLinear-ViscosityApproximation-RiemannianCoordinates} can be found in Appendices \ref{apdx-QLParabolicSystems}, \ref{apdx-SLParabolicSystems} and \ref{apdx-SLHyperbolicSystems}.
Convex entropy pairs for the system \eqref{eq-PSystemSemiLinear-RiemannianCoordinates} are defined just as in Definition \ref{def-EntropyForFullSystem}, but with $\bbR^2$ in place of $\overline{D}$ and $\left\lbrace (w,z) \in \bbR^2 \, \Big| \, \frac{w-z}{2} = f\left( - \frac{w+z}{2} \right) \right\rbrace$ in place of $K$. Weak solutions and weak entropy solutions to \eqref{eq-PSystemSemiLinear-RiemannianCoordinates} are defined according to Definition \ref{def-2x2Solution}, but with $\We$ in place of $\Ue$.

\begin{remark}\label{rmk-EquivalenceOfEntropySolutions}
Suppose that $\Une \in \big[L^{\infty}(\bbR \times \Gamma) \big]^2$, and define $\Wne := A \Une$. Then $\We$ is a weak entropy solution to \eqref{eq-PSystemSemiLinear-RiemannianCoordinates} with initial data $\Wne$ if and only if $\Ue$ is a weak entropy solution to \eqref{eq-PSystemSemilinearFull} with initial data $\Une$. This can be checked using a direct computation. In particular, note that if $\eta(W)$ is a convex entropy for \eqref{eq-PSystemSemiLinear-RiemannianCoordinates}, then $\tilde{\eta}(U) := \eta(AU)$ is a convex entropy for \eqref{eq-PSystemSemilinearFull}. Conversely, if $\eta(U)$ is a convex entropy for \eqref{eq-PSystemSemilinearFull}, then $\tilde{\eta}(W) := \eta(A^{-1}W)$ is a convex entropy for \eqref{eq-PSystemSemiLinear-RiemannianCoordinates}.
\end{remark}

\subsection{Verification of \texorpdfstring{(H2)$'$}{(H2)'}}
We prove that the $p$-system verifies (H2)$'$ in two steps. First, in Theorem \ref{thm-GlobalWellPosednessForViscosityApproximation} we prove global existence and uniqueness of classical solutions to the approximate parabolic system \eqref{eq-PSystemSemilinearFull-ViscosityApproximation}. Second, in Theorem \ref{thm-H2-Step2} we prove that there exists a unique weak entropy solution to the $p$-system \eqref{eq-PSystemSemilinearFull} obtained in the sense described in (H2)$'$.

\begin{theorem}\label{thm-GlobalWellPosednessForViscosityApproximation}
Let $\veps$, $\nu > 0$. Let $\Unen$, $\dx \Unen$, $\dxx \Unen$ all belong to $\big[ L^2(\bbR \times \Gamma) \big]^2 \cap \big[ L^{\infty}(\bbR \times \Gamma) \big]^2$, and assume that
\begin{equation}
    \Vnorm{\Unen} + \Vnorm{\dx \Unen} + \Vnorm{\dxx \Unen} \leq C
\end{equation}
uniformly in $\veps$ and $\nu$, where $\Vnorm{\cdot} = \Vnorm{\cdot}_{\infty} + \Vnorm{\cdot}_{2;\rho}$.
Then there exists a unique global classical solution $\Uen$ of \eqref{eq-PSystemSemilinearFull-ViscosityApproximation} on $\bbR \times [0,\infty) \times \Gamma$ satisfying
\begin{equation}
\begin{split}
    \Uen\,, \, \dt \Uen\,, \, \dx \Uen\,, \, \dxx \Uen \, \in L^{\infty} \left( (0,\infty); \, \big[ L^2_{\rho}(\bbR \times \Gamma) \big]^2 \right) \cap \big[ L^{\infty}(\bbR \times (0,\infty) \times \Gamma) \big]^2\,.
\end{split}
\end{equation}
In addition, the solutions $\Uen$ satisfy
\begin{equation}\label{eq-PSystemSemiLinear-LInftyBoundsForSoln}
    \Vnorm{\Uen}_{\infty} \leq \Vnorm{A^{-1}} \bigg( \beta + \max \left\lbrace \left| f \left( -\frac{\beta}{a} \right) \right|\,, \, \left| f \left( \frac{\beta}{a} \right) \right| \right\rbrace \bigg)
\end{equation}
uniformly in $\veps$ and $\nu$, where $\beta := \sup_{\veps,\nu} \Vnorm{A \Unen}_{\infty}$.
\end{theorem}
\begin{proof}
It is clear that $\Uen$ is a classical solution of \eqref{eq-PSystemSemilinearFull-ViscosityApproximation} with initial data $\Unen$ if and only if $\Wen$ is a classical solution of \eqref{eq-PSystemSemiLinear-ViscosityApproximation-RiemannianCoordinates} with initial data $\Wnen := A \Unen$. Thus, the theorem is proved using the coordinate change $A$ in conjunction with the results of Proposition \ref{prop-LocalExistenceOfClassicalViscApproxSoln-RiemannianCoordinates}, Corollary \ref{cor-PSystemSemiLinear-UniquenessParabolicEquation-RiemannianCoordinates}, and Theorem \ref{thm-PSystemSemiLinear-GlobalExistenceOfViscositySolutions}.
\end{proof}

\begin{theorem}\label{thm-H2-Step2}
Let $\Une \in  \big[ L^2_{\rho}(\bbR \times \Gamma) \big]^2 \cap \big[ L^{\infty}(\bbR \times \Gamma) \big]^2$. Then there exists a unique weak entropy solution $\Ue$ to the system \eqref{eq-PSystemSemilinearFull} with initial data $\Une$. The function $\Ue$ satisfies
\begin{equation}
    \Ue \in C \left( [0,\infty); \, \big[ L^1_{\rho,loc}(\bbR \times \Gamma) \big]^2 \right)
    \cap \big[ L^{\infty}(\bbR \times (0,\infty) \times \Gamma) \big]^2,
\end{equation}
satisfies the $L^{\infty}$ bounds \eqref{eq-PSystemSemiLinear-LInftyBoundsForSoln},
and is obtained as the limit of a sequence of functions described in (H2)$'$.
\end{theorem}

\begin{proof}
Define $U^{\veps,j}_0$ and $U^{\veps,\nu,j}_0$ as in (H2)$'$. By Remark \ref{rmk-EquivalenceOfEntropySolutions}, Theorem \ref{thm-PSystemSemiLinear-L1LocStabilityEstimate} and the invertibility of $A$, there exists at most one weak entropy solution to \eqref{eq-PSystemSemilinearFull}.
Theorems \ref{thm-PSystemSemiLinear-L1ModulusOfContinuityInSpace} and \ref{thm-PSystemSemiLinear-L1ModulusOfContinuityInTime} reveal that the sequence of solutions $U^{\veps,\nu,j}$ to \eqref{eq-PSystemSemilinearFull-ViscosityApproximation} with initial data $U^{\veps,\nu,j}_0$ is equicontinuous in $\nu$ with respect to the norm $C \left( [0,\infty); \, \big[ L^1_{\rho,loc}(\bbR \times \Gamma) \big]^2 \right)$. The Frechet-Kolmogorov Theorem, Theorem \ref{thm-PSystemSemiLinear-AlmostEverywhereConvergenceToEntropySolution}, Remark \ref{rmk-EquivalenceOfEntropySolutions}, and properties of $A$ reveal that the limit of any subsequence in $\nu$ is a weak entopy solution of \eqref{eq-PSystemSemilinearFull}. Since weak entropy solutions of \eqref{eq-PSystemSemilinearFull} are unique, the entire sequence converges to the same limit, defined as $U^{\veps,j}$.
Finally, by reviewing the proof of Theorem $\ref{thm-PSystemSemiLinear-WellPosednessInL1ForDiagonalSystem}$ and using $A$, we see that we can define the weak entropy solution $\Ue$ of \eqref{eq-PSystemSemilinearFull} with initial data $\Une$ as the limit described in (H2)$'$, and the proof is complete.
\end{proof}

\subsection{Verification of (H3)}

Just as with general $2 \times 2$ systems, any convex function $\ell : \bbR \to \bbR$ can be extended to convex entropy for the \textit{p}-system via the application of Theorem \ref{thm-PartialConverseForEntropy}. For general systems, however, it may happen that the sets $D_{\ell}$ ``degenerate" in the following sense: Take an infinite family of convex functions on $\pi_1(D)$ $(\ell_{\alpha})$. Then it is possible that the sets $D_{\ell_{\alpha}}$ are contained in sets of the type $B_{\gamma}$ for arbitrarily small $\gamma$. Therefore, no \textit{a priori} $L^{\infty}$ bound for the solutions can be set. We do not encounter this difficulty in the case of the $p$-system; the key is that the constructed entropy are valid over all of $\bbR^2$ rather than over some bounded set containing the equilibria curve $K$.
We prove this now; the first lemma states the validity of the extension over all of $\bbR^2$, and the second lemma demonstrates that the function $u$ verifies infinitely many entropy inequalities.

\begin{lemma}[Entropy Extension for the $p$-system]\label{lma-PSystemSemiLinear-EntropyExtension}
Let $\ell : \bbR \to \bbR $ be a convex (strictly convex, strongly convex) function that is of class $C^2(\bbR)$, and set
\begin{equation}
q(u) := \intdmt{0}{u}{\ell'(\theta) f'(\theta)}{\theta}\,.
\end{equation}
Then there exist functions $\eta : \bbR^2 \to \bbR$ and $Q : \bbR^2 \to \bbR$ such that $(\eta,Q)$ is a convex (strictly convex, strongly convex) entropy pair for the semi-linear \textit{p}-system \eqref{eq-PSystemSemilinearFull} over all of $\bbR^2$, satisfying
\begin{equation}
    \eta(u,f(u)) = \ell(u)\,, \qquad Q(u,f(u)) = q(u)\,, \qquad u \in \bbR\,.
\end{equation}
In addition, $\eta$ and $Q$ satisfy
\begin{equation}\label{eq-PSystemSemiLinear-MoreEntropyConditions}
\dv \eta(u,f(u)) = 0\,, \qquad 
 \exists \, \gamma > 0 \text{ such that } \, \, \dv \eta(U) \, (v - f(u)) \geq \gamma|v-f(u)|^2\,.
\end{equation}
\end{lemma}

\begin{proof}
Entropy pairs $(\eta,Q)$ for the full system \eqref{eq-PSystemSemilinearFull} satisfy the linear hyperbolic system
\begin{equation}\label{eq-PSystemSemilinear-EntropyCauchyProblem}
\begin{cases}
    \du Q(U) - a^2 \dv \eta(U) = 0 \\
    \dv Q(U) - \du \eta(U) = 0
\end{cases}
\end{equation}
with general solution, c.f. \cite{Dafermos}
\begin{equation}
\begin{cases}
\eta(u,v) = h(au+v) + k(au-v) \\
Q(u,v) = ah(au+v) - ak(au-v)\,.
\end{cases}
\end{equation}
The subcharacteristic condition \eqref{eq-PSystemSemilinear-SubcharacteristicCondition} implies that the equilibrium curve $K$ is non-characteristic for the system.
Thus, given any entropy pair $(\ell,q)$ for the scalar conservation law, one may obtain extended entropy $(\eta,Q)$ just as in Theorem \ref{thm-PartialConverseForEntropy}, but satisfying the conditions of Definition \ref{def-EntropyForFullSystem} and \eqref{eq-PSystemSemiLinear-MoreEntropyConditions} on all of $\bbR^2$.
\end{proof}

\begin{corollary}[(H3) is verified]
Suppose $(\Ue)_{\veps}$ is a sequence of entropy solutions of \eqref{eq-PSystemSemilinearFull} taking values in $D \subset \bbR^2$. Suppose further that there exists a function $U = (u,v) \in \big[ L^{\infty}(\bbR \times (0,T) \times \Gamma) \big]^2$ such that $\Ue \to U$ almost everywhere. Let $(\ell,q)$ be a convex entropy pair for \eqref{eq-ReducedSystemRandom} over $\pi_1(D)$. Then $u$ satisfies
\begin{equation}\label{eq-PSystemSemiLinear-Proof-EntropyAdmissibility}
\iiintdmt{0}{\infty}{\bbR}{\Gamma}{\Big( \ell(u) \dt \vphi + q(u) \dx \vphi  \Big) \, \rho}{y}{x}{t} \geq 0
\end{equation}
for every $\vphi \in C^{\infty}_c(\bbR \times (0,\infty) \times \Gamma)$ with $\vphi \geq 0$.
\end{corollary}
\begin{proof}
By Lemma \ref{lma-PSystemSemiLinear-EntropyExtension} there exists a convex entropy $(\eta,Q)$ for the full system \eqref{eq-PSystemSemilinearFull} convex on all of $\bbR^2$ and hence convex on all of $D$. Then the Dominated Convergence Theorem applied to the entropy inequality \eqref{eq-DefinitionOfEntropySolution-EntropyInequality} for $\Ue$ with said entropy implies that \eqref{eq-PSystemSemiLinear-Proof-EntropyAdmissibility} is satisfied.
\end{proof}

\subsection{Main Theorem}

We summarize the results of this section in one theorem:

\begin{proposition}[Convergence for the semi-linear \textit{p}-system]
Assume the subcharacteristic condition \eqref{eq-PSystemSemilinear-SubcharacteristicCondition}. Assume that the initial data $(\Ue)_{\veps}$ for the \textit{p}-system \eqref{eq-PSystemSemilinearFull} is bounded in $\left[L^{\infty}(\bbR \times \Gamma) \right]^2 \cap \left[ L^2_{\rho}(\bbR \times \Gamma) \right]^2$, uniformly in $\veps$.
Assume also that $(\une,\vne)$ converge almost everywhere to a function $(u_0,v_0) \in \left[ L^{\infty}(\bbR \times \Gamma) \right]^2$ satisfying $v_0 = f(u_0)$. Assume that $f''(u) \neq 0$ almost everywhere.
Then for each $\veps$ there exists a unique global weak entropy solution $\Ue = (\ue,\ve)$ of the \textit{p}-system \eqref{eq-PSystemSemilinearFull} belonging to the class
\begin{equation}
C \left( (0,\infty); \big[ L^1_{\rho, loc}(\bbR \times \Gamma) \big]^2 \right) \cap L^{\infty}\left( (0,\infty); \left[L^2_{\rho}(\bbR \times \Gamma) \right]^2 \right) \cap \big[ L^{\infty}(\bbR \times (0,\infty) \times \Gamma) \big]^2
\end{equation}
and satisfying the bound
\begin{equation}
    \Vnorm{\Ue}_{\infty} \leq \Vnorm{A^{-1}} \bigg( \beta + \max \left\lbrace \left| f \left( -\frac{\beta}{a} \right) \right|\,, \, \left| f \left( -\frac{\beta}{a} \right) \right| \right\rbrace \bigg)\,, \qquad \beta := \sup_{\veps} \Vnorm{A \Une}_{\infty}\,,
\end{equation}
uniformly in $\veps$. The sequence $(\Ue)_{\veps}$ converge as $\veps \to 0$ in $\left[ L^p_{loc}(\bbR \times (0,\infty) \times \Gamma) \right]^2$ to a function \linebreak $U = (u,v) \in \big[ L^{\infty}(\bbR \times (0,\infty) \times \Gamma) \big] \cap C \left( (0,\infty); \big[ L^{1}_{\rho,loc}(\bbR \times \Gamma) \big]^2 \right)$ that satisfies
\begin{enumerate}
    \item[i)] $v = f(u)$ a.e.
    \item[ii)] $u$ is the unique weak entropy solution to the scalar conservation law \eqref{eq-ReducedSystemRandom} with initial data $u_0$.
\end{enumerate}
\end{proposition}
\appendix

\section{Quasilinear Parabolic Systems}\label{apdx-QLParabolicSystems}
This section is devoted to proving the relevant results for a general $2 \times 2$ system of quasilinear parabolic equations. These results will be used to show the global well-posedness of the parabolic systems \eqref{eq-PSystemSemiLinear-ViscosityApproximation-RiemannianCoordinates} and \eqref{eq-PSystemSemilinearFull-ViscosityApproximation} in Appendix \ref{apdx-SLParabolicSystems} and Section \ref{sec-SemiLinearPSystem} respectively. The system we are concerned with in this section is
\begin{equation}\label{eq-QuasilinearParabolicSystem}
\begin{cases}
    &\dt u + \dx F_1(U) + \cQ_1(U)= \dxx u  \quad \text{ on } \bbR \times (0,T) \times \Gamma\,, \\
    &\dt v + \dx F_2(U) + \cQ_2(U) = \dxx v \quad \text{ on } \bbR \times (0,T) \times \Gamma\,, \\
    &U(x,0,y) = U_0(x,y)\,, \quad \text{ on } \bbR \times \Gamma\,.
\end{cases}
\end{equation}

We assume that $F_1$, $F_2$, $\cQ_1$ and $\cQ_2$ are all $C^2$, and we define $F := (F_1,F_2)$ and $\cQ:= (\cQ_1,\cQ_2)$. Throughout the section we assume that $U_0 \in \left[ L^{\infty}(\bbR \times \Gamma) \right]^2$. A classical solution $U$ of will satisfy the integral equation given by the fundamental solution of the heat equation

\begin{equation}\label{eq-PSystemSemiLinear-MildSolutionToHeatEquation}
\begin{split}
    U(x,t,y) &= \frac{1}{(4 \pi t)^{1/2}}\intdm{\bbR}{\e^{\frac{-|x-z|^2}{4t}}U_0(z,y)}{z} + \frac{1}{(4 \pi)^{1/2}}\intdmt{0}{t}{\frac{1}{(t-s)^{1/2}} \intdm{\bbR}{\p_z \left[ \e^{\frac{-|x-z|^2}{4(t-s)}} \right] F(U(z,s,y))}{z}}{s} \\
    &\quad - \frac{1}{(4\pi)^{1/2}}\intdmt{0}{t}{\frac{1}{(t-s)^{1/2}}\intdm{\bbR}{\e^{\frac{-|x-z|^2}{4(t-s)}}\cQ(U(z,s,y))}{z}}{s} \\
    &= \cJ(\cdot,t) \, \ast \, U_0(\cdot,y) + \intdmt{0}{t}{\dx \cJ(\cdot,t-s) \, \ast \, F(U(\cdot,s,y))}{s} - \intdmt{0}{t}{\cJ(\cdot,t-s) \, \ast \, \cQ(U(\cdot,s,y))}{s}\,,
\end{split}
\end{equation}
where
\begin{equation}
    \cJ(x,t) := \frac{1}{(4 \pi t)^{1/2}} \e^{-\frac{|x|^2}{4t}}\,;
\end{equation}
see Theorem \ref{thm-QuasilinearParabolicSystem-MildSolutionIsClassicalSolution} below for a proof. Note that if we write $k(x) = \frac{1}{\sqrt{4\pi}} \e^{-\frac{|x|^2}{4}} \in \cS(\bbR)$ then $\cJ(x,t) = t^{-1/2} k \left( \frac{x}{\sqrt{t}} \right)$ (Here, $\cS(\bbR)$ denotes the space of Schwartz functions on $\bbR$.)
To find of solutions to \eqref{eq-PSystemSemiLinear-MildSolutionToHeatEquation} we prove that the operator
\begin{equation}
\begin{split}
    &U \mapsto LU\,, \\
    &(LU)(t,y) := \cJ(\cdot,t) \, \ast \, U_0(\cdot,y) + \intdmt{0}{t}{\dx \cJ(\cdot,t-s) \, \ast \, F(U(\cdot,s,y))}{s} - \intdmt{0}{t}{\cJ(\cdot,t-s) \, \ast \, \cQ(U(\cdot,s,y))}{s}
\end{split}
\end{equation}
is a contraction on a suitable Banach space. The Banach Fixed Point Theorem gives us the existence of a solution to \eqref{eq-PSystemSemiLinear-MildSolutionToHeatEquation}.

The main result of this section is the following theorem.
\begin{theorem}[Local Existence of Solutions]\label{thm-QLParabolicSystemsMainThm}
Suppose that $K := \{ U \in \bbR^2 | \cQ(U) = 0 \} \neq \emptyset$. Suppose there exists $\Ubar \in \left[ L^{\infty}(\Gamma) \right]^2$ taking values in the set $K$ such that $U_0(x,y) - \Ubar(y) \in \left[ L^2_{\rho}(\bbR \times \Gamma) \right]^2$. 
Suppose also that
\begin{equation}
\begin{split}
\dx U_0, \, \dxx U_0 \in \left[ L^2_{\rho}(\bbR \times \Gamma) \right]^2 \cap \big[ L^{\infty}(\bbR \times \Gamma) \big]^2\,.
\end{split}
\end{equation}
Then there exists a time $T>0$ such that the following holds: There exists a function $U \in \big[ L^{\infty}(\bbR \times (0,T) \times \Gamma) \big]^2$ solving \eqref{eq-PSystemSemiLinear-MildSolutionToHeatEquation} satisfying
\begin{equation}
\begin{split}
    U - \Ubar\,, \dt U \,, \dx U \,, \dxx U \in L^{\infty}\left( (0,T); \left[L^2_{\rho}(\bbR \times \Gamma) \right]^2 \right) \cap \big[ L^{\infty}(\bbR \times (0,T) \times \Gamma) \big]^2\,,
\end{split}
\end{equation}
with
\begin{equation}
    \Vnorm{U- \Ubar} + \Vnorm{\dt U} + \Vnorm{\dx U} \leq C \Big( \Vnorm{U_0-\Ubar} + \Vnorm{\dx U_0} + \Vnorm{\dxx U_0} \Big)
\end{equation}
for some constant $C$ independent of $U$, where $\Vnorm{ \cdot } = \Vnorm{\cdot}_{\infty} + \Vnorm{\cdot}_{\infty, 2; \rho}$.
Moreover, this solution $U$ to the integral equation \eqref{eq-PSystemSemiLinear-MildSolutionToHeatEquation} is in fact a local classical solution to \eqref{eq-QuasilinearParabolicSystem}.
\end{theorem}

We prove this theorem in three lemmas, stated and proved in the rest of this section.

\begin{lemma}\label{thm-ExistenceOfViscositySolutions}
Suppose that the term $\cQ$ vanishes on at least one point. Let $b, b_0$ be constants such that $b > b_0 > 0$. Suppose there exists $\Ubar \in \left[ L^{\infty}(\Gamma) \right]^2$ taking values in the set $K := \{ U \in \bbR^2 \, \big| \, \cQ(U) = 0 \}$ such that $U_0(x,y) - \Ubar(y) \in \left[ L^2_{\rho}(\bbR \times \Gamma) \right]^2$, with $\Vnorm{U_0-\Ubar}_{\infty} < b_0$. For $T > 0$, define 
\begin{equation}
G_T^{\infty} := \left\lbrace U \in \left[ L^{\infty}(\bbR \times (0,T) \times \Gamma) \right]^2 \, \Big| \, \Vnorm{U - \Ubar}_{\infty} \leq b \right\rbrace
\end{equation}
and
\begin{equation}
G_T^2 := \left\lbrace U : \bbR \times (0,T) \times \Gamma \to \bbR^2 \, \Big| \, U - \Ubar \in L^{\infty} \left( (0,T); \left[ L^2_{\rho}(\bbR \times \Gamma) \right]^2 \right)   \right\rbrace\,.
\end{equation}

Then there exists a time $T > 0$ such that the following holds: There exists a function $U \in G_T^{\infty} \cap G_T^2$ solving \eqref{eq-PSystemSemiLinear-MildSolutionToHeatEquation} such that
\begin{equation}\label{eq-PSystemSemiLinear-EstimatesOfMildSolution}
\begin{split} 
    \Vnorm{U - \Ubar}_{\infty} &\leq C \Vnorm{U_0 - \Ubar}_{\infty}\,, \\
    \Vnorm{U - \Ubar}_{\infty, 2;\rho} &\leq C \Vnorm{U_0 - \Ubar}_{2;\rho}\,. \\
\end{split}
\end{equation}
The constant $C$ depends only on $b_0$, $b$ and $T$.
\end{lemma}

\begin{proof}
\underline{Step 1:} We show that there exists a time $T > 0$ such that $L$ is a contraction mapping from the complete metric space $G_T^{\infty}$ into itself.

First, we show that we can choose $T$ such that $L(G_T^{\infty}) \subset G_{T}^{\infty}$. Using the fact that $\cQ(\Ubar) = 0$,  $\intdm{\bbR}{k(x)}{x}=1$, and that $\intdm{\bbR}{\dx \cJ(x,t)}{x} = 0$ for every $t \in [0,T]$,
\begin{equation} 
\begin{split} 
    LU-\Ubar &= \cJ(\cdot,t) \, \ast \, (U_0(\cdot,y) - \Ubar(y)) + \intdmt{0}{t}{\dx \cJ(\cdot,t-s) \, \ast \, \Big( F(U(\cdot,s,y))-F(\Ubar(y)) \Big)}{s} \\
    &\quad - \intdmt{0}{t}{\cJ(\cdot,t-s) \, \ast \, \Big( \cQ(U(\cdot,s,y))-\cQ(\Ubar(y)) \Big)}{s}\,.
\end{split}
\end{equation}
Since $F$ and $\cQ$ are differentiable, define $\alpha_1(b)$ and $\alpha_2(b)$ to be their respective Lipschitz constants on $\bigcup_{y \in \Gamma} B(\Ubar(y),b)$. Then, since $\Vnorm{k}_1 = 1$,
\begin{equation}
    \Vnorm{LU(t)-\Ubar}_{\infty} \leq \Vnorm{U_0 - \Ubar}_{\infty} + \alpha_1(b) \Vnorm{U - \Ubar}_{\infty} \intdmt{0}{t}{\Vnorm{\dx \cJ(t-s)}_1}{s} + \alpha_2(b) T \Vnorm{U - \Ubar}_{\infty}\,.
\end{equation}
Since $\intdm{\bbR}{\left| \dx \cJ(x,t) \right|}{x} = \intdm{\bbR}{t^{-1} \left| k'(t^{-1/2}x) \right|}{x} = t^{-1/2} \intdm{\bbR}{\left| k'(x) \right|}{x} := \mu t^{-1/2}$, we have
\begin{equation}
\begin{split}
    \Vnorm{LU(t)-\Ubar}_{\infty} & \leq b_0 + \mu b \alpha_1(b) \intdmt{0}{t}{\frac{1}{\sqrt{t-s}}}{s}+ b \alpha_2(b) T \leq b_0 + 2 \mu b \alpha_1(b) \sqrt{T} + b \alpha_2(b) T\,.
\end{split}
\end{equation}
Choose $T$ small enough so that,
\begin{equation}
b_0 + 2 \mu b \alpha_1(b) \sqrt{T} + b \alpha_2(b) T \leq b\,.
\end{equation}
Then $LU \in G_T^{\infty}$.
Now we prove the property that the mapping is a contraction. For $U$, $V \in G^{\infty}_T$,
\begin{equation}\label{eq-PSystemSemiLinear-InftyContraction}
\begin{split} 
\Vnorm{(LU(t)-\Ubar) - (LV(t)-\Ubar)}_{\infty} &= \Vnorm{LU(t) - LV(t)}_{\infty} \\
    &\leq \alpha_1(b) \intdmt{0}{t}{\Vnorm{\dx \cJ(t-s)}_1 \Vnorm{V(s)-U(s)}_{\infty}}{s} \\
    &\quad + \alpha_2(b) \intdmt{0}{t}{\Vnorm{V(s)-U(s)}_{\infty}}{s} \\
    &\leq \left( 2 \mu \alpha_1(b) \sqrt{T} + \alpha_2(b) T \right) \Vnorm{V-U}_{\infty} \\
    &= \left( 2 \mu \alpha_1(b) \sqrt{T} + \alpha_2(b) T \right) \Vnorm{(V-\Ubar)-(U-\Ubar)}_{\infty} \,.
\end{split}
\end{equation}
By our choice of $T$, $\left( 2 \mu \alpha_1(b) \sqrt{T} + \alpha_2(b) T \right) < 1$. Thus the mapping $L$ is a contraction on $G_T^{\infty}$.

Step 1 is complete. The Banach fixed-point theorem gives us the existence of a unique fixed point of the mapping $L : G_T^{\infty} \to G_T^{\infty}$. We call the fixed point $U$.

\underline{Step 2:} We show that in fact the fixed point $U$ belongs to $G_T^2$.
We proceed just as in \eqref{eq-PSystemSemiLinear-InftyContraction} with the norm $\Vnorm{\cdot}_{\infty 2;\rho}$ in place of $\Vnorm{\cdot}_{\infty}$ to get
\begin{equation}
    \Vnorm{(LU-\Ubar)-(LV-\Ubar)}_{\infty, 2; \rho} \leq \left( 2 \mu \alpha_1(b) \sqrt{T} + \alpha_2(b) T \right) \Vnorm{(V-\Ubar)-(U-\Ubar)}_{\infty, 2; \rho}\,.
\end{equation}
Thus the mapping $L$ is a contraction for the norm $\Vnorm{\cdot}_{\infty, 2; \rho}$ as well. Therefore, the sequence $\{ U^m - \Ubar \}_{m \geq 0}$, where $U^m$ is defined by
\begin{equation}
    U^{m+1} = LU^m \,, \quad U^0 = \Ubar\,,
\end{equation}
is Cauchy in $\Vnorm{\cdot}_{\infty, 2 ; \rho}$ and thus converges in said norm to a limit function $\tilde{U}$. Then there exists a subsequence which converges almost everywhere. However, it has already been established that the sequence converges in the metric space $G^{\infty}_T$ to the fixed point $U \in G_T^{\infty}$. Since almost-everywhere limits are unique, $\tilde{U} = U$, and so the fixed point $U$ is simultaneously in both $G^{\infty}_T$ and $G_T^2$.

\underline{Step 3:} We prove the estimates \eqref{eq-PSystemSemiLinear-EstimatesOfMildSolution}. 
Since the mapping $L$ is a contraction on both $G^{\infty}_T$ and $G^2_T$, applying the contractive inequalities successively on the sequence of iterates leads to \eqref{eq-PSystemSemiLinear-EstimatesOfMildSolution}, where the constant $C$ has the value $\frac{1}{1-2 \mu \alpha_1(b) \sqrt{T} - \alpha_2(b) T}$.
\end{proof}

We next obtain bounds on derivatives in order to see that the solution to the integral equation \eqref{eq-PSystemSemiLinear-MildSolutionToHeatEquation} is in fact a classical solution of \eqref{eq-QuasilinearParabolicSystem} for data sufficiently smooth.

\begin{lemma}[Estimates on Derivatives]\label{thm-PSystemSemiLinear-EstimatesOnDerivatives}
Let $b$, $b_0$ be constants such that $b > b_0 > 0$. Suppose there exists $\Ubar \in \left[ L^{\infty}(\Gamma) \right]^2$ taking values in $K$ such that $U_0(x,y) - \Ubar(y) \in \left[ L^2_{\rho}(\bbR \times \Gamma) \right]^2$, with $\Vnorm{U_0-\Ubar}_{\infty} < b_0$.
Let $T>0$ be such that the system \eqref{eq-PSystemSemiLinear-MildSolutionToHeatEquation} has a solution $U \in G_T^{\infty} \cap G_T^2$. 

Suppose in addition that
\begin{equation}
\begin{split}
\dx U_0, \, \dxx U_0 \in \left[ L^2_{\rho}(\bbR \times \Gamma) \right]^2 \cap \big[ L^{\infty}(\bbR \times \Gamma) \big]^2\,.
\end{split}
\end{equation}
Then there exist $T_0 \in (0,T]$ and $C > 1$ such that the solution of \eqref{eq-PSystemSemiLinear-MildSolutionToHeatEquation} satisfies
\begin{equation}
\begin{split}
    U - \Ubar\,, \dt U \,, \dx U \,, \dxx U \in L^{\infty}\left( (0,T_0); \left[L^2_{\rho}(\bbR \times \Gamma) \right]^2 \right) \cap \big[ L^{\infty}(\bbR \times (0,T_0) \times \Gamma) \big]^2
\end{split}
\end{equation}
with
\begin{equation} 
\begin{split}
    \Vnorm{\dt U} &\leq C \Big( \Vnorm{U_0-\Ubar} + \Vnorm{\dx U_0} + \Vnorm{\dxx U_0} \Big) \\
    \Vnorm{\dx U} &\leq C \Big( \Vnorm{U_0-\Ubar} + \Vnorm{\dx U_0} \Big) \\
    \Vnorm{\dxx U} &\leq P \Big( \Vnorm{U_0-\Ubar} + \Vnorm{\dx U_0} + \Vnorm{\dxx U_0} \Big)\,, \\
\end{split}
\end{equation}
Where $\Vnorm{ \cdot } = \Vnorm{\cdot}_{\infty} + \Vnorm{\cdot}_{\infty, 2; \rho}$. Here, $P$ is a polynomial with coefficients depending on $b$ and the $C^2$ norm of $F$ in $\bigcup_{y \in \Gamma} B(\Ubar(y),b)$.
\end{lemma}

\begin{proof}
We follow the argument in \cite{Serre}.
It is sufficient to show that if a given $V \in G_T^{\infty} \cap G_T^2$ satisfies these inequalities, then $LV$ satisfies them. Let $V \in G_{T_0}^{\infty} \cap G_{T_0}^2$ with $T_0$ to be determined. Then
\begin{equation} 
\begin{split} 
    \dt LV(x,t,y) &= \intdm{\bbR}{\dt \cJ(x-z,t) U_0(z,y)}{z} \\
    &\quad + \iintdmt{0}{t}{\bbR}{\p_z \cJ(x-z,s) \dt [F(V(z,t-s,y))]}{z}{s} + \intdm{\bbR}{\p_z \cJ(x-z,t) F(V(z,0,y))}{z} \\
    &\quad - \iintdmt{0}{t}{\bbR}{\cJ(x-z,s) \dt [\cQ(V(z,t-s,y))]}{z}{s} - \intdm{\bbR}{\cJ(x-z,t) \cQ(V(z,0,y))}{z}\,,
\end{split}
\end{equation}
where $V(x,0,y) = \lim\limits_{t \to 0^+} V(x,t,y)$. But
\begin{equation}
\Vnorm{\dt F(V(t-s))}_{2;\rho} \leq \alpha_1(b) \Vnorm{\dt V(t-s)}_{2;\rho}\,, \qquad \Vnorm{\dt \cQ(V(t-s))}_{2;\rho} \leq \alpha_2(b) \Vnorm{\dt V(t-s)}_{2;\rho}\,,
\end{equation}
Also, since $\cJ$ solves the heat equation, subtracting $R(\Ubar)$ and using the assumptions on $V$ gives us
\begin{equation} 
\begin{split}
    \Vnorm{\dt LV(t)}_{2;\rho} &\leq \Vnorm{\intdm{\bbR}{\cJ(z,t) \dxx U_0(x-z,y)}{z}}_{2;\rho} \\
    &\quad + \mu \alpha_1(b) \intdmt{0}{t}{ \frac{\Vnorm{\dt V(t-s)}_{2;\rho}}{\sqrt{s}} }{s} + \Vnorm{\intdm{\bbR}{\cJ(x-z,t) \p_z F(V(z,0,y))}{z}}_{2;\rho} \\
    &\quad + \alpha_2(b) \intdmt{0}{t}{\Vnorm{\dt V(t-s)}_{2;\rho}}{s} + \alpha_2(b) \Vnorm{U_0-\Ubar}_{2;\rho} \\
    &\leq \Vnorm{\dxx U_0}_{2;\rho} + \mu \alpha_1(b) \Vnorm{\dt V}_{\infty 2;\rho} \big( 2 \sqrt{t} \big) + \alpha_1(b) \Vnorm{\dx U_0}_{2;\rho} \\
    &\quad + \alpha_2(b) \Vnorm{\dt V}_{\infty 2;\rho} t + \alpha_2(b) \Vnorm{U_0-\Ubar}_{2;\rho}\,.
\end{split}
\end{equation}
Choose $T_0 > 0$ such that 
\begin{equation}
n := 2 \mu \alpha_1(b) \sqrt{T_0} + \alpha_2(b) T_0 < 1\,.
\end{equation}
(Note that the time $T$ chosen in Theorem \ref{thm-ExistenceOfViscositySolutions} would work here.) Then
\begin{equation}
\Vnorm{\dt LV(t)}_{2;\rho} \leq C \left( \Vnorm{U_0-\Ubar}_{2;\rho} + \Vnorm{\dx U_0}_{2;\rho} + \Vnorm{\dxx U_0}_{2;\rho} \right) + n \Vnorm{\dt V}_{\infty 2;\rho}\,.
\end{equation}
Thus, the sequence of iterates $U^m$ satisfy
\begin{equation}
\Vnorm{\dt U^m(t)}_{2;\rho} \leq  \frac{C}{1-n} \left( \Vnorm{U_0-\Ubar}_{2;\rho} + \Vnorm{\dx U_0}_{2;\rho} + \Vnorm{\dxx U_0}_{2;\rho} \right)\,,
\end{equation}
by repeated use of the above argument, since $U^0 \equiv \Ubar$.

Finally, for each $m \in \bbN$,
\begin{equation}
\begin{split}
\Vnorm{\dt LU^m(t)}_{\infty} &\leq \Vnorm{\dxx U_0}_{\infty} + \mu \alpha_1(b) \intdmt{0}{t}{ \frac{\Vnorm{\dt U^m(t-s)}_{\infty}}{\sqrt{s}} }{s} + \alpha_1(b) \Vnorm{\dx U_0}_{\infty} \\
    &\quad + \alpha_2(b) \intdmt{0}{t}{\Vnorm{\dt U^m(t-s)}_{\infty}}{s} + \alpha_2(b) \Vnorm{U_0-\Ubar}_{\infty}\,,
\end{split}
\end{equation}
which leads in a similar way to
\begin{equation}
\Vnorm{\dt LU^m(t)}_{\infty} \leq C \left( \Vnorm{U_0-\Ubar}_{\infty} + \Vnorm{\dx U_0}_{\infty} + \Vnorm{\dxx U_0}_{\infty} \right) + n \Vnorm{\dt U^m}_{\infty}
\end{equation}
and then
\begin{equation}
\Vnorm{\dt U^m(t)}_{\infty} \leq  \frac{C}{1-n} \left( \Vnorm{U_0-\Ubar}_{\infty} + \Vnorm{\dx U_0}_{\infty} + \Vnorm{\dxx U_0}_{\infty} \right)\,.
\end{equation}

To show the other two estimates, one differentiates the non-kernel terms in the integral equation and repeats an argument just like the one above. See \cite{Serre} for details. The estimates are simpler, since there are no boundary terms.
\end{proof}

With these derivative estimates, we have

\begin{lemma}\label{thm-QuasilinearParabolicSystem-MildSolutionIsClassicalSolution}
With all of the assumptions of Theorem \ref{thm-PSystemSemiLinear-EstimatesOnDerivatives}, there exists a local classical solution to \eqref{eq-QuasilinearParabolicSystem} given by the solution $U$ to the integral equation \eqref{eq-PSystemSemiLinear-MildSolutionToHeatEquation}.
\end{lemma}

The proof uses the arguments found in \cite[Section 2.3, Theorem 2]{Evans}. For a fixed $y \in \Gamma$, we show that the fundamental solution \eqref{eq-PSystemSemiLinear-MildSolutionToHeatEquation} solves
\begin{equation}
\dt U + \dx F(U) + \cQ(U) = \dxx U
\end{equation}
$(x,t)$-a.e., and that
\begin{equation}
\lim\limits_{(x,t) \to (x_0,0)} U(x,t,y) = U_0(x_0,y)
\end{equation}
for every $x_0 \in \bbR$.
These are true for every $y \in \Gamma$, so the equation is solved pointwise. The maximal existence time is bounded from below by the same $T_0$ obtained above.

\begin{proof}[Proof of Theorem]
Since $U$ is a classical solution, we can take derivatives of both $U$ and $U_0$. Then
\begin{equation}\label{eq-PSystemSemiLinear-MildSolutionIsClassicalSolution}
\begin{split}
    U(x,t,y) &= \frac{1}{(4 \pi t)^{1/2}} \intdm{\bbR}{\e^{\frac{-|x-z|^2}{4t}}U_0(z,y)}{z} \\
        &\qquad - \iintdmt{0}{t}{\bbR}{\frac{1}{(4 \pi(t-s))^{1/2}}\e^{\frac{-|x-z|^2}{4(t-s)}} \Big( \p_z F(U(z,s,y)) + \cQ(U(z,s,y)) \Big) }{z}{s}\,.
\end{split}
\end{equation}
Set $\cF(U(z,s,y)) := \p_z F(U(z,s,y)) + \cQ(U(z,s,y))$. It suffices to show that $\dt - \dxx$ applied to the right-hand side of \eqref{eq-PSystemSemiLinear-MildSolutionIsClassicalSolution} is equal to $- \cF(U(x,t,y))$.
Since the first term on the right-hand side of \eqref{eq-PSystemSemiLinear-MildSolutionIsClassicalSolution} is convolution with the heat kernel, we get that
\begin{equation}
\begin{split}
    (\dt - \dxx)(U(x,t,y)) &= -  (\dt - \dxx) \left[ \iintdmt{0}{t}{\bbR}{\frac{1}{(4 \pi s)^{1/2}}\e^{\frac{-|z|^2}{4s}} \cF(U(x-z,t-s,y))}{z}{s} \right] \\
    &= - \iintdmt{0}{t}{\bbR}{\frac{1}{(4 \pi s)^{1/2}}\e^{\frac{-|z|^2}{4s}} (\dt - \dxx) \big[ \cF(U(x-z,t-s,y)) \big]}{z}{s} \\
    &\qquad - \intdm{\bbR}{\frac{1}{(4 \pi t)^{1/2}} \e^{\frac{-|z|^2}{4t}} \cF(U(x-z,0,y))}{z}\,.
\end{split}
\end{equation}
Splitting the integral and changing the variable of differentiation, we get
\begin{equation}
\begin{split}
    (\dt - \dxx)(U(x,t,y)) &= - \iintdmt{\delta}{t}{\bbR}{\frac{1}{(4 \pi s)^{1/2}}\e^{\frac{-|z|^2}{4s}} (- \p_s - \p_{zz}) \big[ \cF(U(x-z,t-s,y)) \big]}{z}{s} \\
    &\qquad - \iintdmt{0}{\delta}{\bbR}{\frac{1}{(4 \pi s)^{1/2}}\e^{\frac{-|z|^2}{4s}} (- \p_s - \p_{zz}) \big[ \cF(U(x-z,t-s,y)) \big]}{z}{s} \\
    &\qquad - \intdm{\bbR}{\frac{1}{(4 \pi t)^{1/2}} \e^{\frac{-|z|^2}{4t}} \cF(U(x-z,0,y))}{z} \\
    &= \mathrm{I} + \mathrm{II} + \mathrm{III}\,.
\end{split}
\end{equation}
We handle each separately.
\begin{equation}
    | \mathrm{II} | \leq \Big( \Vnorm{\dt \cF(U)}_{L^{\infty}} + \Vnorm{\dxx \cF(U)}_{L^{\infty}} \Big) \iintdmt{0}{\delta}{\bbR}{\cJ(z,s)}{z}{s} \leq C \delta\,.
\end{equation}
Integrating by parts, we get
\begin{equation}
\begin{split}
    \mathrm{I} &= - \iintdmt{\delta}{t}{\bbR}{(\p_s - \p_{zz}) \big[ \cJ(z,s) \big] \cF(U(x-z,t-s,y))}{z}{s} \\
    &\qquad - \intdm{\bbR}{\cJ(z,\delta) \cF(U(x-z,t-\delta,y))}{z} + \intdm{\bbR}{\cJ(z,t) \cF(U(x-z,0,y))}{z} \\
    &= - \intdm{\bbR}{\cJ(z,\delta) \cF(U(x-z,t-\delta,y))}{z} - \mathrm{III}\,.
\end{split}
\end{equation}
since $\cJ$ solves the heat equation. Thus,
\begin{equation}
    (\dt - \dxx)(U(x,t,y)) = - \intdm{\bbR}{\cJ(z,\delta) \cJ(U(x-z,t-\delta,y))}{z} + O(\delta)\,.
\end{equation}
It therefore suffices to show that
\begin{equation}\label{eq-SemiLinearPSystem-MildSolutionSolvesEquation}
    \lim\limits_{\delta \to 0} \intdm{\bbR}{\cJ(z,\delta) \cF(U(x-z,t-\delta,y))}{z} = \cF(U(x,t,y))\,.
\end{equation}
Let $\alpha > 0$ be given. Choose $\beta > 0$ and $\delta > 0$ small such that
\begin{equation}
    |z| < \beta \Rightarrow |\cF(U(x-z,t-\delta,y)) - \cF(U(x,t,y))| < \alpha\,.
\end{equation}
Then since $\Vnorm{\cJ(t)}_1 = 1$, we have
\begin{equation}
\begin{split}
    \left| \intdm{\bbR}{\cJ(z,\delta) \cF(U(x-z,t-\delta,y))}{z} - \cF(U(x,t,y)) \right| &\leq  \intdm{\bbR}{\cJ(z,\delta) \big| \cF(U(x-z,t-\delta,y)) - \cF(U(x,t,y)) \big| }{z} \\
    &\leq \intdm{|z| < \beta}{\cdots}{z} + \intdm{|z| > \beta}{\cdots}{z} := \mathrm{IV} +  \mathrm{V}\,.
\end{split}
\end{equation}
Now,
\begin{equation}
    \mathrm{IV} \leq \alpha \Vnorm{\cJ(t)}_1 = \alpha\,,
\end{equation}
and
\begin{equation}
\begin{split}
    \mathrm{V} \leq 2 \Vnorm{\cF(U)}_{L^{\infty}} \intdm{|z| \geq \beta}{\cJ(z,\delta)}{z} = \intdm{|z| \geq \beta / \sqrt{\delta}}{\e^{\frac{-|z|^2}{4}}}{z}\,.
\end{split}
\end{equation}
Therefore, for $\delta > 0$ sufficiently small we have that
\begin{equation}
    \mathrm{IV} + \mathrm{V} \leq 2 \alpha\,.
\end{equation}
This shows that the mild solution solves the equation. To see convergence to the initial data, note that (see \cite[Section 2.3, Theorem 1]{Evans})
\begin{equation}
    \lim\limits_{(x,t) \to (x_0,0)} \intdm{\bbR}{\cJ(x-z,t)U_0(z,y)}{z} = U_0(x,y)
\end{equation}
and that
\begin{equation}
\begin{split}
    \left| \iintdmt{0}{t}{\bbR}{\frac{1}{(4 \pi(t-s))^{1/2}}\e^{\frac{-|x-z|^2}{4(t-s)}} \Big( \p_z F(U(z,s,y)) + \cQ(U(z,s,y)) \Big) }{z}{s} \right| &\leq \Vnorm{\cF(U)}_{L^{\infty}} \iintdmt{0}{t}{\bbR}{\cJ(z,t-s)}{z}{s} \\
    & = C t \to 0 \text{ as } t \to 0\,.
\end{split}
\end{equation}
\end{proof}

\section{Weakly Coupled Semi-linear Parabolic Systems}\label{apdx-SLParabolicSystems}

We prove the global well-posedness of the system \eqref{eq-PSystemSemiLinear-ViscosityApproximation-RiemannianCoordinates}, which in turn is used to prove the global well-posedness of the system \eqref{eq-PSystemSemilinearFull-ViscosityApproximation} in Section \ref{sec-SemiLinearPSystem} and to demonstrate the existence of a solution to \eqref{eq-PSystemSemiLinear-RiemannianCoordinates} in Appendix \ref{apdx-SLHyperbolicSystems}. The following results are straightforward generalizations of standard theorems for weakly coupled semi-linear parabolic systems found in \cite{NataliniPSystem, FifePaulC1979MAoR}, included here for completeness. First, we need the following comparison theorem for diagonal semi-linear parabolic equations.
It is a special case of \cite[Theorem~5.1]{FifePaulC1979MAoR} adapted to our situation.

\begin{theorem}[Comparison Principle]\label{thm-PSystemSemiLinear-QuasiMonotonicity}
Let $\cQ : \bbR^2 \to \bbR^2$ be a given $C^1$ function. Consider two classical solutions $V^1$, $V^2$ of the $2 \times 2$ weakly-coupled diagonal system
\begin{equation}\label{eq-PSystemSemiLinear-QuasiMonotonicity}
\begin{cases}
\dt v_i + \lambda_i \dx v_i  = \cQ_i(V) + \dxx v_i \\
V|_{t=0} = V_0\,,
\end{cases}
\end{equation}
where $\lambda_i \in \bbR$, with initial data $V_0^1$, $V_0^2$ respectively, defined on some domain $G \subset \bbR \times (0,\infty) \times \Gamma$. Suppose $\cQ$ is quasimonotone, i.e.
\begin{equation}
\frac{\p \cQ_i}{\p v_j} \geq 0 \text{ for each } i \neq j\,.
\end{equation}
Then $V_0^1 \leq V_0^2 \, \Rightarrow \, V^1 \leq V^2$ $(x,t,y)$-a.e. in $G$.
\end{theorem}
The proof of Theorem \ref{thm-PSystemSemiLinear-QuasiMonotonicity} is a straightforward extension of the proof of \cite[Theorem~5.1]{FifePaulC1979MAoR}. For each fixed $y \in \Gamma$, apply the result in \cite{FifePaulC1979MAoR} to the functions $V^1(\cdot, \cdot, y)$ and $V^2(\cdot,\cdot,y)$. Then the result follows.

\begin{corollary}\label{cor-PSystemSemiLinear-UniquenessParabolicEquation-RiemannianCoordinates}
There exists at most one classical solution of the diagonalized approximate parabolic system \eqref{eq-PSystemSemiLinear-ViscosityApproximation-RiemannianCoordinates}.
\end{corollary}

\begin{proof}
In components, the system \eqref{eq-PSystemSemiLinear-ViscosityApproximation-RiemannianCoordinates} takes the form
\begin{equation}
\begin{cases}
    \dt \wen - a \dx \wen = - \frac{1}{\veps} H(\wen,\zen) + \nu \dxx \wen \\
    \dt \zen + a \dx \zen = \frac{1}{\veps} H(\wen,\zen) + \nu \dxx \zen \\
    (\wen,\zen)|_{t=0} = (\wnen,\znen)\,,
\end{cases}
\end{equation}
which is easily seen to be of the form \eqref{eq-PSystemSemiLinear-QuasiMonotonicity}.
The weak coupling term $\cH$ is quasimonotone as a result of the subcharacteristic condition, and the statement follows.
\end{proof}

\begin{proposition}\label{prop-LocalExistenceOfClassicalViscApproxSoln-RiemannianCoordinates}
Let $\veps$, $\nu > 0$. Suppose $\Wnen$, $\dx \Wnen$, $\dxx \Wnen \in \big[ L^2(\bbR \times \Gamma) \big]^2 \cap \big[ L^{\infty}(\bbR \times \Gamma) \big]^2$. Then there exists a time $T^{\veps,\nu} > 0$ such that \eqref{eq-PSystemSemiLinear-ViscosityApproximation-RiemannianCoordinates} possesses a unique classical solution $\Wen$ in the strip $\bbR \times [0,T^{\veps,\nu}] \times \Gamma$.
\end{proposition}

\begin{proof}
Note that $\cH(0,0) = (0,0)$. Set $\Ubar \equiv (0,0)$. By Theorem \ref{thm-QLParabolicSystemsMainThm} there exists a time $(T^{\veps, \nu})' >0$ such that the equation \eqref{eq-QuasilinearParabolicSystem}
with
\begin{equation}
    F(U) =  \diag(-a,a)\, U \,, \qquad \cQ(U) =  \frac{\nu}{\veps} \cH(U)\,,
\end{equation}
and initial data $\widetilde{\Wnen}(x,y) := \Wnen(\nu x,y)$ has a classical solution $\widetilde{\Wen}$ in the strip $\bbR \times [0,(T^{\veps,\nu})'] \times \Gamma$.
Set $\Wen(x,t,y) := \widetilde{\Wen} \left( \frac{x}{\nu},\frac{t}{\nu},y \right)$. By checking the scaling directly, one sees that $\Wen$ is a classical solution of \eqref{eq-PSystemSemiLinear-ViscosityApproximation-RiemannianCoordinates} with initial data $\Wnen(x,y)$ in the strip $\bbR \times [0,T^{\veps,\nu}] \times \Gamma$, where $T^{\veps,\nu} := \nu (T^{\veps,\nu})'$. By Corollary \ref{cor-PSystemSemiLinear-UniquenessParabolicEquation-RiemannianCoordinates} this classical solution is unique.
\end{proof}

\begin{theorem}[Continuation Principle]
Let $\Wnen \in \big[ L^2(\bbR \times \Gamma) \big]^2  \cap \big[ L^{\infty}(\bbR \times \Gamma) \big]^2 $. Let $\Ubar \equiv (0,0)$. Let $T^{\veps,\nu}_*$ be the supremum of all $T > 0$ such that the system \eqref{eq-PSystemSemiLinear-ViscosityApproximation-RiemannianCoordinates} has a solution $\Wen \in G_T^{\infty} \cap G_T^2$. Then, either $T^{\veps,\nu}_* = \infty$ or
\begin{equation}
\lim\limits_{T \to (T^{\veps,\nu}_*)^-} \sup_{0 \leq t \leq T} \Vnorm{U(t)}_{\infty} = \infty\,.
\end{equation}
\end{theorem}
\begin{proof}
Suppose $T^{\veps,\nu}_* < \infty$ and that 
\begin{equation}
\lim\limits_{T \to (T^{\veps,\nu}_*)^-} \sup_{0 \leq t \leq T} \Vnorm{U(t)}_{\infty} < \infty\,.
\end{equation}
We show that the solution can be extended. (This is a standard argument.) By the assumption and by properties of the constructed solution, we see that for any $t \in (0,T^{\veps,\nu}_*)$, $\Vnorm{U(t)}_{\infty} \leq C \Vnorm{U_0}_{\infty}$.
The maximal existence time $T^{\veps,\nu}_*$ is at least
\begin{equation}
\frac{1}{2} \left( \frac{-2 \mu \alpha_1(b) + \sqrt{4 \mu^2 \alpha_1(b)^2 - 4 \alpha_2(b) \left( \frac{b_0}{b} - 1 \right)}}{2 \alpha_2(b)} \right)^2 := T^{\veps,\nu}_{**}\,.
\end{equation}
Here, $b_0 = \Vnorm{\Wnen}_{\infty}$, $b$ is any number bigger than $b_0$ (W.L.O.G. say $b = 2b_0$), $\alpha_1(b) = a$, and $\alpha_2(b) = \frac{\nu}{\veps} \Vnorm{\dU \cH}_{L^{\infty}(B^2_b(0))}$.
Let $\delta > 0$ such that
\begin{equation}
\delta < \frac{1}{2} T^{\veps,\nu}_{**}\,.
\end{equation}
Then a repeat of the arguments in Appendix \ref{apdx-QLParabolicSystems} shows that the maximal existence time of the fixed point of 
\begin{equation}
\begin{split}
    (\tilde{L}V)(t) &:= \cJ(\cdot,t) \, \ast \, U(\cdot,T^{\veps,\nu}_*-\delta,y) + \intdmt{0}{t}{\dx \cJ(\cdot,t-s) \, \ast \, \diag(-a,a) \, V(\cdot,s,y)}{s} \\
    & \qquad - \intdmt{0}{t}{\cJ(\cdot,t-s) \, \ast \, \frac{\nu}{\veps} \cH(V(\cdot,s,y))}{s}\,.
\end{split}
\end{equation}
is at least $2 \delta$. Since solutions to \eqref{eq-PSystemSemiLinear-ViscosityApproximation-RiemannianCoordinates} are unique by Corollary \ref{cor-PSystemSemiLinear-UniquenessParabolicEquation-RiemannianCoordinates}, the solution is thus extended onto the interval $[0,T^{\veps,\nu}_*+\delta]$.
\end{proof}

Before taking $\veps \to 0$ or $\nu \to 0$ we will in fact see in the next theorem that the solutions $\Uen$ are global for small data, and so later the $\veps$, $\nu$ limits can be taken without worrying about the interval of existence for solutions.

\begin{theorem}[Global Existence]\label{thm-PSystemSemiLinear-GlobalExistenceOfViscositySolutions}
For fixed $\veps$, $\nu > 0$ assume that
\begin{equation}
\begin{split}
\Wnen, \, \dx \Wnen, \, \dxx \Wnen \in \left[ L^2_{\rho}(\bbR \times \Gamma) \right]^2 \cap \left[L^{\infty}(\bbR \times \Gamma) \right]^2\,,
\end{split}
\end{equation}
and that
\begin{equation}
    \Vnorm{\Wnen} + \Vnorm{\dx \Wnen} + \Vnorm{\dxx \Wnen}\leq C
\end{equation}
uniformly in $\veps$ and $\nu$, where $\Vnorm{\cdot} = \Vnorm{\cdot}_{\infty} + \Vnorm{\cdot}_{2;\rho}$.
Then the unique local classical solution $\Wen$ of \eqref{eq-PSystemSemiLinear-ViscosityApproximation-RiemannianCoordinates} can be extended globally (i.e. $T_*^{\veps, \nu} = \infty$).
\end{theorem}
\begin{proof}
By the continuation principle it suffices to show that $\Wen$ is bounded in $\left[ L^{\infty}(\bbR \times (0,T_*^{\veps,\nu}) \times \Gamma) \right]^2$ uniformly in $\veps$ and $\nu$.
The idea is to use the comparison principle and construct appropriate initial data and corresponding solutions to compare to $\Wen$.

Consider the system of ODE
\begin{equation}\label{eq-PSystemSemiLinear-ODESystem}
    \dot{p}_1 = -\frac{1}{\veps} H(p_1,p_2)\,, \qquad \qquad 
    \dot{p}_2 = \frac{1}{\veps} H(p_1,p_2)\,,
\end{equation}
with $p_1(0)=p_2(0)=p_0$, to be determined later.
The global solutions are given by
\begin{equation} 
\begin{split} 
    p_1(t) = p_0 + f \left( - \frac{p_0}{a} \right) \left( 1 - \e^{-t/\veps}\right)\,,  \qquad \qquad 
    p_2(t) = p_0 - f \left( - \frac{p_0}{a} \right) \left( 1 - \e^{-t/\veps}\right)\,.
\end{split}
\end{equation}
$(p_1,p_2)$ is in fact a solution of the system \eqref{eq-PSystemSemiLinear-ViscosityApproximation-RiemannianCoordinates} with initial data $(p_0,p_0)$.
Set
\begin{equation} 
p_0^{\pm} = \pm \sup_{\veps, \nu} \max \left( \Vnorm{\wnen}_{\infty}, \Vnorm{\znen}_{\infty} \right)\,.
\end{equation}
Denote $(p_1^+,p_2^+)$, $(p_1^-,p_2^-)$ to be the solutions of \eqref{eq-PSystemSemiLinear-ODESystem} with initial times $p_0^+$, $p_0^-$ respectively. Then since
\begin{equation}
    p_0^- \leq \wnen(x,y) \leq p_0^+\,, \qquad p_0^- \leq \znen(x,y) \leq p_0^+\,,
\end{equation}
we conclude by the comparison principle that
\begin{equation} 
p_1^-(t) \leq \wen(x,t,y) \leq p_1^+(t)\,, \qquad p_2^-(t) \leq \zen(x,t,y) \leq p_2^+(t)\,,
\end{equation}
for every $(x,t,y) \in \bbR \times (0,T^{\veps, \nu}) \times \Gamma$.
Therefore, using the exact solutions $(p_1^{\pm},p_2^{\pm})$,
\begin{equation}\label{eq-PSystemSemiLinear-LInftyBoundsOnViscosityApproximation}
|\Wen(x,t,y)| \leq p_0^+ + \max \left\lbrace \left| f \left( \frac{p_0^+}{a} \right) \right|\,, \left| f \left( - \frac{p_0^+}{a} \right) \right| \right\rbrace \leq C \,.
\end{equation}
\end{proof}

\section{Weakly Coupled Semi-linear Hyperbolic Systems}\label{apdx-SLHyperbolicSystems}
Here we state and prove the well-posedness theory in the $L^1$ framework for the weakly coupled $2 \times 2$ hyperbolic system \eqref{eq-PSystemSemiLinear-RiemannianCoordinates}. This will be used in Section \ref{sec-SemiLinearPSystem} to show the global well-posedness of the $p$-system \eqref{eq-PSystemSemilinearFull} with initial data in $L^{\infty}$. The system can be written in components as
\begin{equation}
\begin{cases}
    \dt \we - a \dx \we = - \frac{1}{\veps} H(\we,\ze) \\
    \dt \ze + a \dx \ze = \frac{1}{\veps} H(\we,\ze) \\
    (\we,\ze)|_{t=0} = (\wne,\zne)\,.
\end{cases}
\end{equation}
These results of well-posedness are classical in the case of deterministic initial data \cite{NataliniPSystem}, and much of the proofs generalize in a straightforward manner. We include them for completeness.

When applied to the diagonal system \eqref{eq-PSystemSemiLinear-RiemannianCoordinates}, Definition \ref{def-2x2Solution} becomes the following:

\begin{definition}\label{def-PSystemSemiLinear-EntropySolutionForDiagonalSystem}
For $\veps > 0$, $0 < T \leq \infty$, and $\Wne \in \big[ L^{\infty}(\bbR \times \Gamma) \big]^2$, a function $\We \in \big[ L^{\infty}(\bbR \times (0,T) \times \Gamma) \big]^2$ is a \textit{weak solution} of \eqref{eq-PSystemSemiLinear-RiemannianCoordinates} if
\begin{equation}
\begin{split}
    \iiintdmt{0}{T}{\bbR}{\Gamma}{\Big( \We \cdot \dt \vphi + \diag(-a,a) \dx \vphi + \frac{1}{\veps}\cH(\We) \cdot \vphi \Big) \rho}{y}{x}{t}
    + \iintdm{\bbR}{\Gamma}{\Wne \cdot \vphi(x,0,y)\rho}{y}{x} = 0
\end{split}
\end{equation}
for every $\vphi = (\vphi_1, \vphi_2) \in \left[C^{\infty}_c(\bbR \times [0,T) \times \Gamma) \right]^2$.
If in addition $\We$ satisfies
\begin{equation}\label{eq-PSystemSemiLinear-DiagonalSystemEntropyCondition1}
\begin{split}
\iiintdmt{0}{T}{\bbR}{\Gamma}{\Big( \eta(\We) \dt \vphi + Q(\We) \dx \vphi + \frac{1}{\veps} \p_W \eta(\We) \cdot \cH(\We) \vphi \Big) \rho}{y}{x}{t} \geq 0
\end{split}
\end{equation}
for every convex entropy pair $(\eta,Q)$ for \eqref{eq-PSystemSemiLinear-RiemannianCoordinates} over $\bbR^2$ and for every $\vphi \in C^{\infty}_c(\bbR \times (0,T) \times \Gamma)$ with $\vphi \geq 0$, and
\begin{equation}\label{eq-PSystemSemiLinear-DiagonalSystemEntropyCondition2}
\lim\limits_{T \to 0^+} \frac{1}{T} \iiintdmt{0}{T}{V}{\Gamma}{|\We(x,t,y)-\Wne(x,y)| \rho(y)}{y}{x}{t} = 0
\end{equation}
for every $V \Subset \bbR$, then we say that $\We$ is a \textit{weak entropy solution} of \eqref{eq-PSystemSemiLinear-RiemannianCoordinates}.
\end{definition}
The goal of this appendix is to prove the following theorem:

\begin{theorem}\label{thm-PSystemSemiLinear-WellPosednessInL1ForDiagonalSystem}
Let $\Wne \in \big[ L^{\infty}(\bbR \times \Gamma) \big]^2$. Then there exists a unique weak entropy solution $\We$ of \eqref{eq-PSystemSemiLinear-RiemannianCoordinates} with initial data $\Wne$ and $\We$ belongs to $C \left( [0,\infty); \big[ L^1_{\rho,loc}(\bbR \times \Gamma) \big]^2 \right)$.
\end{theorem}

Just as in the nonparametric case the apriori $L^1_{loc}$ estimates are crucial, and used throughout the subsequent arguments.

\begin{lemma}\label{lma-PSystemSemiLinear-L1LocStabilityEstimate}
Let $\We$ and $\widetilde{\We}$ be two weak entropy solutions of \eqref{eq-PSystemSemiLinear-RiemannianCoordinates} corresponding to $\big[ L^{\infty}(\bbR \times \Gamma) \big]^2$ initial data $\Wne$, $\widetilde{\Wne}$ respectively.
Then for every $m > 0$ and every $t \in (0,T]$ we have
\begin{equation}\label{eq-PSystemSemiLinear-L1LocStabilityEstimateForW}
\begin{split}
    &\int\limits_{|x|<m} \int\limits_{\Gamma_1}{|\we(x,t,y)-\widetilde{\we}(x,t,y)| \, \rho^2}\, \mathrm{d}y \, \mathrm{d}x \leq  \int\limits_{|x| < m + at} \int\limits_{\Gamma_1}{|\Wne(x,y) - \widetilde{\Wne}(x,y)| \, \rho^2}\, \mathrm{d}y \, \mathrm{d}x \\
    &\qquad + \int\limits_{0}^{t} \int\limits_{|x| \leq m + a(t-s)} \int\limits_{\Gamma_1}{\sgn \big( \we(x,s,y)-\widetilde{\we}(x,s,y) \big) \left( H \big( \We(x,s,y) \big) - H \big( \widetilde{\We}(x,s,y) \big) \right) \, \rho^2}\, \mathrm{d}y \, \mathrm{d}x \, \mathrm{d}s
\end{split}
\end{equation}
and
\begin{equation}\label{eq-PSystemSemiLinear-L1LocStabilityEstimateForZ}
\begin{split}
    &\int\limits_{|x|<m} \int\limits_{\Gamma_1}{|\ze(x,t,y)-\widetilde{\ze}(x,t,y)| \, \rho^2}\, \mathrm{d}y \, \mathrm{d}x \leq  \int\limits_{|x| < m + at} \int\limits_{\Gamma_1} {|\Wne(x,y) - \widetilde{\Wne}(x,y)| \, \rho^2}\, \mathrm{d}y \, \mathrm{d}x \\
    &\qquad + \int\limits_{0}^{t} \int\limits_{|x| \leq m + a(t-s)} \int\limits_{\Gamma_1}{\sgn \big( \ze(x,s,y)-\widetilde{\ze}(x,s,y) \big) \left( H \big( \widetilde{\We}(x,s,y) \big) - H \big( \We(x,s,y) \big) \right) \, \rho^2}\, \mathrm{d}y \, \mathrm{d}x \, \mathrm{d}s\,.
\end{split}
\end{equation}
\end{lemma}

\begin{proof}
It suffices to prove \eqref{eq-PSystemSemiLinear-L1LocStabilityEstimateForW}; \eqref{eq-PSystemSemiLinear-L1LocStabilityEstimateForZ} is proved similarly.
We choose a sequence of $C^2$ convex functions $\eta_n$ such that $\eta_n(W)$ converges locally uniformly to the convex function $|w - k|$ for some $k \in \bbR$. Thus, the entropy inequalities \eqref{eq-PSystemSemiLinear-DiagonalSystemEntropyCondition1} for $\We$ and $\widetilde{\We}$ become
\begin{equation}
\iiintdmt{0}{T}{\bbR}{\Gamma}{\big( |\we-k| (\dt \vphi - a \dx \vphi) - \frac{1}{\veps} \sgn(\we-k) H(\We) \big) \rho}{y}{x}{t} \geq 0
\end{equation}
and
\begin{equation}
\iiintdmt{0}{T}{\bbR}{\Gamma}{\big( |\widetilde{\we}-k| (\dt \vphi - a \dx \vphi) - \frac{1}{\veps} \sgn(\widetilde{\we}-k) H(\widetilde{\We}) \big) \rho}{y}{x}{t} \geq 0\,.
\end{equation}
Mechanically speaking, the proof now proceeds in a fashion virtually identical to the proof of Theorem \ref{thm-Appendix-UniquenessOfEntropySolution}.
The point of departure is the presence of the nonlinearity $H$ in the entropy inequalities for $\We$ and $\widetilde{\We}$, which accounts for the additional integral term on the right-hand side of \eqref{eq-PSystemSemiLinear-L1LocStabilityEstimateForW}.
\end{proof}

\begin{theorem}\label{thm-PSystemSemiLinear-L1LocStabilityEstimate}
Let $\Wne$, $\widetilde{\Wne} \in \big[ L^{\infty}(\bbR \times \Gamma) \big]^2$ and let $\We$ and $\widetilde{\We}$ be entropy solutions of \eqref{eq-PSystemSemiLinear-RiemannianCoordinates} corresponding to $\Wne$, $\widetilde{\Wne}$ respectively. Then for any $m > 0$, $\Gamma_1 \Subset \Gamma$ and any $t>0$,
\begin{equation}\label{eq-PSystemSemiLinear-L1LocStabilityEstimate}
\begin{split}
    \int\limits_{|x| < m} \int\limits_{\Gamma_1}{\left|\We(x,t,y) - \widetilde{\We}(x,t,y) \right| \rho^2(y)}\, \mathrm{d}y \, \mathrm{d}x \leq 2 \e^{2(a+1)t} \int\limits_{|x| \leq m + at} \int\limits_{\Gamma_1}{\left| \Wne(x,y) - \widetilde{\Wne}(x,y) \right| \, \rho^2(y)}\, \mathrm{d}y \, \mathrm{d}x\,.
\end{split}
\end{equation}
\end{theorem}
\begin{proof}
For any $t \in [0,T]$ and $0 < s < t$ set
\begin{equation}
    \mu(s) := \int\limits_{|x| \leq m+a(t-s)} \int\limits_{\Gamma_1}{ \left|\We(x,t,y) - \widetilde{\We}(x,t,y) \right| \rho^2(y)}\, \mathrm{d}y \, \mathrm{d}x\,.
\end{equation}
Then using the fact that $H$ is globally Lipschitz with bound $\Vnorm{\dU \cH}_{\infty} \leq a+1 $, we have from Lemma \ref{lma-PSystemSemiLinear-L1LocStabilityEstimate} the estimate
\begin{equation}
    \mu(t) \leq 2\mu(0) + 2(a+1) \intdmt{0}{t}{\mu(s)}{s}\,.
\end{equation}
The conclusion follows by Gronwall's inequality.
\end{proof}

It follows immediately from the estimate \eqref{eq-PSystemSemiLinear-L1LocStabilityEstimate} that

\begin{corollary}\label{cor-PSystemSemiLinear-EntropySolutionsOfRiemannianCoordinatesAreUnique}
For any $\Wne \in \big[ L^{\infty}(\bbR \times [0,T] \times \Gamma) \big]^2$ there exists at most one entropy solution of \eqref{eq-PSystemSemiLinear-RiemannianCoordinates}.
\end{corollary}

With the uniqueness theory for the diagonal system \eqref{eq-PSystemSemiLinear-RiemannianCoordinates} in place, we now establish existence using a vanishing viscosity argument. Consider solutions $\Wen$ to the viscosity approximation \eqref{eq-PSystemSemiLinear-ViscosityApproximation-RiemannianCoordinates} with suitable initial data; by Theorem \ref{thm-QLParabolicSystemsMainThm}, Corollary \ref{cor-PSystemSemiLinear-UniquenessParabolicEquation-RiemannianCoordinates} and Theorem \ref{thm-PSystemSemiLinear-GlobalExistenceOfViscositySolutions} the functions $\Wen$ are classical, unique, global solutions to \eqref{eq-PSystemSemiLinear-ViscosityApproximation-RiemannianCoordinates}.

\begin{theorem}\label{thm-PSystemSemiLinear-AlmostEverywhereConvergenceToEntropySolution}
For $\veps > 0$, let $\Wne \in \big[ L^1(\bbR \times \Gamma) \big]^2 \cap \big[ L^{\infty}(\bbR \times \Gamma) \big]^2$ take values in a bounded open convex set $D \subset \bbR^2$. Let $(\Wnen)_{\nu>0} \subset \big[ C^{\infty}_c(\bbR \times \Gamma) \big]^2$ be a sequence satisfying $\Vnorm{\Wnen}_{\infty} \leq \Vnorm{\Wne}_{\infty}$ and converging to $\Wne$ in $\big[ L^1_{\rho, loc}(\bbR \times \Gamma) \big]^2 $. Let $\Wen$ be the unique global classical solution of \eqref{eq-PSystemSemiLinear-ViscosityApproximation-RiemannianCoordinates} with initial data $\Wnen$. Suppose also that there exists a subsequence (still denoted by $\nu$) such that $\Wen$ converges to some function $\We$ in the space $C \left( [0,\infty) ;  \left[ L^1_{\rho, loc}(\bbR \times \Gamma) \right]^2 \right)$. Then $\We \in \big[ L^{\infty}(\bbR \times (0,\infty) \times \Gamma) \big]^2$ and $\We$ is a weak entropy solution of \eqref{eq-PSystemSemiLinear-RiemannianCoordinates}.
\end{theorem}

\begin{proof}
By virtue of \eqref{eq-PSystemSemiLinear-LInftyBoundsOnViscosityApproximation} the function $\We$ belongs to $\big[ L^{\infty}(\bbR \times (0,\infty) \times \Gamma) \big]^2$. 
Take any convex entropy $\eta$ for the diagonal system \eqref{eq-PSystemSemiLinear-RiemannianCoordinates}.
Multiply \eqref{eq-PSystemSemiLinear-ViscosityApproximation-RiemannianCoordinates} by $\p_W \eta(\Wen)$ to get
\begin{equation}\label{eq-PSystemSemiLinear-ViscosityApproximationMultipliedByEntropy}
    \p_t \eta(\Wen) + \p_x Q(\Wen) = \frac{1}{\veps} \p_W \eta(\Wen) \cH(\Wen) + \nu \dxx \eta(\Wen) - \nu \Vint{\p_{WW} \eta(\Wen) \dx \Wen, \dx \Wen}\,.
\end{equation}
multiply by a nonnegative smooth test function $\vphi$ compactly supported on $\bbR \times (0,\infty) \times \Gamma$, integrate over $\bbR \times [0,\infty) \times \Gamma$, and integrate by parts. The fact that the last term on the right-hand side of \eqref{eq-PSystemSemiLinear-ViscosityApproximationMultipliedByEntropy} is nonnegative gives us the inequality
\begin{equation}
\begin{split}
    \iiintdmt{0}{\infty}{\bbR}{\Gamma}{\big( \eta(\Wen) \dt \vphi + Q(\Wen) \dx \vphi + & \frac{1}{\veps} \p_W \eta(\Wen) \cH(\Wen) \vphi \big) \, \rho}{y}{x}{t} \\
    &\qquad \geq - \nu \iiintdmt{0}{\infty}{\bbR}{\Gamma}{\eta(\Wen) \dxx \vphi \, \rho}{y}{x}{t}\,.
\end{split}
\end{equation}
Taking the limit as $\nu \to 0$,
\begin{equation}
    \iiintdmt{0}{\infty}{\bbR}{\Gamma}{\big( \eta(\We) \dt \vphi + Q(\We) \dx \vphi + \frac{1}{\veps} \p_W \eta(\We) \cH(\We) \vphi \big) \, \rho}{y}{x}{t} \geq 0\,,
\end{equation}
which is the entropy inequality \eqref{eq-PSystemSemiLinear-DiagonalSystemEntropyCondition1}. To see the $L^1_{loc}$ convergence to the initial data \eqref{eq-PSystemSemiLinear-DiagonalSystemEntropyCondition2}, use the triangle inequality to obtain
\begin{equation}
\begin{split}
\frac{1}{T} \iiintdmtlim{0}{T}{V}{\Gamma}{|\We(x,t,y) - \Wne(x,y)| \rho(y)}{y}{x}{t} &\leq \frac{1}{T} \iiintdmtlim{0}{T}{V}{\Gamma}{|\We(x,t,y) - \Wen(x,t,y)| \rho(y)}{y}{x}{t} \\
&\qquad + \frac{1}{T} \iiintdmtlim{0}{T}{V}{\Gamma}{|\Wen(x,t,y) - \Wnen(x,y)| \rho(y)}{y}{x}{t} \\
&\qquad + \frac{1}{T} \iiintdmtlim{0}{T}{V}{\Gamma}{|\Wnen(x,y) - \Wne(x,y)| \rho(y)}{y}{x}{t} \\
&= \mathrm{I} + \mathrm{II} + \mathrm{III}
\end{split}
\end{equation}
for every $V \Subset \bbR$. By the assumption of convergence, $\mathrm{I}$ can be made arbitrarily small and independent of $T$ for $\nu$ small enough. By assumption, $\mathrm{III}$ can also be made arbitrarily small for $\nu$ small enough. By Theorem \ref{thm-QLParabolicSystemsMainThm}, for $T$ small enough,
\begin{equation}
\begin{split}
\mathrm{II} &= \frac{1}{T} \iiintdmt{0}{T}{V}{\Gamma}{|\dt \Wen(x,\theta t,y) \, t| \, \rho(y)}{y}{x}{t} \\
&\leq C |V| \Vnorm{\rho}_1 \Big( \Vnorm{\Wnen}_{\infty} + \Vnorm{\dx \Wnen}_{\infty} + \Vnorm{\dxx \Wnen}_{\infty} \Big)  \left( \frac{1}{T} \intdmt{0}{T}{t}{t} \right) = C T\,,
\end{split}
\end{equation}
where $C$ is independent of $T$. Therefore, $\mathrm{II} \to 0$ as $T \to 0^+$, and \eqref{eq-PSystemSemiLinear-DiagonalSystemEntropyCondition2} is proved.
\end{proof}
We are now ready to estimate the $L^1$ modulus of continuity of the solutions $\Wen$.

\begin{theorem}\label{thm-PSystemSemiLinear-L1ModulusOfContinuityInSpace}
Let $\Wnen$ be a sequence of functions in  $\big[ C^{\infty}_c(\bbR \times \Gamma) \big]^2$ taking values in a bounded open convex set $D_0 \subset \bbR^2$. For each $\veps, \nu$ let $\Wen$ be the unique global smooth solution of \eqref{eq-PSystemSemiLinear-ViscosityApproximation-RiemannianCoordinates} corresponding to the data $\Wnen$. Let $m > 0$ and let $\Gamma_1 \Subset \Gamma$. Let $h \in \bbR$ and let $k \in \bbR^N$ such that $|k| < \frac{1}{2} \mathrm{dist}(\Gamma_1, \Gamma^c)$. Then for every $T > 0$ we have the estimate
\begin{equation}
\begin{split}
    \iintdm{|x| < m}{\Gamma_1}{ \big| \Wen(x+h,T,y+k) &- \Wen(x,T,y) \big| \rho(y)}{y}{x} \\
    &\leq C \iintdm{\bbR}{\Gamma_1}{|\Wnen(x+h,y+k)-\Wnen(x,y)| \, \rho(y)}{y}{x}\,,
\end{split}
\end{equation}
where the constant $C$ is independent of $\nu$.
\end{theorem}

\begin{proof}
We follow the argument in \cite{NataliniQuasilinearSystems}. Throughout the proof we use the fact that by Theorem \ref{thm-PSystemSemiLinear-GlobalExistenceOfViscositySolutions} the solutions $\Wen$ are uniformly bounded in $L^{\infty}$ with respect to $\nu$, taking values in some bounded open convex set $D \subset \bbR^2$, where $D$ depends on $D_0$, $a$ and $f$. For a fixed $h$ and $k$, define
\begin{equation}
    \Wenhk(x,t,y):= \Wen(x+h,t,y+k) - \Wen(x,t,y)\,.
\end{equation}
Then $\Wenhk$ satisfies
\begin{equation}\label{eq-PSystemSemiLinear-L1ContinuityEquation}
    \dt \Wenhk + \diag(-a,a) \dx \Wenhk - \frac{1}{\veps} B(\Wenhk) \Wenhk = \nu \dxx \Wenhk
\end{equation}
pointwise in $(x,t,y) \in \bbR \times (0,\infty) \times \Gamma_1$. Here,
\begin{equation}
\begin{split}
    B(\Wenhk) &= \diag\big( b_1^{\veps,\nu,h,k},b_2^{\veps,\nu,h,k} \big)\,, \\
    b_1^{\veps,\nu,h,k}(x,t,y) &:= - \intdmt{0}{1}{\p_w \Big[ H \big(\beta \wen(x+h,t,y+k) + (1-\beta)\wen(x,t,y), \zen(x,t,y) \big) \Big]}{\beta}\,, \\
    b_2^{\veps,\nu,h,k}(x,t,y) &:= \intdmt{0}{1}{\p_z \Big[ H \big(\wen(x,t,y), \beta \zen(x+h,t,y+k) + (1-\beta)\zen(x,t,y) \big) \Big]}{\beta}\,.
\end{split}
\end{equation}
Since $\cH$ is Lipschitz on $D$, $B$ has matrix norm bounded independent of $x$, $t$, $y$, $\veps$, $\nu$, $h$, and $k$.
Let $T>0$. For $\veps$, $\nu > 0$ let $\vphi^{\veps, \nu} = (\vphi^{\veps,\nu}_1, \vphi^{\veps, \nu}_2) \in \big[ C^{\infty}(\bbR \times (0,T] \times \Gamma) \big]^2 \cap \big[ L^{\infty}(\bbR \times [0,T] \times \Gamma) \big]^2$.
For reasons that will become clear, we choose $\vphi^{\veps, \nu}$ to be the solution of the adjoint differential system
\begin{equation}\label{eq-PSystemSemiLinear-BackwardParabolicSystem}
\begin{cases}
\dt \vphi^{\veps, \nu} + \diag(-a,a) \dx \vphi^{\veps, \nu} + \frac{1}{\veps} B(\Wenhk) \vphi^{\veps, \nu} + \nu \dxx \vphi^{\veps, \nu} = 0\,, &\qquad (x,t,y) \in \bbR \times [0,T] \times \Gamma_1 \\
\vphi^{\veps, \nu}(x,T,y) = \psi^{\veps, \nu}(x,y) &\qquad x \in \bbR\,, \, y \in \Gamma_1\,.
\end{cases}
\end{equation}
Here $\psi^{\veps, \nu}$ is a given function in $\big[ L^{\infty}(\bbR \times \Gamma) \big]^2$ with $\supp \psi^{\veps, \nu} \subset [-m,m] \times \Gamma_1$ and $\Vnorm{\psi^{\veps, \nu}}_{\infty} \leq 1$, to be chosen in a moment. Following line-by-line the proof of \cite[Lemma 4]{Kruzkov} we see that any solution of \eqref{eq-PSystemSemiLinear-BackwardParabolicSystem} satisfies
\begin{equation}\label{eq-PSystemSemiLinear-UniformBoundOnParabolicAdjointOperatorSolution}
    |\vphi^{\veps, \nu}(x,t,y)| \leq 2 \min \left\lbrace 1, \e^{- \frac{1}{\nu} (|x|-m-(1+a)(T-t))} \e^{ \frac{L_1}{\veps}(T-t) } \right\rbrace\,, \qquad (x,t,y) \in \bbR \times [0,T] \times \Gamma_1\,,
\end{equation}
where $L_1 = \Vnorm{\p_W \cH}_{L^{\infty}(D)}$.
Now, taking the Euclidean inner product of \eqref{eq-PSystemSemiLinear-L1ContinuityEquation} with $\vphi^{\veps, \nu}$, multiplying by $\rho(y)$, integrating over $\bbR \times [0,T] \times \Gamma_1$ and integrating by parts gives
\begin{equation}
\begin{split}
    \iintdm{\bbR}{\Gamma_1}{&\Wenhk(x,T,y) \cdot \vphi^{\veps, \nu}(x,T,y) \, \rho}{y}{x} = \iintdm{\bbR}{\Gamma_1}{\Wenhk(x,0,y) \cdot \vphi^{\veps, \nu}(x,0,y) \, \rho}{y}{x} \\
    &+ \iiintdmt{0}{T}{\bbR}{\Gamma_1}{\Wenhk(x,t,y) \Big[ \dt \vphi^{\veps, \nu} + \diag(-a,a) \dx \vphi^{\veps, \nu} + \frac{1}{\veps}B \vphi^{\veps, \nu} + \nu \dxx \vphi^{\veps, \nu} \Big] \, \rho}{y}{x}{t}\,.
\end{split}
\end{equation}
The only boundary term remaining after the integration by parts is the one involving the time derivative, since
\begin{equation}
|\vphi^{\veps,\nu}(x,t,y)| \leq C \e^{-|x|}\,, \qquad  (x,t,y) \in \bbR \times [0,T] \times \Gamma_1\,, \qquad \nu \in (0,1)\,,
\end{equation}
for some constant $C$ independent of $x$.
Therefore, choosing
\begin{equation}
\psi^{\veps,\nu}(x,y) = \chi_{[-m,m] \times \Gamma_1}(x,y)
    \begin{bmatrix}
    \sgn(\wenhk(x,T,y)) \\
    \sgn(\zenhk(x,T,y)) \\
    \end{bmatrix}
\end{equation}
and using the bound \eqref{eq-PSystemSemiLinear-UniformBoundOnParabolicAdjointOperatorSolution},
\begin{equation}
\begin{split}
    \iintdmt{-m}{m}{\Gamma_1}{\big| \Wenhk(x,T,y) \big| \, \rho(y)}{y}{x} &= \iintdm{\bbR}{\Gamma_1}{\Wenhk(x,0,y) \cdot \vphi^{\veps, \nu}(x,0,y) \, \rho}{y}{x} \\
    &\leq 2 \iintdm{\bbR}{\Gamma_1}{ \big| \Wnen(x+h,y+k) - \Wnen(x,y) \big| \, \rho(y)}{y}{x}\,, \\
\end{split}
\end{equation}
as desired.
\end{proof}

Now we estimate the $L^1$ modulus of continuity in time.

\begin{theorem}\label{thm-PSystemSemiLinear-L1ModulusOfContinuityInTime}
Let $\Wnen$ be a sequence of functions in  $\big[ C^{\infty}_c(\bbR \times \Gamma) \big]^2$ taking values in a bounded open convex set $D_0 \subset \bbR^2$. For each $\veps, \nu$ let $\Wen$ be the unique global smooth solution of \eqref{eq-PSystemSemiLinear-ViscosityApproximation-RiemannianCoordinates} corresponding to the data $\Wnen$. Let $m > 0$ and let $\Gamma_1 \Subset \Gamma$. Let $t$, $\tau > 0$. Then
\begin{equation}\label{eq-PSystemSemiLinear-L1ModulusOfContinuityInTime-Estimate}
\begin{split}
    \iintdm{|x| < m}{\Gamma_1}{&|\wen(x,t+\tau,y) - \wen(x,t,y)| + |\zen(x,t+\tau,y) - \zen(x,t,y)| \, \rho(y)}{y}{x} \\
    &\leq C \min_{0 < h \leq h_0} \left[ h + \max_{\substack{x' \in \{ -1,1 \} \\ y' \in \bbS^{d-1}}} \iintdm{\bbR}{\Gamma_1}{|\Wnen(x+hx',y+hy')-\Wnen(x,y)| \, \rho(y)}{y}{x} + \frac{\tau}{h^2} \right]\,,
\end{split}
\end{equation}
where the constant $C$ is independent of $\nu$.
\end{theorem}

\begin{proof}
As in the previous proof we use the fact that by Theorem \ref{thm-PSystemSemiLinear-GlobalExistenceOfViscositySolutions} the solutions $\Wen$ are uniformly bounded in $L^{\infty}$ with respect to $\nu$, taking values in some bounded open convex set $D \subset \bbR^2$, where $D$ depends on $D_0$, $a$ and $f$. We begin with the estimate for $\wen$. Define $\went(x,y) := \wen(x,t+\tau,y) - \wen(x,t,y)$. To estimate $\went(x,t,y)$ in $L^1_{\rho, loc}(\bbR \times \Gamma)$ we will add and subtract an appropriately mollified version of $\went(x,t,y)$ and estimate each term separately.
Let $\psi : \bbR \times \bbR^N \to \bbR$ be a standard mollifier, with $\psi \geq 0$, $\supp \psi \subset \{ (x,y) \in \bbR \times \bbR^N \, \big| \, |(x,y)| < 1 \}$, and $\Vnorm{\psi}_1 =  1$. Choose $0 < h_0 < 1$ so that for any $h \in (0,h_0]$ the sets $\Gamma_{1,h} := \{ y \in \Gamma_1 \, | \, \text{dist}(y,\p \Gamma_1) > h \} $ and $[-m+h,m-h]$ are both nonempty. Define $\psi_h$ as the convolution
\begin{equation}
    \psi_h(x,y) := \frac{1}{h^{N+1}} \iintdm{\bbR}{\bbR^N}{\psi \left( \frac{x-x'}{h} , \frac{y-y'}{h} \right) \chi_{|x| < m-h}(x') \chi_{\Gamma_{1,h}}(y') \sgn(\went(x',y'))}{y'}{x'}\,.
\end{equation}
The integral defining $\psi_h$ converges absolutely, and $\supp \psi_h \subset [-m,m] \times \Gamma_1$.
Multiplying \eqref{eq-PSystemSemiLinear-ViscosityApproximation-RiemannianCoordinates}$_1$ by $\psi_h \, \rho$, integrating over $[-m,m] \times [t,t+\tau] \times \Gamma_1$ and integrating by parts, 
\begin{equation}\label{eq-PSystemSemiLinear-L1ModulusOfContinuityInTimeProof-Estimate1}
\begin{split}
    \iintdm{|x| \leq m}{\Gamma_1}{\psi_h(x,y) (\went(x,y)) \rho(y)}{x}{y} &= \iiintdmt{t}{t+\tau}{|x| \leq m}{\Gamma_1}{\Big( -a \, \dx \psi_h \, \wen(x,s,y) \\
    & \qquad - \frac{1}{\veps} H(\wen,\zen) \psi_h + \nu \, \dxx \psi_h \, \wen(x,s,y) \Big) \, \rho}{y}{x}{s}\,.
\end{split}
\end{equation}
Note that there exists a constant $C$ such that $|\psi_h| \leq C$, $|\dx \psi_h| < Ch^{-1}$ and $|\dxx \psi_h| < Ch^{-2}$. Since $H(0,0) = 0$, we have that $|H(\wen,\zen)| \leq \Vnorm{\p_W H}_{L^{\infty}(D)} |\Wen|$. Note that the $\Wen$ take values in $D$. For $\nu < 1$ the right-hand side of \eqref{eq-PSystemSemiLinear-L1ModulusOfContinuityInTimeProof-Estimate1} is therefore majorized by
\begin{equation}\label{eq-PSystemSemiLinear-L1ModulusOfContinuityInTimeProof-Estimate2}
    C \, |D| \Vnorm{\rho}_{L^1(\Gamma)} \left( a \, h^{-1} + \frac{\Vnorm{\p_W H}_{L^{\infty}(D)}}{\veps} + \nu h^{-2} \right) (2m) \tau \leq \frac{C_1 \tau}{h^2}\,,
\end{equation}
since $h < 1$. The constant $C_1$ depends on $a$, $D$, $m$, $\rho$, $\veps$ and $H$ but not $\nu$, $\tau$ or $h$.
Now, by properties of standard mollifiers
\begin{equation}
\begin{split}
    \iintdmlim{|x| \leq m}{\Gamma_1}{&\Big| |\went(x,y)| - \psi_h \, \went(x,y) \Big| \rho }{y}{x} \\
    &= \iintdmlim{|x| \leq m}{\Gamma_1}{\bigg|  \frac{1}{h^{N+1}} \iintdm{\bbR}{\bbR^N}{\psi \left( \frac{x-x'}{h}, \frac{y-y'}{h} \right) \\ 
    & \qquad \times \Big( |\went(x,y)| - \chi_{\substack{|x| < m - h \\ y \in \Gamma_{1,h}}}(x',y') \sgn(\went(x',y')) \went(x,y) \Big) }{y'}{x'} \bigg| \rho}{y}{x} \\
    &\leq \iintdmlim{|x| \leq m}{\Gamma_1}{ \frac{1}{h^{N+1}} \iintdmlim{\bbR}{\bbR^N}{\psi \left( \frac{x-x'}{h}, \frac{y-y'}{h} \right) \\
    &\qquad \qquad \times \Big| |\went(x,y)| -  \sgn(\went(x',y')) \went(x,y) \Big| }{y'}{x'}  \rho}{y}{x} \\
    &\qquad + \iintdmlim{|x| \leq m}{\Gamma_1}{ \frac{1}{h^{N+1}} \iintdm{\bbR}{\bbR^N}{\psi \left( \frac{x-x'}{h}, \frac{y-y'}{h} \right) \\
    &\qquad \qquad \times \big| \sgn(\went(x',y') \went(x,y) \big| \, \left| 1 - \chi_{\substack{|x| < m - h \\ y \in \Gamma_{1,h}}}(x',y') \right|}{y'}{x'}  \rho}{y}{x} \\
    &:= \mathrm{I} + \mathrm{II}\,,
\end{split}
\end{equation}
where in the second (in)equality the quantity 
\begin{equation}
    \frac{1}{h^{N+1}} \psi \left( \frac{x-x'}{h}, \frac{y-y'}{h} \right) \sgn(\went(x',y')) \went(x,y)
\end{equation}
was added and subtracted. We now estimate $\mathrm{I}$ and $\mathrm{II}$. Denoting $(x,y)$ by $\bx$ and $(x',y')$ by $\bx'$,
\begin{equation}
    \mathrm{II} \leq C |D| \, \Vnorm{\rho}_{L^\infty(\Gamma)} \iintdm{[-m,m] \times \Gamma_1}{[-m+h, m-h]^c \times \Gamma^c_{1,h}}{\frac{1}{h^{N+1}} \psi \left( \frac{\bx - \bx'}{h} \right)}{\bx'}{\bx}\,.
\end{equation}
Since $\supp \psi \subset \{ (x,y) \in \bbR \times \bbR^N \, \big| \, |(x,y)| < 1 \}$, we have that $|x'| \leq |x'-x| + |x| \leq h + m$ and $|y'| \leq |y'-y| + |y| \leq h + \sqrt{N}$ on the regions of integration. Therefore,
\begin{equation}\label{eq-PSystemSemiLinear-L1ModulusOfContinuityInTimeProof-Estimate3}
\begin{split}
    \mathrm{II} &\leq C \iintdm{m-h \leq |x'| \leq m+h}{\Gamma_1}{\left( \iintdm{|x| \leq m}{|y| \leq h + \sqrt{N}}{\frac{1}{h^{N+1}} \psi \left( \frac{x-x'}{h}, \frac{y-y'}{h} \right) }{y}{x} \right)}{y'}{x'} \\
    &\leq C  \iintdm{m-h \leq |x'| \leq m+h}{\Gamma_1}{1}{y'}{x'} = Ch\,.
\end{split}
\end{equation}
The constant $C$ depends on $D$, $\rho$, and $\Gamma_1$ but not $\veps$, $\nu$, $\tau$ or $h$. To estimate $\mathrm{I}$, note that for any function $\vphi : \bbR \to \bbR$
\begin{equation}
    \big| |\vphi(x)| - \sgn(\vphi(z)) \vphi(x) \big| = \big| |\vphi(x)| - |\vphi(z)| + [\vphi(z) - \vphi(x) ] \sgn(\vphi(z)) \big| \leq 2 |\vphi(x) - \vphi(z)|\,.
\end{equation}
Then by Theorem \ref{thm-PSystemSemiLinear-L1ModulusOfContinuityInSpace},
\begin{equation}\label{eq-PSystemSemiLinear-L1ModulusOfContinuityInTimeProof-Estimate4}
\begin{split}
    \mathrm{I} &\leq \iintdm{|x| \leq m}{\Gamma_1}{\frac{2}{h^{N+1}} \iintdm{\bbR}{\bbR^N}{ \psi \left( \frac{x-x'}{h}, \frac{y-y'}{h} \right) \big| \went(x,y) - \went(x',y') \big| }{y'}{x'} \,  \rho}{y}{x} \\
    &\leq \iintdm{|x| \leq m}{\Gamma_1}{2 \iintdm{\bbR}{\bbR^N}{ \psi (x'',y'') \big| \went(x,y) - \went(x+hx'',y+hy'') \big| }{y''}{x''} \,  \rho}{y}{x} \\
    &\leq 2 \iintdm{\bbR}{\bbR^N}{\psi(x'',y'') \left( \iintdm{|x| \leq m}{\Gamma_1}{ \big| \wen(x,t+\tau,y) - \wen(x+hx'',t+\tau,y+hy'') \big| \, \rho(y)}{y}{x}\right) }{y''}{x''}\\
    &\quad + 2 \iintdm{\bbR}{\bbR^N}{\psi(x'',y'') \left( \iintdm{|x| \leq m}{\Gamma_1}{ \big| \wen(x,t,y) - \wen(x+hx'',t,y+hy'') \big| \, \rho(y)}{y}{x}\right) }{y''}{x''}\\
    &\leq 4 \max_{\substack{x'' \in \{ -1, 1 \} \\ y'' \in \bbS^{N-1}}} \iintdm{|x| \leq m}{\Gamma_1}{|\Wnen(x,y)-\Wnen(x+hx'',y+hy'')| \, \rho(y)}{y}{x}\,.
\end{split}
\end{equation}
Combining \eqref{eq-PSystemSemiLinear-L1ModulusOfContinuityInTimeProof-Estimate2}, \eqref{eq-PSystemSemiLinear-L1ModulusOfContinuityInTimeProof-Estimate3} and \eqref{eq-PSystemSemiLinear-L1ModulusOfContinuityInTimeProof-Estimate4} we arrive at \eqref{eq-PSystemSemiLinear-L1ModulusOfContinuityInTime-Estimate}, but with only $\wen$ on the left-hand side. Repeat the proof for $\zen$ and the proof is complete.
\end{proof}

We are finally ready to present the proof of the main theorem in this section.

\begin{proof}[Proof of Theorem \ref{thm-PSystemSemiLinear-WellPosednessInL1ForDiagonalSystem}]
Let $\veps > 0$. First assume that $\Wne \in \big[ L^1_{\rho}(\bbR \times \Gamma) \big]^2 \cap \big[ L^{\infty}(\bbR \times \Gamma) \big]^2$. For $\nu \in (0,1)$, let $(\Wnen)$ be a sequence in $\big[ C^{\infty}_c(\bbR \times \Gamma) \big]^2$ satisfying $\Vnorm{\Wnen}_{\infty} \leq \Vnorm{\Wne}_{\infty} $ and converging to $\Wne$ in $\big[ L^1_{\rho}(\bbR \times \Gamma) \big]^2$. Therefore, there exists a nondecreasing function $\omega^{\veps} : [0,\infty) \to \bbR$ with $\lim\limits_{m \to 0} \omega^{\veps}(m) = 0$ such that
\begin{equation}\label{eq-PSystemSemiLinear-InitialDataModulusOfContinuityEstimate}
    \iintdm{\bbR}{\Gamma}{|\Wnen(x+h,y+k)-\Wnen(x,y)| \, \rho(y)}{y}{x} \leq \omega^{\veps}(|h| + |k|)
\end{equation}
for every $\nu > 0$. Now for each $\nu > 0$ define $\Wen$ as the unique global smooth solution of \eqref{eq-PSystemSemiLinear-ViscosityApproximation-RiemannianCoordinates} with initial data $\Wnen$. From Theorem \ref{thm-PSystemSemiLinear-L1ModulusOfContinuityInTime} and the estimate \eqref{eq-PSystemSemiLinear-InitialDataModulusOfContinuityEstimate} we have that the sequence $(\Wen)_{\nu}$ is equicontinuous in the space $C \left( [0,\infty); \big[ L^1_{\rho, loc}(\bbR \times \Gamma) \big]^2 \right)$. Further, Theorem \ref{thm-PSystemSemiLinear-L1ModulusOfContinuityInSpace} and the estimate \eqref{eq-PSystemSemiLinear-InitialDataModulusOfContinuityEstimate} reveals that for each fixed $t > 0$, the sequence $(\Wen(\cdot,t,\cdot))_{\nu}$ is compact in $\big[ L^1_{\rho,loc}(\bbR \times \Gamma) \big]^2$. Therefore, the Frechet-Kolmogorov Theorem implies that for any subsequence $\nu_j$ there exists a further subsequence (still denoted $\nu_j$) and a function $\We \in C \left( [0,\infty); \big[ L^1_{\rho, loc}(\bbR \times \Gamma) \big]^2 \right)$ such that $(W^{\veps,\nu_j}(\cdot,t,\cdot))_j$ converges in $\big[ L^1_{\rho,loc} (\bbR \times \Gamma) \big]^2$ to $\We(\cdot,t,\cdot)$, uniformly in $t$ on compact subsets of $[0,\infty)$.
By Theorem \ref{thm-PSystemSemiLinear-AlmostEverywhereConvergenceToEntropySolution} the function $\We$ is an entropy solution to \eqref{eq-PSystemSemiLinear-RiemannianCoordinates} and satisfies the same $\veps$-independent $L^{\infty}$ bounds \eqref{eq-PSystemSemiLinear-LInftyBoundsOnViscosityApproximation} as the $\Wen$. Since the entropy solution to \eqref{eq-PSystemSemiLinear-RiemannianCoordinates} is unique (Corollary \ref{cor-PSystemSemiLinear-EntropySolutionsOfRiemannianCoordinatesAreUnique}), the entire sequence $(\Wen)_{\nu}$ converges to $\We$ in $C \left( [0,\infty); \big[ L^1_{\rho,loc}(\bbR \times \Gamma) \big]^2 \right) $.

Now, suppose that $\Wne \in \big[ L^{\infty}(\bbR \times \Gamma) \big]^2$. For $j > 0$ define $\chi_j  : \bbR^2 \to \bbR$ to be the characteristic function on the set $[-j,j] \times \Gamma$. Define $W^{\veps,j}$ to be the unique entropy solution of \eqref{eq-PSystemSemiLinear-RiemannianCoordinates} corresponding to the initial data $\chi_j \Wne \in \big[ L^1_{\rho}(\bbR \times \Gamma) \big]^2 \cap \big[ L^{\infty}(\bbR \times \Gamma) \big]^2$. As $j \to \infty$, $\chi_j \Wne \to \Wne$ in $\big[ L^1_{\rho^2, loc} (\bbR \times \Gamma) \big]^2 \cap \big[ L^1_{\rho, loc} (\bbR \times \Gamma) \big]^2$.
On account of \eqref{eq-PSystemSemiLinear-L1LocStabilityEstimate} the sequence $W^{\veps,j}$ converges in $\big[ L^1_{\rho^2,loc} (\bbR \times (0,\infty) \times \Gamma) \big]^2$ to some function $\We$.
Clearly $\We$ satisfies the entropy inequality \eqref{eq-PSystemSemiLinear-DiagonalSystemEntropyCondition1}. $\We$ also satisfies \eqref{eq-PSystemSemiLinear-DiagonalSystemEntropyCondition2}; write
\begin{equation}
\begin{split}
\frac{1}{T} \iiintdmtlim{0}{T}{V}{\Gamma}{|\We(x,t,y) - \Wne(x,y)| \rho(y)}{y}{x}{t} &\leq \frac{1}{T} \iiintdmtlim{0}{T}{V}{\Gamma}{|\We(x,t,y) - W^{\veps,j}(x,t,y)| \rho(y)}{y}{x}{t} \\
&\qquad + \frac{1}{T} \iiintdmtlim{0}{T}{V}{\Gamma}{|W^{\veps, j}(x,t,y) - \chi_j \Wne(x,y)| \rho(y)}{y}{x}{t} \\
&\qquad + \frac{1}{T} \iiintdmtlim{0}{T}{V}{\Gamma}{|\chi_j \Wne(x,y) - \Wne(x,y)| \rho(y)}{y}{x}{t} \\
&= \mathrm{I} + \mathrm{II} + \mathrm{III}\,,
\end{split}
\end{equation}
for any $V \Subset \bbR$. By the estimate \eqref{eq-PSystemSemiLinear-L1LocStabilityEstimate} we see that $W^{\veps,j}=\We$ on any compact subset of $\bbR \times [0,\infty) \times \Gamma$ for $j$ sufficiently large. Thus, $\mathrm{I} = 0$ and $\mathrm{III}=0$ for $j$ large enough. By definition of $W^{\veps,j}$ as an entropy solution, $\mathrm{II}$ converges to zero as $T \to 0^+$. Thus, $\We$ satisfies \eqref{eq-PSystemSemiLinear-DiagonalSystemEntropyCondition2}, and is therefore a weak entropy solution of \eqref{eq-PSystemSemiLinear-RiemannianCoordinates} with initial data $\Wne$. In addition, $\We$ satisfies
\begin{equation}\label{eq-LInftyBoundOnWVeps}
    \Vnorm{\We}_{\infty} \leq \beta + \max \left\lbrace \left| f \left( \frac{\beta}{a} \right) \right| \, , \, \left| f \left( - \frac{\beta}{a} \right) \right| \right\rbrace\,,
\end{equation}
where $\beta := \sup_{\veps} \Vnorm{\Wne}_{\infty}$.
By Corollary \ref{cor-PSystemSemiLinear-EntropySolutionsOfRiemannianCoordinatesAreUnique} this entropy solution $\We$ is unique.
Since $W^{\veps,j} = \We$ on compact subsets of $\bbR \times [0,\infty) \times \Gamma$ for $j$ large enough, and since $W^{\veps,j} \in C \left( [0,\infty); \big[ L^1_{\rho,loc}(\bbR \times \Gamma) \big]^2 \right)$ it follows that $\We \in C \left( [0,\infty); \big[ L^1_{\rho, loc}(\bbR \times \Gamma) \big]^2 \right)$. The proof is complete.
\end{proof}
\section{Proof of Step 1 in Theorem \ref{thm-Theorem4.1}}\label{apdx-Step1}
\begin{proof}
We follow \cite[Theorem~5.2.1]{Dafermos} with appropriate modifications. We use the relaxation term to prove $L^2$ convergence of the solution to equilibrium rather than proving the $L^2$ stability of weak solutions. Our choices of test function are also modified to account for the lack of initial conditions in the definition of the entropy inequality. 
Use one of the strongly convex entropy pairs constructed in Lemma \ref{lma-EntropyExistence}, denoted here by $(\eta,Q)$. For $U \in B_{\gamma}$, define the relative entropy pair 
\begin{equation}\label{eq-Step1Proof-1}
\begin{split}
    h(U,\Ubar) &:= \eta(U)-\eta(\Ubar) - \dU \eta(\Ubar) (U-\Ubar)\,,\\
    Y(U,\Ubar) &:= Q(U) - Q(\Ubar) - \dU \eta(\Ubar)(F(U)-F(\Ubar))\,.
\end{split}
\end{equation}
We claim that there exists an $\alpha > 0$ such that for every $U \in B_{\gamma}$, $\Ubar \in K$,
\begin{equation}\label{eq-Theorem4.1Proof-Step1-EntropyFluxBound}
    |Y(U,\Ubar)| \leq \alpha h(U,\Ubar)\,.
\end{equation}
Indeed, both $h$ and $Y$ are quadratic in $U - \Ubar$, and since $\dUU \eta(U)$ is uniformly positive definite on $B_{\gamma}$ we have that the quotient $|Y| / h$ is bounded from above by a positive constant $\alpha$.

Now, let $T>0$, and let $\vphi \in C^1_c(\bbR \times (0,T) \times \Gamma)$ be a nonnegative test function. Then since $\Ue$ is an entropy solution of \eqref{eq-FullSystemRandom} and $\Ubar \in K$ is a classical solution of \eqref{eq-FullSystemRandom},
\begin{equation}\label{eq-Step1Proof-2}
\begin{split}
    &0 \leq \iiintdmt{0}{T}{\bbR}{\Gamma}{\left( \dt \vphi \, \eta(\Ue) + \dx \vphi \, Q(\Ue) - \frac{1}{\veps} \vphi \, \dv \eta(\Ue)r(\Ue) \right) \rho}{y}{x}{t} \\
    &= \iiintdmt{0}{T}{\bbR}{\Gamma}{\left( \dt \vphi \big(  \eta(\Ue) - \eta(\Ubar) \big) + \dx \vphi \big ( Q(\Ue) - Q(\Ubar) \big) - \frac{1}{\veps} \vphi \, \dv \eta(\Ue)r(\Ue) \right) \rho}{y}{x}{t}\,.
\end{split}
\end{equation}
Subtracting the linear terms in the definitions of $h$ and $Y$ (see \eqref{eq-Step1Proof-1}) from both sides of \eqref{eq-Step1Proof-2} gives
\begin{equation}\label{eq-Step1Proof-3}
\begin{split}
    \iiintdmt{0}{T}{\bbR}{\Gamma}{&\left( \dt \vphi \, h(\Ue,\Ubar) + \dx \vphi \, Y(\Ue,\Ubar) - \frac{1}{\veps} \vphi \, \dv \eta(\Ue)r(\Ue) \right) \rho}{y}{x}{t} \\
    &\geq - \iiintdmt{0}{T}{\bbR}{\Gamma}{\left( \dt \vphi \, \dU \eta(\Ubar) (\Ue-\Ubar) + \dx \vphi \, \dU \eta(\Ubar)(F(\Ue)-F(\Ubar))  \right) \rho}{y}{x}{t}\,.
\end{split}
\end{equation}
Now, since $\dv \eta(\Ubar) = 0$ by \eqref{eq-EntropyRestrictionToK}, we can subtract $\frac{1}{\veps} \vphi \, \dv \eta(\Ubar) (r(\Ue) - r(\Ubar)) \rho$ from the right hand side of \eqref{eq-Step1Proof-3} and obtain
\begin{equation}\label{eq-Step1Proof-4}
\begin{split}
    \iiintdmt{0}{T}{\bbR}{\Gamma}{&\left( \dt \vphi \, h(\Ue,\Ubar) + \dx \vphi \, Y(\Ue,\Ubar) - \frac{1}{\veps} \vphi \, \dv \eta(\Ue)r(\Ue) \right) \rho}{y}{x}{t} \\
    &\geq - \iiintdmt{0}{T}{\bbR}{\Gamma}{\left( \dt \vphi \, \dU \eta(\Ubar) \, \Ue + \dx \vphi \, \dU \eta(\Ubar) \, F(\Ue) - \frac{1}{\veps} \vphi \, \dv \eta(\Ubar) \, r(\Ue) \right) \rho}{y}{x}{t} \\
    & \quad + \iiintdmt{0}{T}{\bbR}{\Gamma}{\left( \dt \vphi \, \dU \eta(\Ubar) \, \Ubar + \dx \vphi \, \dU \eta(\Ubar) \, F(\Ubar) - \frac{1}{\veps} \vphi \, \dv \eta(\Ubar) r(\Ubar) \right) \rho}{y}{x}{t}\,.
\end{split}
\end{equation}
Note that $\vphi \, \dU \eta(\Ubar) \in C^1_c(\bbR \times (0,T) \times \Gamma)$ is a valid test function. Using that $\Ue$, $\Ubar$ are weak (resp. classical) solutions of \eqref{eq-FullSystemRandom}, we get that both terms on the right-hand side of \eqref{eq-Step1Proof-5} are equal to $0$. Therefore,
\begin{equation}\label{eq-Theorem4.1Proof-Step1-InequalityBeforeTestFunctionIsChosen}
\begin{split}
    \iiintdmt{0}{T}{\bbR}{\Gamma}{&\left( \dt \vphi \, h(\Ue,\Ubar) + \dx \vphi \, Y(\Ue,\Ubar) - \frac{1}{\veps} \vphi \, \dv \eta(\Ue)r(\Ue) \right) \rho}{y}{x}{t} \geq 0\,.
\end{split}
\end{equation}
Now we will choose an appropriate test function $\vphi(x,t,y)$ and use Lebesgue differentiation to obtain the appropriate estimate. Fix $t_1$ and $t_2$ such that $0 < t_1 < t_2 < T$. Let $m > 0$. Let $\sigma_1 \in (t_1, t_2)$, $\sigma_2 \in (t_1, t_2)$ such that $\sigma_1 < \sigma_2$. Let $\delta> 0$ be a constant sufficiently small so that $\sigma_1 + \delta < \sigma_2$. 
Set $\vphi(x,t,y) := \chi(x,t) \, \psi(t)$, where
\begin{equation}
\psi(t) :=
\begin{cases}
    0\,, & 0 \leq t < \sigma_1\,, \\
    \frac{1}{\delta}(t - \sigma_1)\,, & \sigma_1 \leq t < \sigma_1 + \delta\,, \\
    1\,, & \sigma_1 + \delta \leq t < \sigma_2\,, \\
    \frac{1}{\delta} (\sigma_2 - t) + 1\,, & \sigma_2 \leq t < \sigma_2 + \delta\,, \\
    0\,, & \sigma_2 + \delta \leq t < \infty\,,
\end{cases}
\end{equation}
and
\begin{equation}
\chi(x,t) :=
\begin{cases}
    1\,, & |x|-m-\alpha(t_2 - t) < 0\,, \\
    \frac{1}{\delta}(m+\alpha(t_2-t) - |x|) + 1\,, & 0 \leq |x|-m-\alpha(t_2 - t) < \delta\,, \\
    0\,, & |x|-m-\alpha(t_2 - t) \geq \delta\,.
\end{cases}
\end{equation}
With this choice of $\vphi$, \eqref{eq-Theorem4.1Proof-Step1-InequalityBeforeTestFunctionIsChosen} becomes
\begin{equation}\label{eq-Step1Proof-5}
\begin{split}
    & \frac{1}{\delta} \int\limits_{\sigma_1}^{\sigma_1+\delta} \int\limits_{|x| \leq m +\alpha(t_2-t)} \int\limits_{\Gamma}{ h(\Ue,\Ubar) \, \rho}\, \mathrm{d}y \, \mathrm{d}x \, \mathrm{d}t - \frac{1}{\delta} \int\limits_{\sigma_2}^{\sigma_2+\delta} \int\limits_{|x| \leq m +\alpha(t_2-t)} \int\limits_{\Gamma}{h(\Ue,\Ubar) \, \rho}\, \mathrm{d}y \, \mathrm{d}x \, \mathrm{d}t \\
    & \qquad - \frac{1}{\veps} \int\limits_{\sigma_1+\delta}^{\sigma_2} \int\limits_{|x| \leq m + \alpha(t_2-t)} \int\limits_{\Gamma}{\dv \eta(\Ue) r(\Ue) \rho}\, \mathrm{d}y \, \mathrm{d}x \, \mathrm{d}t + O(\delta) \\
    &\qquad \qquad \geq \frac{1}{\delta} \int\limits_{\sigma_1+\delta}^{\sigma_2} \int\limits_{0 \leq |x| - m - \alpha(t_2-t) \leq \delta} \int\limits_{\Gamma}{\left( \alpha h(\Ue,\Ubar) + Y(\Ue,\Ubar) \frac{x}{|x|} \right) \rho}\, \mathrm{d}y \, \mathrm{d}x \, \mathrm{d}t\,.
\end{split}
\end{equation}
By \eqref{eq-Theorem4.1Proof-Step1-EntropyFluxBound} the term on the right-hand side of \eqref{eq-Step1Proof-5} is nonnegative. Thus we can drop said term, and taking $\delta \to 0$ we get
\begin{equation}
\begin{split}
    &\int\limits_{|x| \leq m +\alpha(t_2-\sigma_1)}{\int\limits_{\Gamma}{ h(\Ue(x,\sigma_1,y),\Ubar) \, \rho}\, \mathrm{d}y }\, \mathrm{d}x - \int\limits_{|x| \leq m +\alpha(t_2-\sigma_2)}{\int\limits_{\Gamma}{h(\Ue(x,\sigma_2,y),\Ubar) \, \rho}\, \mathrm{d}y}\, \mathrm{d}x \\
    & \qquad - \frac{1}{\veps} \int\limits_{\sigma_1}^{\sigma_2} \int\limits_{|x| \leq m + \alpha(t_2-t)} \int\limits_{\Gamma}{\dv \eta(\Ue) r(\Ue) \rho}\, \mathrm{d}y \, \mathrm{d}x \, \mathrm{d}t \geq 0\,,
\end{split}
\end{equation}
for almost every $\sigma_1, \sigma_2 \in (t_1,t_2)$ with $\sigma_1 < \sigma_2$.
Using the relative entropy expansion for $h$, the fact that $\dUU \eta(U)$ is positive definite and that $\dv \eta(\ue,e(\ue)) = 0 = r(\ue,e(\ue))$, we have
\begin{equation}\label{eq-Step1Proof-6}
\begin{split}
    \int\limits_{|x| \leq m +\alpha(t_2-\sigma_2)}& \int\limits_{\Gamma} |\Ue(x,\sigma_2,y)-\Ubar|^2 \, \rho \, \mathrm{d}y\, \mathrm{d}x \\
    & + \frac{1}{\veps} \int\limits_{\sigma_1}^{\sigma_2} \int\limits_{|x| \leq m + \alpha(t_2-t)}\int\limits_{\Gamma}{ \Big( \dv \eta(\Ue) - \dv \eta(\ue,e(\ue)) \Big) \Big( r(\Ue) - r(\ue,e(\ue)) \Big) \rho}\, \mathrm{d}y \, \mathrm{d}x \, \mathrm{d}t \\
    & \qquad \qquad \leq C \int\limits_{|x|\leq  m +\alpha(t_2-\sigma_1)}{\int\limits_{\Gamma}{ |\Ue(x,\sigma_1,y)-\Ubar|^2 \, \rho}\, \mathrm{d}y}\, \mathrm{d}x \,,
\end{split}
\end{equation}
for almost every $\sigma_1, \sigma_2 \in (t_1,t_2)$ with $\sigma_1 < \sigma_2$.

We can drop the first term on the left-hand side of \eqref{eq-Step1Proof-6} since it is positive for every $\sigma_2$. We now use that $\dvv \eta(U) > c$, that $\dv r(\ue,e(\ue)) > 0$, and that $\dvv r(U) \equiv 0$ on $\overline{B_{\gamma}}$ along with Taylor's theorem to get the inequality
\begin{equation}\label{eq-Theorem4.1Proof-Step1-EstimateBeforeLSC2}
\begin{split}
    & \frac{1}{\veps} \int\limits_{\sigma_1}^{\sigma_2} \int\limits_{|x| \leq m + \alpha(t_2-t)}\int\limits_{\Gamma} { |\ve - e(\ue)|^2 \rho}\, \mathrm{d}y \, \mathrm{d}x \, \mathrm{d}t \leq C \int\limits_{|x|\leq m +\alpha(t_2-\sigma_1)} \int\limits_{\Gamma} |\Ue(x,\sigma_1,y)-\Ubar|^2 \, \rho \, \mathrm{d}y \, \mathrm{d}x \,,
\end{split}
\end{equation}
for almost every $\sigma_1, \sigma_2 \in (t_1,t_2)$ with $\sigma_1 < \sigma_2$. Now, entropy solutions $\Ue(x,t,y)$ satisfy
$$
\Ue \in C([0,T) \setminus \cF ;(L^2_{loc}(\bbR \times \Gamma))^2)\,, \qquad \cF \text{ at most countable}\,,
$$
by a straightforward adaptation of \cite[Theorem 4.5.1]{Dafermos}. Therefore,
\eqref{eq-Theorem4.1Proof-Step1-EstimateBeforeLSC2} holds for every $\sigma_1 \in [t_1,t_2]$ and for every $\sigma_2 \in (\sigma_1, t_2]$, and for almost every $t_1$, $t_2 \in (0,T)$.
Thus,
\begin{equation}
\begin{split}
    \frac{1}{\veps}\int\limits_{0}^{T} \int\limits_{|x| \leq m} \int\limits_{\Gamma}{ |\ve - e(\ue)|^2 \rho}\, \mathrm{d}y \, \mathrm{d}x \, \mathrm{d}\tau \leq C \int\limits_{|x|\leq m +\alpha T}{\int\limits_{\Gamma}{ |\Ue(x,t,y)-\Ubar|^2 \, \rho}\, \mathrm{d}y}\, \mathrm{d}x \,.
\end{split}
\end{equation}
Letting $T' > 0$ and integrating in $t$ from $0$ to $T'$, we obtain
\begin{equation}
\begin{split}
    \int\limits_{0}^{T'}{\frac{1}{T'}}\, \mathrm{d}t \, \frac{1}{\veps}\int\limits_{0}^{T} \int\limits_{|x| \leq m} \int\limits_{\Gamma}{ |\ve - e(\ue)|^2 \rho}\, \mathrm{d}y \, \mathrm{d}x \, \mathrm{d}\tau \leq \frac{C}{T'} \int\limits_{0}^{T'}{\int\limits_{|x|\leq m +\alpha T}{\int\limits_{\Gamma}{ |\Ue(x,t,y)-\Ubar|^2 \, \rho}\, \mathrm{d}y}\, \mathrm{d}x}\, \mathrm{d}t \,.
\end{split}
\end{equation}
Now use the triangle inequality;
\begin{equation}\label{eq-Step1Proof-15}
\begin{split}
    \frac{1}{\veps}\int\limits_{0}^{T} \int\limits_{|x| \leq m} \int\limits_{\Gamma}{ |\ve - e(\ue)|^2 \rho}\, \mathrm{d}y \, \mathrm{d}x \, \mathrm{d}\tau &\leq \frac{C}{T'} \intdmt{0}{T'}{\int\limits_{|x|\leq m +\alpha T}{\int\limits_{\Gamma}{ |\Ue(x,t,y)-\Une(x,y)|^2 \, \rho}\, \mathrm{d}y}\, \mathrm{d}x}\, \mathrm{d}t \\
    &\qquad + C \int\limits_{|x| \leq m+\alpha T}\int\limits_{\Gamma}{|\Une-\Ubar|^2 \rho}\, \mathrm{d}y \, \mathrm{d}x\,.
\end{split}
\end{equation}
Since the $\Ue$ are entropy solutions, the first term on the right hand side of \eqref{eq-Step1Proof-15} converges to $0$ as $T' \to 0$. Thus,
\begin{equation}
    \frac{1}{\veps} \int\limits_{0}^{T} \int\limits_{|x| \leq m} \int\limits_{\Gamma}{|\ve - e(\ue)|^2 \rho}\, \mathrm{d}y \, \mathrm{d}x \, \mathrm{d}t \leq C \int\limits_{|x| \leq m + \alpha T} \int\limits_{\Gamma}{|\Une - \Ubar|^2 \rho}\, \mathrm{d}y \, \mathrm{d}x\,.
\end{equation}
Letting $m \to \infty$ and then $T \to \infty$ we obtain \eqref{eq-Theorem4.1Proof-L2Estimate}.
\end{proof}

\section{Proof of Theorem \ref{thm-Appendix-UniquenessOfEntropySolution}}\label{apdx-Uniqueness}
\begin{proof}
We follow the argument of \cite{Kruzkov} as outlined in \cite{Evans}.
Choose a sequence of $C^2$ convex functions $\ell_n$ such that
\begin{equation}
\ell_n(\theta) \to |\theta| \text{ uniformly on compact subsets } \quad \ell'_n(\theta) \to \sgn(\theta) \text{ almost everywhere\,. }
\end{equation}
Then for any $k$, $\tilde{k}$ in $\bbR$, by the Lebesgue Dominated Convergence Theorem we have
\begin{equation}
\iiintdmt{0}{\infty}{\bbR}{\Gamma}{\left( |u-k|\dt \vphi + \sgn(u-k)(f(u)-f(k)) \dx \vphi \right) \rho}{y}{x}{t} \geq 0, 
\end{equation}
\begin{equation}
\iiintdmt{0}{\infty}{\bbR}{\Gamma}{\left( |\tilde{u}-\tilde{k}|\dt \tilde{\vphi} + \sgn(\tilde{u}-\tilde{k})(f(\tilde{u})-f(\tilde{k})) \dx \tilde{\vphi} \right) \rho}{y}{x}{t} \geq 0, 
\end{equation}
for every $\vphi$, $\tilde{\vphi}$ in $C^1_c(\bbR \times (0,\infty) \times \Gamma)$, $\vphi, \tilde{\vphi} \geq 0$.

Define the measure $\mu$ on $\bbR^N$ by
$$
\mathrm{d} \mu := \chi_{\Gamma}(y) \, \rho(y) \mathrm{d}y\,.
$$

Now let $\Phi \in C^1_c(\bbR \times \bbR \times (0,\infty) \times (0,\infty) \times \Gamma \times \Gamma)$, $\Phi \geq 0$, $\Phi = \Phi(x,\tilde{x},t,\tilde{t},y,\tilde{y})$. For each $(\tilde{x},\tilde{t},\tilde{y})$, let $k = \tilde{u}(\tilde{x},\tilde{t},\tilde{y})$ and let $\vphi(x,t,y) = \Phi(x,\tilde{x},t,\tilde{t},y,\tilde{y})$. Then, integrating with respect to the measure $\mathrm{d}\tilde{x} \, \mathrm{d} \tilde{t} \, \mathrm{d}\mu(\tilde{y})$,
\begin{equation}
\begin{split}
    &\intdmt{0}{\infty}{\intdmt{0}{\infty}{\iiintdm{\bbR}{\bbR}{\bbR^N}{\intdm{\bbR^N}{\left( |u(x,t,y)-\tilde{u}(\tilde{x},\tilde{t},\tilde{y})| \dt \Phi \right) \\
    &\qquad + \left( \sgn(u(x,t,y)-\tilde{u}(\tilde{x},\tilde{t},\tilde{y})) \Big( f(u(x,t,y))-f(\tilde{u}(\tilde{x},\tilde{t},\tilde{y})) \Big) \dx \Phi \right) }{\mu( \tilde{y})}}{\mu(y)}{\tilde{x}}{x}}{\tilde{t}}}{t} \geq 0\,.
\end{split}
\end{equation}
Similarly, we also have
\begin{equation}
\begin{split}
    &\intdmt{0}{\infty}{\intdmt{0}{\infty}{\iiintdm{\bbR}{\bbR}{\bbR^N}{\intdm{\bbR^N}{\left( |\tilde{u}(\tilde{x},\tilde{t},\tilde{y})-u(x,t,y)| \p_{\tilde{t}} \Phi \right) \\
    &\qquad + \left( \sgn(\tilde{u}(\tilde{x},\tilde{t},\tilde{y})-u(x,t,y))\Big( f(\tilde{u}(\tilde{x},\tilde{t},\tilde{y}))-f(u(x,t,y)) \Big) \p_{\tilde{x}} \Phi \right) }{\mu( \tilde{y})}}{\mu(y)}{\tilde{x}}{x}}{\tilde{t}}}{t} \geq 0\,.
\end{split}
\end{equation}
Adding the above two, we have
\begin{equation}\label{eq-EntropyUniquenessProof-AddingTwoInequalities}
\begin{split}
    &\intdmt{0}{\infty}{\intdmt{0}{\infty}{\iiintdm{\bbR}{\bbR}{\bbR^N}{\intdm{\bbR^N}{|u(x,t,y)-\tilde{u}(\tilde{x},\tilde{t},\tilde{y})| (\p_{\tilde{t}} \Phi + \dt \Phi) \\
    &\qquad + \sgn(u(x,t,y)-\tilde{u}(\tilde{x},\tilde{t},\tilde{y})) \big( f(u(x,t,y))-f(\tilde{u}(\tilde{x},\tilde{t},\tilde{y})) \big) (\p_{\tilde{x}} \Phi + \dx \Phi )}{\mu(\tilde{y})}}{\mu(y)}{\tilde{x}}{x}}{\tilde{t}}}{t} \geq 0\,.
\end{split}
\end{equation}
We make the following choice of $\Phi$: Let $\psi : \bbR \to \bbR$ be a standard mollifier, with $\psi_{\delta}(x) = \frac{1}{\delta} \psi \left( \frac{x}{\delta} \right)$. Define
$$
\Phi(x,\tilde{x},t,\tilde{t},y,\tilde{y}) = \psi_{\delta} \left( \frac{x-\tilde{x}}{2} \right) \psi_{\delta} \left( \frac{t-\tilde{t}}{2} \right) \psi_{\delta} \left( \frac{y-\tilde{y}}{2} \right)  w \left( \frac{x+\tilde{x}}{2} , \frac{t+\tilde{t}}{2} , \frac{y + \tilde{y}}{2} \right),
$$
where $w = w(z_1,z_2, z_3) \in C^1_c(\bbR \times (0,\infty) \times \Gamma)$, $w \geq 0$. Plugging this function into \eqref{eq-EntropyUniquenessProof-AddingTwoInequalities}, we have
\begin{equation}
\begin{split}
    \intdmtlim{0}{\infty}{\intdmtlim{0}{\infty}{\iiintdm{\bbR}{\bbR}{\bbR^N}{\intdm{\bbR^N}{&\Bigg\lbrace |u(x,t,y)-\tilde{u}(\tilde{x},\tilde{t},\tilde{y})| \p_{z_2} w \left( \frac{x+\tilde{x}}{2}, \frac{t+\tilde{t}}{2}, \frac{y+\tilde{y}}{2} \right)  \\
    &+ \sgn(u(x,t,y)-\tilde{u}(\tilde{x},\tilde{t},\tilde{y})) (f(u(x,t,y))-f(\tilde{u}(\tilde{x},\tilde{t},\tilde{y}))) \p_{z_1} w \left( \frac{x+\tilde{x}}{2}, \frac{t+\tilde{t}}{2}, \frac{y+\tilde{y}}{2} \right) \Bigg\rbrace \\
    & \psi_{\delta} \left( \frac{x-\tilde{x}}{2} \right) \psi_{\delta} \left( \frac{t-\tilde{t}}{2} \right) \psi_{\delta} \left( \frac{y-\tilde{y}}{2} \right) }{\mu(\tilde{y})}}{\mu(y)}{\tilde{x}}{x}}{\tilde{t}}}{t} \geq 0\,. \\
\end{split}
\end{equation}
By a change of variables
$$
x_1 = \frac{x+\tilde{x}}{2}, \quad x_2 = \frac{x-\tilde{x}}{2}, \quad t_1 = \frac{t+\tilde{t}}{2}, \quad t_2 = \frac{t-\tilde{t}}{2}, \quad y_1 = \frac{y+\tilde{y}}{2}, \quad y_2 =  \frac{y-\tilde{y}}{2}
$$
we have
\begin{equation}
    \iiintdmt{0}{\infty}{\bbR}{\bbR^N}{G(x_2,t_2,y_2) \psi_{\delta}(x_2) \psi_{\delta}(t_2) \psi_{\delta}(y_2) }{y_2}{x_2}{t_2} \geq 0\,,
\end{equation}
where
\begin{equation}
\begin{split}
    G(x_2,t_2,y_2) = \iiintdmtlim{0}{\infty}{\bbR}{\bbR^N}{\Bigg\lbrace &|u(x_1+x_2,t_1+t_2,y_1+y_2) - \tilde{u}(x_1-x_2,t_1-t_2,y_1-y_2)| \p_{z_2} w(x_1,t_1, y_1) \\
    & + \sgn(u(x_1+x_2,t_1+t_2,y_1+y_2)-\tilde{u}(x_1-x_2,t_1-t_2,y_1-y_2)) \\
    &\Big( f(u(x_1+x_2,t_1+t_2,y_1+y_2)) - f(\tilde{u}(x_1-x_2,t_1-t_2,y_1-y_2)) \Big) \p_{z_1} w(x_1,t_1,y_1) \Bigg\rbrace \\
    & \quad \times \chi_{\Gamma}(y_1+y_2) \, \chi_{\Gamma}(y_1-y_2) \, \rho(y_1+y_2) \rho(y_1-y_2)}{y_1}{x_1}{t_1}\,.
\end{split}    
\end{equation}

Now, if $x_2,t_2,y_2 \to 0$, then
\begin{align*}
u(x_1+x_2,t_1+t_2,y_1+y_2) &\to u(x_1,t_1,y_1)\,, \\
\tilde{u}(x_1-x_2,t_1-t_2,y_1-y_2) &\to \tilde{u}(x_1,t_1,y_1)\,, \\
\chi_{\Gamma}(y_1+y_2) \, \rho(y_1+y_2) &\to \chi_{\Gamma}(y_1) \rho(y_1)\,, \\
\chi_{\Gamma}(y_1-y_2) \, \rho(y_1-y_2) &\to \chi_{\Gamma}(y_1) \rho(y_1)\,.
\end{align*}
Since the mappings $(a,b) \mapsto |a-b|$, $(a,b) \mapsto \sgn (a-b) (f(a)-f(b))$ are Lipschitz, Lebesgue differentiation gives
\begin{equation}\label{eq-EntropyUniquenessProof-LebesgueDiffLimit}
\lim\limits_{\delta \to 0} \iiintdmt{0}{\infty}{\bbR}{\bbR^N}{G(x_2,t_2,y_2) \psi_{\delta}(x_2) \psi_{\delta}(t_2) \psi_{\delta}(y_2) }{y_2}{x_2}{t_2} = G(0,0,0) \geq 0\,.
\end{equation}
Writing $G(0,0,0)$ explicitly, \eqref{eq-EntropyUniquenessProof-LebesgueDiffLimit} becomes
\begin{equation}
\begin{split}
    \iiintdmt{0}{\infty}{\bbR}{\Gamma}{\bigg\lbrace &|u(x_1,t_1,y_1) - \tilde{u}(x_1,t_1,y_1)| \p_{z_2} w(x_1,t_1,y_1) \\
    & + \sgn(u(x_1,t_1,y_1)-\tilde{u}(x_1,t_1,y_1)) (f(u(x_1,t_1,y_1)) - f(\tilde{u}(x_1,t_1,y_1))) \p_{z_1} w(x_1,t_1,y_1) \bigg\rbrace \\
    & \rho^2(y_1)}{y_1}{x_1}{t_1} \geq 0\,.
\end{split}
\end{equation}
For convenience of notation set $x_1 \to x$, $t_1 \to t$, $y_1 \to y$:
\begin{equation}
   \iiintdmt{0}{\infty}{\bbR}{\Gamma}{\left( a(x,t,y) \p_{z_2}w(x,t,y) + b(x,t,y) \p_{z_1}w(x,t,y) \right) \rho^2(y)}{y}{x}{t} \geq 0,
\end{equation}
with
\begin{equation*}
\begin{cases}
a(x,t,y) &:= |u(x,t,y)-\tilde{u}(x,t,y)| \\
b(x,t,y) &:= \sgn (u(x,t,y)-\tilde{u}(x,t,y))(f(u(x,t,y))-f(\tilde{u}(x,t,y))).
\end{cases}
\end{equation*}
Now, define $\alpha := \Vnorm{f'}_{L^{\infty}([-M,M])}$, the Lipschitz constant of $f$. Fix $t_1$, $t_2 \in (0,\infty)$ with $t_1 < t_2$ (no relation to the $t_1$ and $t_2$ above) and let $\sigma_1$, $\sigma_2 \in (t_1,t_2)$ with $\sigma_1 < \sigma_2$. Let $m > 0$ and let $\Gamma_1 \Subset \Gamma$. For $\delta > 0$ small, choose as a test function $w(x,t,y) = \psi(t) \, \Xi(x,t) \, \chi_{\Gamma_1}(y)$, where
\begin{equation}
\psi(t) :=
\begin{cases}
    0 & 0 \leq \tau < \sigma_1 \\
    \frac{1}{\delta}(t-\sigma_1) & \sigma_1 \leq t < \sigma_1 + \delta \\
    1 & \sigma_1 + \delta \leq t < \sigma_2 \\
    \frac{1}{\delta} (\sigma_2-t) + 1 & \sigma_2 \leq t < \sigma_2 + \delta \\
    0 & \sigma_2 + \delta \leq t < \infty
\end{cases}
\end{equation}
and
\begin{equation}
\Xi(x,t) :=
\begin{cases}
    1 & |x|-m-\alpha(t_2-t) < 0 \\
    \frac{1}{\delta}(m+\alpha(t_2-t) - |x|) + 1 & 0 \leq |x|-m-\alpha(t_2-t) < \delta \\
    0 & |x|-m-\alpha(t_2-t) \leq \delta\,.
\end{cases}
\end{equation}
Then we obtain
\begin{equation}
\begin{split}
    \frac{1}{\delta} \int\limits_{\sigma_1}^{\sigma_1+\delta}\int\limits_{|x| \leq m + \alpha(t_2-t)}\int\limits_{\Gamma_1}& a(x,t,y) \rho^2(y)\, \mathrm{d}y \, \mathrm{d}x \, \mathrm{d}t - \frac{1}{\delta} \int\limits_{\sigma_2}^{\sigma_2+\delta}{\int\limits_{|x| \leq m + \alpha(t_2-t)}{\int\limits_{\Gamma_1}{a(x,t,y) \rho^2(y)}\, \mathrm{d}y}\, \mathrm{d}x}\, \mathrm{d}t \\
    &- \frac{1}{\delta}\int\limits_{\sigma_1+\delta}^{\sigma_2}{\int\limits_{0 \leq |x| - m - \alpha(t_2-t) \leq \delta }{\int\limits_{\Gamma_1}{\left( \alpha a(x,t,y) - b(x,t,y) \frac{x}{|x|} \right) \rho^2(y)}\, \mathrm{d}y}\, \mathrm{d}x}\, \mathrm{d}t  \geq O(\delta)\,.
\end{split}
\end{equation}
By definition of $a$ and $b$ we have $|b(x,t,y)| \leq \alpha a(x,t,y)$, so the integrand in the last left-hand term is nonnegative. Therefore,
\begin{equation}
        \frac{1}{\delta} \int\limits_{\sigma_1}^{\sigma_1+\delta}{\int\limits_{|x| \leq m + \alpha(t_2-t)}{\int\limits_{\Gamma_1}{a(x,t,y) \rho^2(y)}\, \mathrm{d}y}\, \mathrm{d}x}\, \mathrm{d}t - \frac{1}{\delta} \int\limits_{\sigma_2}^{\sigma_2+\delta}{\int\limits_{|x| \leq m + \alpha(t_2-t)}{\int\limits_{\Gamma_1}{a(x,t,y) \rho^2(y)}\, \mathrm{d}y}\, \mathrm{d}x}\, \mathrm{d}t \geq O(\delta)\,,
\end{equation}
and as $\delta \to 0$ Lebesgue differentiation yields
\begin{equation}\label{eq-EntropyUniquenessProofForSCL}
    \int\limits_{|x| \leq m + \alpha(t_2-\sigma_2)}\int\limits_{\Gamma_1}{|u(x,\sigma_2,y)-\tilde{u}(x,\sigma_2,y)| \rho^2}\, \mathrm{d}y \, \mathrm{d}x \leq \int\limits_{|x| \leq m + \alpha(t_2-\sigma_1)}\int\limits_{\Gamma_1}{|u(x,\sigma_1,y)-\tilde{u}(x,\sigma_1,y)| \rho^2}\, \mathrm{d}y \, \mathrm{d}x
\end{equation}
for almost every $\sigma_1$, $\sigma_2 \in (t_1,t_2)$ with $\sigma_1 < \sigma_2$. Now, entropy solutions $u(x,t,y)$ satisfy
\begin{equation}
    u \in C([0,T) \setminus \cF \,; (L^1_{loc}(\bbR \times \Gamma)))\,, \quad \cF \text{ at most countable, }
\end{equation}
by a straightforward adaptation of the proof of \cite[Theorem 4.5.1]{Dafermos}. Therefore, the inequality \eqref{eq-EntropyUniquenessProofForSCL} holds for almost every $t_1$, $t_2$, $\sigma_1$, $\sigma_2$ satisfying $0 < t_1 \leq \sigma_1 < \sigma_2 \leq t_2 < \infty$. Thus, by appropriate choice of $t_1$ and $t_2$ outside of $\cF$,
\begin{equation}
    \int\limits_{|x| \leq m} \int\limits_{\Gamma_1}{|u(x,t_2,y)-\tilde{u}(x,t_2,y)| \rho^2(y)}\, \mathrm{d}y \, \mathrm{d}x \leq \int\limits_{|x| \leq m + \alpha t_2} \int\limits_{\Gamma_1}{|u(x,t_1,y)-\tilde{u}(x,t_1,y)| \rho^2(y)}\, \mathrm{d}y \, \mathrm{d}x\,.
\end{equation}
holds for almost every $t_1$, $t_2$ satisfying $0 < t_1 < t_2 < \infty$.

Finally, let $T > 0$. Integrate in $t_1$ from $0$ to $T$ and divide by $T$. We obtain
\begin{equation}
\begin{split}
    \frac{1}{T} \int\limits_{0}^{T} \, \mathrm{d}t_1 \int\limits_{|x| \leq m} \int\limits_{\Gamma_1} & |u(x,t_2,y)- \tilde{u}(x,t_2,y)| \rho^2(y) \, \mathrm{d}y \, \mathrm{d}x \\
    &\leq \frac{1}{T} \int\limits_{0}^{T} \int\limits_{|x| \leq m+ \alpha t_2} \int\limits_{\Gamma_1}{|u(x,t_1,y)-\tilde{u}(x,t_1,y)| \rho^2(y)}\, \mathrm{d}y \, \mathrm{d}x \, \mathrm{d}t_1\,.
\end{split}
\end{equation}
Now add and subtract the quantity
\begin{equation}
\int\limits_{|x|\leq m + \alpha t_2} \int\limits_{\Gamma_1}{|u_0(x,y)-\tilde{u}_0(x,y)|\rho^2(y)}\, \mathrm{d}y \, \mathrm{d}x
\end{equation}
on the right-hand side to get
\begin{equation}\label{eq-ReducedSystem-UniquenessOfEntropySolutions-ProofI}
\begin{split}
    \int\limits_{|x| \leq m} \int\limits_{\Gamma_1}& |u(x,t_2,y)- \tilde{u}(x,t_2,y)| \rho^2(y) \, \mathrm{d}y \, \mathrm{d}x \\
    &\leq \int\limits_{|x| \leq m+ \alpha t_2}\int\limits_{\Gamma_1}{|u_0(x,y)-\tilde{u}_0(x,y)| \rho^2(y)}\, \mathrm{d}y \, \mathrm{d}x + \mathrm{I}\,,
\end{split}
\end{equation}
where
\begin{equation}
\mathrm{I} = \frac{1}{T} \int\limits_{0}^{T} \int\limits_{|x| \leq m+ \alpha t_2} \int\limits_{\Gamma_1} {|u(x,t_1,y)-\tilde{u}(x,t_1,y)| \rho^2(y)}\, \mathrm{d}y \, \mathrm{d}x \, \mathrm{d}t_1 - \int\limits_{|x|\leq m + \alpha t_2} \int\limits_{\Gamma_1}{|u_0(x,y)-\tilde{u}_0(x,y)|\rho^2(y)}\, \mathrm{d}y \, \mathrm{d}x\,.
\end{equation}
Noting that $\rho \in L^{\infty}(\Gamma)$ and using the reverse triangle inequality we have
\begin{equation}
\begin{split}
    |\mathrm{I}| &\leq \frac{1}{T}\iiintdmt{0}{T}{|x| < m + \alpha t_2}{\Gamma_1}{\Big| |u(x,t_1,y)-\tilde{u}(x,t_1,y)| - |u_0(x,y) - \tilde{u}_0(x,y)| \Big| \rho^2(y)}{y}{x}{t_1} \\
    &\leq \frac{C_0}{T}\iiintdmt{0}{T}{|x| < m + \alpha t_2}{\Gamma_1}{\Big| \big( u(x,t_1,y)-u_0(x,y) \big) - \big( \tilde{u}(x,t_1,y) - \tilde{u}_0(x,y) \big) \Big| \rho(y)}{y}{x}{t_1} \\
    &\leq \frac{C_0}{T}\iiintdmt{0}{T}{|x| < m +\alpha t_2}{\Gamma_1}{|u(x,t_1,y)-u_0(x,y)| \rho(y)}{y}{x}{t_1} \\
    & \qquad + \frac{C_0}{T}\iiintdmt{0}{T}{|x| < m +\alpha t_2}{\Gamma_1}{|\tilde{u}(x,t_1,y)-\tilde{u}_0(x,y)| \rho(y)}{y}{x}{t_1}\,,
\end{split}
\end{equation}
where $C_0 = \Vnorm{\rho}_{L^{\infty}(\Gamma)}$. Therefore $\lim\limits_{T \to 0^+} \mathrm{I} =  0$ by the definition of $u$ and $\tilde{u}$ as entropy solutions. Taking $T \to 0$ on both sides of \eqref{eq-ReducedSystem-UniquenessOfEntropySolutions-ProofI} we arrive at \eqref{eq-ReducedSystem-L1locStabilityEstimate}.
\end{proof}

\bibliography{Relaxationbib}
\bibliographystyle{siam}

\end{document}